\documentclass[a4paper,11pt]{article}
\usepackage[T1]{fontenc}
\usepackage[utf8x]{inputenc} %extra unicode
\usepackage{lmodern}
\usepackage{amsmath}
\usepackage{amsthm}
\usepackage{amssymb}
\usepackage{authblk}
\usepackage{graphicx}
\usepackage{enumitem}
\usepackage{xcolor}
\usepackage{dsfont}
\usepackage{mathrsfs}
\usepackage{hyperref} \hypersetup{colorlinks=true, citecolor=blue, linkcolor=black}

\usepackage{amsthm}
\usepackage{amssymb}
\usepackage{xparse}
\usepackage{thmtools}
\usepackage{stackrel}
\usepackage{bbold}
\usepackage{relsize}
\usepackage{tikz}

\usetikzlibrary{hobby}
\newcommand{\R}{\mathbb{R}}
\newcommand{\C}{\mathbb{C}}
\newcommand{\Z}{\mathbb{Z}}
\newcommand{\N}{\mathbb{N}}

\newcommand{\dd}{\mathrm{d}}

\renewcommand{\O}{\mathcal{O}}

\renewcommand{\S}{\mathbb{S}}

\DeclareMathOperator{\supp}{supp}

\DeclareMathOperator{\tr}{tr}

\DeclareMathOperator{\dist}{dist}

\makeatletter
\newtheoremstyle{indented}
{7pt} %vertical space before
{7pt} % vertical space after
{} %{\addtolength{\@totalleftmargin}{2.5em}
{1.5em} %indent
{\bfseries} %header font
{.} %punctuation
{.5em} %horizontal space after header
{} %header specification

\theoremstyle{definition}

\newtheorem{defn}{Definition}[section]

\newtheorem{notation}{Notations}[section]

\theoremstyle{plain}
\newtheorem*{theorem*}{Theorem}

\newtheorem{theorem}{Theorem}

\newtheorem{prop}[defn]{Proposition}
\newtheorem{prop*}{Proposition}      
\newtheorem{corr}[defn]{Corollary}

\newtheorem{lem}[defn]{Lemma}
\newtheorem{conj}{Conjecture}
\theoremstyle{definition}
\newtheorem{rem}[defn]{Remark} %remarks are not indented

\numberwithin{equation}{section}
\makeatletter
\renewcommand*\env@matrix[1][*\c@MaxMatrixCols c]{%
  \hskip -\arraycolsep
  \let\@ifnextchar\new@ifnextchar
  \array{#1}}
\makeatother

\usepackage{dsfont}

\newcommand{\1}{\mathds{1}}
\newcommand{\X}{\mathbf{X}}
\newcommand{\x}{\mathbf{x}}
\newcommand{\Co}{\mathscr{C}}
\newcommand{\var}{\operatorname{var}}

\newcommand{\J}{\mathrm J}
\newcommand{\cst}{\mathrm c}
\newcommand{\Cst}{\mathrm C}
\usepackage{a4wide}
\title{Widom's conjecture: variance asymptotics and entropy bounds for
counting statistics of free fermions}
\author{Alix
Deleporte\thanks{alix.deleporte@universite-paris-saclay.fr},
Gaultier Lambert\thanks{glambert@kth.se}}
\affil{Universit\'e Paris-Saclay, CNRS, Laboratoire de math\'ematiques d'Orsay, 91405, Orsay, France.}
\affil{KTH Royal Institute of Technology, Department of Mathematics, 11428, Stockholm, Sweden.}

\newtheorem{assump}{Assumptions}

\begin{document} 

\maketitle

\begin{abstract}
{ We obtain a central limit theorem for bulk counting statistics
of free fermions in smooth domains of $\R^n$ with an explicit
description of the covariance structure.} This amounts to a study of
the asymptotics of norms of commutators between
spectral projectors of semiclassical Schrödinger operators and indicator functions { supported in the bulk}. In the spirit of the
Widom conjecture, we show that the squared Hilbert-Schmidt norm of
these commutators is of order $\hbar^{-n+1}\log(\hbar)$ as the
semiclassical parameter $\hbar$ tends to $
0$. We also give a new upper bound on the trace norm of these commutators
and applications to estimations of the entanglement entropy for free fermions.
\end{abstract}

\tableofcontents 

\section{Introduction}

The goal of this article is to explore the relationships between the
\emph{number variance} and \emph{entanglement entropy} of certain free
fermionic systems, on one hand, and the spectral theory of
semiclassical Schrödinger operators on the other hand.
Consider a Schrödinger operator on $L^2(\R^n)$:
\[
H_{\hbar} :=-\hbar^2\Delta+V , 
\]
where $\Delta$ is the standard (negative) Laplacian on $\R^n$, $\hbar>0$ plays the role of
the Planck constant, and $V:\R^n\to \R$ is bounded from below. 
The operator $H_{\hbar}$ is essentially self-adjoint with domain \[{\mathrm H}^2(\R^n)\cap \{u\in L^2,Vu\in
L^2\}.\]
We are interested in the fluctuations of the $N$ particles free fermionic state
associated with the operator $H_{\hbar}$ at zero temperature. Under assumptions on $V$ and $N$ which are specified below, this (random) point process is described by the probability density
$\mathbb{P}_N$ on $\R^{n\times N}$ with density
\begin{equation} \label{slater}
\mathbb{P}_N[\dd \x]=
\frac{1}{N!}\Big|\det_{N\times
N}\left[v_k(x_j)\right]\Big|^2, \qquad (x_1,\dots,x_N)  \in \R^{n\times N} , 
\end{equation}
where $(\lambda_k,v_k)_{1\leq k\leq N}$ are the $N$ lowest
eigenvalues, and associated orthonormal eigenfunctions, of
$H_{\hbar}$. Even though it is not emphasized, this spectral data and
the measure $\mathbb{P}_N$ depend on the semiclassical parameter~$\hbar$. We are interested in a joint limit where $\hbar\to 0$ and $N\to +\infty$,
while keeping fixed the \emph{Fermi energy} $\mu$. To this end, we fix
$\mu\in \R$ and choose
\[
N=N(\hbar) := \max\big\{ k\in\N : \lambda_k(\hbar) \le \mu \big\}.
\]
We work under the following assumptions on the external potential $V$:

\begin{assump} \label{asymp:V}
Fix $\mu \in\R$ and assume that $V \in L^1_{\rm loc}(\R^n)$ with $-\infty<\inf(V)<\mu$.
We assume that $\mathcal{D} : = \{V<\mu\}$ is relatively compact, $V$ is $C^\infty$ on a neighborhood of $\mathcal{D}$, and
$\partial_x V(x) \neq 0$ for all $x\in \partial \mathcal{D} =\{V=\mu\}$. 
\end{assump}

The probability measure \eqref{slater} gives rise to a determinantal
point process $\X := \sum_{j=1}^N \boldsymbol\delta_{x_j}$ on $\R^n$
associated with the integral kernel (also denoted $\Pi_{\hbar,\mu}$) of the spectral projection
\begin{equation} \label{Fermiproj}
\Pi_{\hbar,\mu}=\1_{(-\infty,\mu]}(H_{\hbar}).
\end{equation}

The set $\mathcal{D}$ is usually called the \emph{bulk} or
\emph{droplet}, and it is the region where the fermions concentrate.
This framework and our assumptions are similar to that of our previous
work \cite{deleporte_universality_2024} where we study the scaling
limits of the kernel $\Pi$ both in the bulk and at the boundary of the
droplet (see also \cite{dean_noninteracting_2016}) and give applications to the (real)
random variables $\X(\mathrm{f})=\sum_{j=1}^N\mathrm{f}(x_j)$ for
certain smooth functions $\mathrm{f}$.

The determinantal structure means that the correlation functions or
marginals of the measure
\eqref{slater} are of the form
\[
\rho_{N,k}(x_1,\cdots,x_k) = \det_{k\times k}\big[\Pi_{\hbar,\mu}(x_i,x_j) \big] , \qquad k\le N .
\]
Moreover, for any $\mathrm{f}\in L^{\infty}(\R^n)$, the random variable
$\X(\mathrm{f})=\sum_{j=1}^N\mathrm{f}(x_j)$ has the following Laplace transform:
\begin{equation}\label{eq:Laplace_determinantal}
\mathbb{E}[\exp \X(\mathrm{f}) ]=\det(1+(e^{\mathrm{f}}-1)\Pi_{\hbar,\mu})
\end{equation}
where the right-hand side is a Fredholm determinant on
$L^2(\R^n)$, since $ \Pi_{\hbar,\mu} $ is a finite-rank projection
with $\tr \Pi_{\hbar,\mu} = N(\hbar)$.

In \cite{deleporte_universality_2024}, we showed that with overwhelming probability, as
$\hbar\to0$, one has weak-$*$ convergence of probability measures 
\begin{equation} \label{dos}
N^{-1}\X \to \dd\varrho : = Z^{-1}\big(\mu-V\big)_+^{\frac{n}{2}} \dd x , 
\end{equation}
where $\dd x$ denotes the Lebesgue measure on $\R^n$, 
and
\begin{equation} \label{Z}
\begin{aligned}
Z&=\int(\mu-V)_+^{\frac n2}\dd x\\
N(\hbar) &\sim \frac{|B_n|Z}{(2\pi\hbar)^n} \quad\text{as
$\hbar\to0$,}
\end{aligned}
\end{equation}
where $|B_n|$ is the volume of the Euclidean ball $B_n = \{x\in\R^n : |x| \le 1 \}$. 
{ 
The convergence \eqref{dos} yields a \emph{law of large numbers} for counting statistics, in the sense that $\X(\mathrm{f})$ concentrates around its mean $N\int \mathrm{f} \dd\rho$ for $\mathrm{f}\in C^{\infty}_c(\R^n)$. 
}
In \cite{deleporte_central_2023,deleporte_universality_2024}, we also studied the fluctuations of
$\X(\mathrm{f})$ and we obtain a central limit theorem as
$\hbar \to 0$ with the variance being at most of order $\hbar^{-n+1}\asymp
N\hbar$.
Moreover, in dimension~1, \cite{Smith_21,deleporte_central_2023} the variance converges and its limit is described by a weighted $H^{1/2}$ Sobolev seminorm. 

\smallskip

This article focuses instead on the case where $\mathrm{f} =\1_\Omega$ is the
indicator function of an open, relatively compact
  subset $\Omega$ of the droplet $\mathcal{D}$ with a smooth
  boundary.
The \emph{counting statistics} $\X(\Omega)$ are directly relevant for physical applications; in particular the problem of measuring the \emph{entanglement entropy} of sub-systems that we shortly review.

Entanglement is a crucial property of quantum states which, for free
fermions, results from the Pauli exclusion principle, and is quantified by suitable concepts of entropy. 
For any open, relatively compact set $\Omega\Subset\R^n$, one can define its von Neumann entanglement entropy $\mathcal{S}_\Omega$ which measures the correlations between the state restricted to $\Omega$ and its complement. This quantity plays a crucial role in condensed matter physics and in quantum information, however it is difficult to estimate, both experimentally and theoretically.
Remarkably, at zero temperature, the entanglement entropy is not
extensive and is often of order $\hbar^{n-1}|\partial\Omega|$ where
$\hbar$ is the typical inter-particle distance. This is known as the \emph{area law} in the physics literature; see \cite{srednicki_entropy_1993,callan_geometric_1994} for an historic references.  
This property is expected for states with a finite correlation length (gapped systems), while for gapless systems, which exhibit long-range correlations, it is expected that the entropy is enhanced by an extra $\log(\hbar^{-1})$ factor. This has been demonstrated for critical quantum systems \cite{calabrese_entanglement_2004,korepin_universality_2004,dubail_conformal_2017}; for one-dimensional models, conformal field theory methods predict that the leading behavior of the entanglement entropy is universal and 
\begin{equation} \label{cft}
\mathcal{S}_\Omega \sim \tfrac{\cst}{3} \log(\hbar^{-1}) \qquad\text{where $\cst$ is the central charge of the corresponding CFT.}
\end{equation}
In dimension $n> 1$, physical arguments also predict that, for a
smooth domain $\Omega$, the entropy $\mathcal{S}_\Omega$  is of order $\hbar^{1-n} \log(\hbar^{-1})$. 
This problem has been investigated for gapless systems of free fermions, both numerically \cite{vicari_entanglement_2012} and  analytically \cite{gioev_entanglement_2006,calabrese_exact_2012,calabrese_random_2015}, providing many evidence for this \emph{enhanced area law}. 
Remarkably, these works predict a universal relationship between the
fluctuations of the counting statistic $\X(\Omega)$ and the
entanglement entropy, independent of the dimension:
\begin{equation} \label{Scorr}
\mathcal{S}_\Omega  \stackrel[N\to +\infty]{}{\sim} \tfrac{\pi^2}3{\rm var} \X(\Omega) .
\end{equation}
This relation provides an experimental mean to measure the
entanglement entropy of a fermionic system by estimating its quantum
fluctuations, which is easier to measure.

Apart from the case of constant coefficients
  differential operators, there is no rigorous bound in the literature
  on the
  entanglement entropy. The goal of this article is to rigorously explore these questions for the (Schr\"odinger) free Fermion model defined above by using techniques from semiclassical analysis. 
In particular, we obtain an equivalent for ${\rm var} \X(\Omega)$
which is consistent with Widom's conjecture (Theorem~\ref{thm:J2}) and
deduce a central limit theorem for $\X(\Omega)$
(Theorem~\ref{thm:clt}). We also obtain lower and upper bounds for the
entanglement entropy (Theorem~\ref{thm:ent}) which match
  up to $\log\log(\hbar^{-1})$.

\subsection{Statement of results}

We first compute an equivalent for the variance ${\rm
var}\X(\Omega)$ for an open set $\Omega \Subset \mathcal{D}$ with
smooth boundary, also known as \emph{number variance} in random matrix theory (Section~\ref{sec:rmt}).
By the determinantal structure, this variance corresponds to the Hilbert-Schmidt norm ($\J^2$-norm -- see \cite{simon_trace_2005} for some background on Schatten norms) of the commutator $[\Pi_{\hbar,\mu},\1_{\Omega}]$. 
\begin{theorem} \label{thm:J2}
Suppose that Assumptions \ref{asymp:V} hold. Let  $\Omega \Subset
\mathcal{D}$ be an open set with a smooth boundary and let $f\in C^\infty(\R^n)$. Then as $\hbar\to0$,
\[
\var \X(f|_\Omega)  = \tfrac12
\big\| [f|_\Omega, \Pi_{\hbar,\mu}] \big\|^2_{\J^2}  
=    (2\pi\hbar)^{1-n} \big( C_{\Omega,f} \log (\hbar^{-1})+ \O(1)\big)  
\]
where 
\begin{equation}\label{eq:variance}
C_{\Omega,f} =  \frac{\cst_{n-1}}{2\pi^2}  \int_{\partial\Omega} \big(\mu-V(\hat x)\big)_+^{\frac{n-1}2} f(\hat x)^2 \dd \hat x ,
\end{equation}
$\dd\hat x$ is the volume measure on the boundary $\partial\Omega$,
and $\cst_n =\frac{\pi^{\frac{n}{2}}}{\Gamma(\frac{n}{2}+1)}$ for
$n\ge 0$ $($in particular $\cst_n=|B_{n}|$ for $n\ge 1)$. 
\end{theorem}

We expect that these asymptotics also hold for arbitrary sets $\Omega \Subset \R^n$ with a  piecewise smooth boundary.

An important consequence of Theorem \ref{thm:J2} is a central limit
theorem for counting statistics.
\begin{theorem} \label{thm:clt}
Suppose that Assumptions \ref{asymp:V} hold. Let $k\in\N$,
$\Omega_1, \cdots , \Omega_k \Subset \mathcal{D}$ be a collection of
open sets with smooth boundaries. 
Let $C_{\Omega} = C_{\Omega,1}$ be as in \eqref{eq:variance}.
Assume that for all $1\leq i,j \leq k$,  if $i\neq j$, $|\partial \Omega_i
\cap \partial \Omega_j|=0$ for the $(n-1)$-Haussdorff measure. Then,
one has in distribution as $\hbar\to 0$,  
\[
\left(\frac{\X(\Omega_1) - \mathbb{E}[\X(\Omega_1)]}{\sqrt{(2\pi\hbar)^{1-n}\log( \hbar^{-1})}} , \cdots, \frac{\X(\Omega_k) -  \mathbb{E}[\X(\Omega_k)]}{\sqrt{(2\pi\hbar)^{1-n}\log(\hbar^{-1})}} \right) \Rightarrow \mathcal{N}_{0,\Sigma}
\]
where the limit covariance matrix is \[\Sigma = \operatorname{diag}\big(C_{\Omega_1},\cdots, C_{\Omega_k}\big).\]
\end{theorem}

Physical systems whose number variance are lower-order compared to the volume (or expected number of points) are called hyperuniform. 
This concept has been introduced in the theoretical
chemistry and statistical physics literature \cite{torquato_point_2008} where it provides a framework to classify quasicrystals and other disordered systems; we refer to the survey \cite{torquato_hyperuniform_2018} for examples of hyperuniform point processes and their properties. 
in the framework of \cite{torquato_point_2008,torquato_hyperuniform_2018}, Theorem~\ref{thm:J2} shows that the free fermion ground states are Class II hyperuniform point processes.

\smallskip

An important physical quantity related to $\X(\Omega)$ is the 
\emph{entanglement entropy} defined as follows. Let $s : \lambda
\in [0,1] \mapsto - \lambda \log(\lambda) -
(1-\lambda)\log(1-\lambda)$ -- this function is continuous with $s(0)=s(1)=0$, and
$s(\lambda)$ is the entropy of a Bernoulli random variable
${\bf B}_{\lambda}$ with parameter $\lambda$.

Given $\Omega \Subset \R^n$, we now define 
\begin{equation} \label{Ent}
\mathcal{S}_\Omega(\X) 
= \tr s( \Pi|_{\Omega})
=  \sum_{n\in\N} s(\sigma_{n}) , 
\end{equation}
where $\sigma_{n}$ are the eigenvalues  of operator $\Pi|_{\Omega} = \1_\Omega \Pi_{\hbar,\mu} \1_{\Omega}$. 
In particular, $0\le \Pi|_{\Omega} \le 1$ as operators so that $\sigma_{n}\in[0,1]$. 
Formula \eqref{Ent} comes from the fact that $\X(\Omega) \overset{\rm law}{=} \sum_{n\in\N} {\bf B}_{\sigma_{n}}$ where $({\bf B}_{\sigma_{n}})_{n\in\N}$ are independent Bernoulli random variables, so that there entropies sum up. 
This key observation is a consequence of Wick's Theorem for free
fermions, e.g. \cite{ST03a,ST03b} or \cite{hough_determinantal_2006} for a probabilistic interpretation. Moreover, \eqref{Ent} agrees with the physical definition of the entanglement entropy of the reduced state $\rho_{\Omega} = \tr_{\mathcal{H}(\Omega^c)}(\Pi_{\hbar,\mu})$
which comes from the decomposition of the fermionic Fock space $\mathcal{H}(\R^n) = \mathcal{H}(\Omega) \otimes \mathcal{H}(\Omega^c)$ where $\mathcal{H}(\Omega)=  \bigwedge L^2(\Omega)$. 
Indeed, using the canonical decomposition of the state $\Pi_{\hbar,\mu}$, one verifies that 
$\tr\big( \rho_{\Omega} \log \rho_{\Omega} \big) =  \sum_{n\in\N}
s(\sigma_n) $: see the introduction of \cite{gioev_szego_2006} or \cite[Section~7]{charles_entanglement_2018}  for a detailed computation.

A natural observation is that the entropy is bounded from below
by the variance, since $s(x)\geq x(1-x)$.
Hence, Theorem~\ref{thm:J2} yields a lower-bound for the entanglement entropy of $\Omega$.
We prove a complementary upper bound of the same order up to a
factor $\log\log(\hbar^{-1})$.

\begin{theorem} \label{thm:ent}Suppose that Assumptions \ref{asymp:V} hold.
Let $\Omega \Subset \mathcal{D}$ be an open set with a smooth boundary,
there exists constants $C_\Omega, c_{\Omega} >0$ so  that 
\[
c_{\Omega} \hbar^{1-n} \log(\hbar^{-1}) \le \mathcal{S}_\Omega\le  C_{\Omega} \hbar^{1-n} \log(\hbar^{-1}) \log\log (\hbar^{-1}) . 
\]
\end{theorem}
The upper estimate is proved by interpolation between Theorem
\ref{thm:J2} and the following estimate on the trace-norm ($\J^1$-norm) of a commutator.

\begin{theorem} \label{thm:J1}
Suppose that Assumptions \ref{asymp:V} hold. Let $\Omega \Subset
\mathcal{D}$ be an open set with a smooth boundary.
Then as $\hbar\to0$,  
\[
\big\| [\Pi_{\hbar,\mu}, \1_\Omega] \big\|_{\J^1} = \O\big( \hbar^{1-n}
\log^2 \hbar \big). 
\]
\end{theorem}

From the Widom conjecture (Conjecture \ref{conj:Widom} below), we expect that as $\hbar\to0$,
$\mathcal{S}_\Omega  \sim \tfrac{\pi^2}3 {\rm var} \X(\Omega)$ and 
$\big\| [\Pi_{\hbar,\mu}, \1_\Omega] \big\|_{\J^1} \sim  \pi {\rm var}
\X(\Omega)$ so that the magnitude of the entropy is
  expected to be captured by our lower bound.

The more trivial estimate \[\big\| [\Pi_{\hbar,\mu}, \1_\Omega]
\big\|_{\J^1} = \O\big( \hbar^{-n} \big)\] leads via the same
interpolation to the upper-bound $S_{\Omega}=\O(\hbar^{1-n}\log^2(\hbar))$  appearing in the physics literature \cite{gioev_entanglement_2006}.

Note that the trace-norm of a commutator is generally more technical
to estimate than its Hilbert-Schmidt norm
(Theorem~\ref{thm:J2}). This is the main reason why we
presume that the estimate
of Theorem~\ref{thm:J1} is not sharp. 
Theorem~\ref{thm:J1} should be compared, both in its result and its
methods, to the main result of  \cite{fournais_optimal_2020} that we
now recall.

\begin{prop*} \label{thm:FM}
Let $f \in C^{\infty}_c(\R^n)$ and suppose that $\int |t| |  \widehat{f}(t) | \dd t <\infty$. Under the Assumptions~\ref{asymp:V}, 
\begin{equation} \label{FM}
\big\| [\Pi_{\hbar,\mu}, f] \big\|_{\J^1}= \O\bigg(\hbar^{1-n} \int |t| |  \widehat{f}(t) | \dd t \bigg)
\end{equation}
where the implied constant depends only on $(V,\mu)$. 
\end{prop*}

\subsection{Organisational remarks}

In Subsection~\ref{sec:rmt}, we concisely review the connections between our results and the random matrix literature with an emphasize on the Szeg\H{o} asymptotics for determinants in the Fisher-Hartwig class.
Then, in Subsection~\ref{sec:Widom}, we report on a conjecture of Widom \cite{widom_class_1982} which relates  number variance and entanglement entropy of (free) fermionic systems to spectral function of  Schrödinger operators and we review the main progress on this conjecture.  

As in our previous work \cite{deleporte_central_2023,deleporte_universality_2024}, our methods to study the commutator $[\Pi_{\hbar,\mu},\1_{\Omega}]$ rely on the semiclassical machinery, in particular, on an approximation of the projection $\Pi_{\hbar,\mu}$ by a Fourier Integral operator. Under Assumptions~\ref{asymp:V}, this
approximation holds modulo an error of order $\hbar^{1-n}$ in trace-norm, which is
negligible in this context. 
In Section~\ref{sec:semiclassic}, we regroup the main notations and results from the literature on semiclassical Schrödinger operator $H_{\hbar}$ that we need to prove our main results.

Section~\ref{sec:hilb-schm-estim} gathers several additional Hilbert-Schmidt estimates for commutators between between regular or \emph{mild} spectral functions of  $H_{\hbar}$ and spatial functions at arbitrary scale. 

Section~\ref{sec:comm} is dedicated to the proofs of Theorem~\ref{thm:J2} and its consequence, Theorem~\ref{thm:clt}. 
The method relies on applying several stationary phase argument when the amplitude is discontinuous. The arguments are rather involved and therefore the proof is divided in several intermediate steps.  

In Section~\ref{sec:tr}, we rely on a  multi-scale argument (we
decompose both $\Pi_{\hbar,\mu}$ and $\1_{\Omega}$ in  dyadic pieces)
and the estimates from Section~\ref{sec:hilb-schm-estim} to prove
Theorem~\ref{thm:J1}.
Then, Theorem~\ref{thm:ent} follows by a simple Schatten-norm interpolation which is explained in Subsection~\ref{sec:ent}. 

Finally, in the Appendix~\ref{sec:stat-phase}, we gather different versions of the \emph{stationary phase method} that we need when the amplitude varies on possibly arbitrary small scales.

\subsection{Relation to random matrices and determinants in the Fisher-Hartwig class} \label{sec:rmt}
In dimension 1, for a non-trivial interval $[a,b] \subset \mathcal{D}$, Theorem~\ref{thm:J2} states that, with $N$ the particle number, 
\[
\var \X([a,b]) \sim  \tfrac{1}{\pi^2} \log N .
\]
In particular, this quantity is independent of the confining potential $V$ and it matches with the Gaussian Unitary Ensemble  (and sine process) number variance.
The GUE variance computation goes back to Dyson and Mehta \cite{mehta_random_2004} and the study of the fluctuations to 
\cite{costin_gaussian_1995}, see \cite{smith_full_2021} for additional precisions and \cite{marino_number_2016} for a study of large deviations. 
This universality is consistent with \eqref{cft} and not surprising as it comes from the microscopic fluctuations (given by the sine process) of the particles around $a,b$. 
Moreover, there is an exact correspondence between the model \eqref{slater} with $V : x\in\R \mapsto x^2$ and the Gaussian Unitary Ensemble (GUE) eigenvalues \cite{Soshnikov_01,calabrese_random_2015}.

\medskip

These number variance asymptotics are explained by the log-correlated structure of the eigenvalues of random matrices, an observation which originates from \cite{hughes_characteristic_2001,gustavsson_gaussian_2005,fyodorov_freezing_2012} and has been revisited in \cite{claeys_how_2021} from the perspective of Gaussian multiplicative chaos. In short, one can view $b\in\R\mapsto \X((-\infty,b])$ as a Gaussian log-correlated field regularized on scale $\hbar \asymp N^{-1}$ in the bulk.
This has been a very active research topic lately and we cannot review all related results here.
We simply mention that for the Gaussian $\beta$-ensemble\footnote{The
  Gaussian $\beta$-ensemble corresponds to a fermionic model with
  Calogero-Sutherland type interaction for $\beta\neq2$ and the
  $\O(1)$-term has been conjectured in \cite{smith_full_2021}. This
  term has been computed in the classical cases $\beta\in\{1,2,4\}$.}
(normalised so that $\mathcal D =(-1,1)$), with $\mathbf{Z}(b) :=
\frac{\X((-\infty,b])- \tfrac2\pi\int_{\lambda\le b}
  (1-\lambda^2)_+^{1/2} \dd \lambda }{\sqrt{\log N}}$, one has (uniformly) for $b_1,\dots, b_k \in \mathcal D$ such that $\Sigma$ exists, 
\begin{equation} \label{log_correlated}
\big(\mathbf{Z}(b_1),\cdots, \mathbf{Z}(b_k) \big) \Rightarrow \mathcal{N}_{0,\Sigma}  
\quad\text{where}\quad
\Sigma_{ij} =  \frac{1}{\pi^2\beta} \lim_{N\to\infty}\bigg( \frac{1}{\log N} \log_+ \min\bigg\{ N (1-b_j^2)^{3/2}, \frac{1-b_j^2}{|b_i-b_j|} \bigg\} \bigg) .
\end{equation}
\eqref{log_correlated} is due to \cite{gustavsson_gaussian_2005} for GUE ($\beta=2$) and also holds for general Wigner matrices \cite{orourke_gaussian_2010,mody_log-characteristic_2024} and $\beta$-ensembles \cite{bourgade_optimal_2022}. 
This CLT captures the behavior down to the microscopic scale as well as the edge effect, it should be compared to Theorem~\ref{thm:clt} for free fermions on $\R$. 
We refer to \cite{lambert_strong_2023,lambert_strong_2021} for more details on the log-correlated structure of the Gaussian $\beta$-ensemble, and \cite{lambert_law_2019,paquette_extremal_2022,bourgade_optimal_2023} for the most recent results on the maximum fluctuations of the eigenvalue counting function of different random matrix model. 
A result analogous to Theorem~\ref{thm:clt} has also been obtained recently for the circular Riesz gas using methods from statistical mechanics \cite{boursier_optimal_2023}.

\medskip

For the eigenvalues of Hermitian (or unitary) matrix models, which are also determinantal processes, there are explicit formulae for the Laplace transform of any linear statistic in terms of large Hankel (or Toeplitz) determinants. The basic example, known as the circular unitary ensemble (CUE), corresponds to determinants of the type \eqref{eq:Laplace_determinantal} involving spectral projector $\Pi_{\hbar}=\1(-\hbar^{2}\Delta_{\S^1} \le 1)$. In contrast to the Schrödinger model, Toeplitz determinants with a smooth symbol, as well as their continuous counterpart with $\1(-\hbar^{2}\Delta_{\R} \le 1)$, can be analyzed by a miscellany of different techniques and the (Gaussian) behavior of such determinants are known as \emph{Szeg\H{o} asymptotics}. These asymptotics have constituted a central problem in mathematical physics because of a plethora of applications (Ising model, impenetrable bosons, Toda chain, random matrices, etc.), see \cite{deift_toeplitz_2013} for a review and \cite{deleporte_central_2023} for more on the relation to one-dimensional free fermions. 
A special class of singular symbols, called the Fisher-Hartwig class,
plays a fundamental role in these applications, in particular to
describe the fluctuations of the characteristic polynomial and
eigenvalue counting functions of matrix models. These determinants also arise when computing the entanglement entropy of one-dimensional quantum spin chains \cite{jin_quantum_2004,keating_random_2004,keating_entanglement_2005}.
Fisher-Hartwig \cite{shuler_toeplitz_1969} already conjectured that the asymptotics of these determinants involved  extra $\log N$ terms. The first results on Fisher-Hartwig determinants go back to Widom's seminal work \cite{widom_toeplitz_1973} (see also \cite{widom_asymptotic_1974,widom_asymptotic_1976}) and the full conjecture was resolved in \cite{deift_asymptotics_2011} based on the connection with orthogonal polynomials with varying weights and the corresponding Riemann-Hilbert problem. 
The article \cite{deift_asymptotics_2011} relies on the (non-linear) steepest descent method \cite{deift_steepest_1993} for oscillatory Riemann-Hilbert problem to obtain precise asymptotics valid, for instance, for the Laplace transform of counting statistics. This requires  to develop new \emph{local parametrices} at the singular points. 
Our approach (to the free fermions case) is similar in spirit, since by using Fourier integral approximation for the Schrödinger propagator, we reduce the problem of computing the $\log N$ term in the variance to applying the stationary phase method to an oscillatory integral with a discontinuous amplitude. 
We cannot review the rich literature on Fisher-Hartwig determinants and we refer instead to the survey \cite{deift_toeplitz_2013}, \cite{charlier_asymptotics_2019,fahs_uniform_2021} for recent developments and further references, as well as \cite{krasovsky_correlations_2007,webb_characteristic_2015,claeys_how_2021,junnila_multiplicative_2023} for some probabilistic applications to random matrices. 

In dimension $n \ge 2$, there are several results in the physics literature concerning disk counting statistics of free fermions confined by a rotation-invariant potential \cite{Smith_21,smith_counting_2022} (by scaling, one can compare these predictions to the asymptotics of Theorem~\ref{thm:J2} in case $V(x) = |x|^q$ with $q>0$ and $\mu=1$). 
According to \eqref{eq:variance}, writing $V(x) =\mathrm{v}(r)$, for a disk $\Omega = \{|x|\le r\}$ within the bulk ($\mathrm{v}(r)<\mu$), 
\[
C_{\Omega} = \frac{\cst_{n-1}| \partial B_{n}| }{2\pi^2} r^{n-1}\big(\mu-\mathrm{v}(r)\big) ^{\frac{n-1}2} = \frac{n\cst_{n-1}\cst_n}{2\pi^2} r^{n-1}\big(\mu-\mathrm{v}(r)\big)^{\frac{n-1}2}  = \frac{1}{\pi^2\Gamma(n)} \big(2\pi r\big)^{n-1} \big(\mu-\mathrm{v}(r)\big)^{\frac{n-1}2} 
\]
since $\cst_{n-1}\cst_n = 2 \frac{(2\pi)^{n-1}}{\Gamma(n+1)} $ by Legendre duplication formula.
Hence, we obtain  as $\hbar\to0$,
\[ 
\var \X(rB_n)  =  \frac{\hbar^{1-n}\big( \log \hbar^{-1}+\O(1)\big)}{\pi^2\Gamma(n)}  \Big(r\sqrt{1-\mathrm{v}(r)}\Big)^{n-1} 
\]
which is consistent with \cite{Smith_21}, formula (4) with $\mu \leftarrow 1/\hbar$.
We also refer to formulae (S55) and (S56) in the supplementary material of \cite{Smith_21} for an expression of the $\O(1)$ correction terms in case of a rotation-invariant potential.

\subsection{Widom's conjecture \& entanglement entropy}
\label{sec:Widom}

From the viewpoint of spectral theory and physical applications, it is
of interest to describe spectral functions of the type
$g(\Pi_\hbar|_{\Omega})$ where $g:[0,1] \to \R_+$ is continuous,
$\Pi_\hbar$ is a semiclassical (self-adjoint) projection on
$L^2(\R^n)$, $\Omega \Subset \R^n$ be open and let $\Pi_\hbar|_{\Omega} = \1_\Omega\Pi_\hbar\1_\Omega $. The case where $g(0)=g(1)=0$ is of special interest. For instance, with $g:t \mapsto t(1-t)$, 
\[
g(\Pi_\hbar|_{\Omega}) =  \1_\Omega(\Pi_\hbar- \Pi_\hbar\1_\Omega\Pi_\hbar)\1_\Omega 
\]
so that 
\[
\tr g(\Pi_\hbar|_{\Omega}) =  \tr( (\Pi_\hbar- \Pi_\hbar\1_\Omega\Pi_\hbar)\1_\Omega  )
=- \tfrac12 \tr [\Pi_{\hbar},\1_\Omega]^2 ,
\]
which corresponds to the case of Theorem~\ref{thm:J2}; see more
generally Lemma \ref{lem:spectral_link}.

The model case where $\Pi_{\hbar}=\1(-\Delta<\hbar^{-2})$ in dimension
1 and $\Omega$ is an interval has been studied in \cite{landau_eigenvalue_1980} using the Mellin transform and the results were generalised to pseudodifferential operators with discontinuous symbols (replacing $\Pi_\hbar$ by $\operatorname{Op}_\hbar(a)$ where the amplitude $a:\R^{2}\to \C$ is  not necessarily continuous) in  \cite{widom_class_1982}.
Widom also conjectured in
\cite{widom_class_1982} an analogous result for pseudodifferential
operators with discontinuous symbols in higher dimensions. A notable first step was the
computation of the variance in \cite{gioev_entanglement_2006} (that
is, the case $g:\lambda\mapsto \lambda(1-\lambda)$), as well as an
upper estimate on the entropy (using a cruder bound on the
$\J^1$-norm) of the form $S_{\Omega}\leq C\hbar^{1-n}\log(\hbar)^2$. The proof
of the Widom conjecture when $g:\R\to\R$ is analytic, along with some
history and further references, is gathered in
\cite{sobolev_pseudo-differential_2013}.

\medskip
Recent developments focus on the case where $g$ is less regular, including
the entropy \cite{sobolev_functions_2017}, pseudodifferential
operators with matrix-valued symbols
\cite{finster_fermionic_2023,bollmann_widoms_2023}, and magnetic
Laplacians with large magnetic fields
\cite{charles_entanglement_2018,leschke_asymptotic_2021,pfeiffer_entanglement_2023}; in this
latter case the
analysis is slightly different due to a spectral gap in the ground state.
Having in mind the rich
applications to fermionic many-body physics and determinantal point
processes, it is relevant to try and extend this program to more
general spectral projectors.
Inspired by Widom's conjecture \cite{widom_class_1982}, in the context of this article, we propose the following conjecture.

\begin{conj}\label{conj:Widom}
Under the Assumptions \ref{asymp:V}, let $\Pi_{\hbar,\mu}$ be as in \eqref{Fermiproj} and let $\Pi_\hbar|_{\Omega} =\1_{\Omega} \Pi_{\hbar,\mu} \1_{\Omega}$. 
Let $g: [0,1] \to\R$ be continuous with $g(0)=g(1)=0$ such that either $g\ge 0$ or $t\mapsto \frac{g(t)}{t(1-t)}$ is in $L^1$. Then, as $\hbar \to 0$,
\[
\frac{\tr g(\Pi_\hbar|_{\Omega})}{{(2\pi\hbar)^{1-n}}\log\hbar^{-1}} \to 2C_{\Omega}\int_{[0,1]}\frac{g(\lambda)}{\lambda(1-\lambda)}\dd \lambda
\]
where $C_{\Omega}$ is given by \eqref{eq:variance} with $f=1$.
\end{conj}

Lemma \ref{lem:spectral_link} gives a relationship
between Schatten-like norms of $[\Pi_{\hbar,\mu},\1_{\Omega}]$ and
traces of spectral functions of $\Pi_{\hbar}|_{\Omega}$. In
particular, Conjecture \ref{conj:Widom} contains as a particular case
the asymptotics of $\|[\Pi_{\hbar,\mu},\1_{\Omega}\|_{\J^1}$, of the
von Neumann entropy of Theorem \ref{thm:ent}, and also the Rényi
entropies. 
In these cases, as we already mentioned, this conjecture is also supported by numerical
simulations \cite{vicari_entanglement_2012} and  
 theoretical physics computations
 \cite{calabrese_exact_2012,calabrese_random_2015,dubail_conformal_2017}. Moreover,
 formula \eqref{eq:variance} agrees with the geometric
constants appearing in previously studied forms of the conjecture \cite{sobolev_pseudo-differential_2013,gioev_szego_2006,leschke_scaling_2014}.

In future work, we intend to return to proving Widom-type asymptotics for general
spectral functions, and to bridge the gap with the existing results concerning pseudodifferential operators with discontinuous symbols.

\subsection{Acknowledgements.}

A.D. acknowledges the support of the CNRS PEPS 2021 grant. G.L. acknowledges the support of the starting grant 2022-04882 from the Swedish Research Council and of the starting grant from the Ragnar S\"oderbergs Foundation. 

\section{Pseudodifferential operators and Fourier Integral operators}
\label{sec:semiclassic}

\subsection{Notation} \label{sec:not}

Throughout the article, we will not always emphasize that the operators $H= H_\hbar, \Pi=\Pi_{\hbar,\mu}$, and 
that functions, depend on the semiclassical parameter $\hbar$.
In particular, we will consider a class $S^\delta$ of functions (symbols) which depend on a small scale $\delta(\hbar)$.
\begin{defn}\label{def:Sdelta}
Let  $d\in\N$ and $\delta = \delta(\hbar)\in[\hbar,1]$. 
Given $\Omega \Subset \R^d$ open and bounded,  we define the following class of functions:
\[
S^{\delta}(\Omega)
= \big\{ a\in \Co^\infty_c(\Omega) :   \|a\|_{\Co^k} \le C_k\delta^{-k} 
\text{ for every $k\in\N_0$}\big\}. 
\]
We also use the notations $a\in S^{\delta}(\R^d)$, or $a\in
S^{\delta}$, if $a\in S^{\delta}(\Omega)$ for some $\Omega \Subset
\R^d$ bounded. \\
For $\Theta = \Theta(\hbar)\in (0,1]$, we also write $a = \O_{S^\delta}(\Theta)$ if $a/\Theta \in S^\delta$. 
\end{defn}

In particular, $a\in S^1(\R^d)$ if it is $C^\infty_c(\R^d)$, $\supp
(a) \subseteq \Omega$ with $\Omega$ bounded, and $\Omega, \|a\|_{\Co^k}$ are independent of $\hbar$ for all~$k\in\N$.

\smallskip

If a symbol depends on several variables, e.g.~$(x,z)\in\R^{n\times
  m}$, with different regularity exponents, say $\epsilon, \delta \in[\hbar,1]$, we will denote $a \in S^\epsilon_x \times S^\delta_z$ if for every indices $\alpha\in\N_0^n, \beta\in\N_0^m$ with $|\alpha|+|\beta|\le k$, 
\[
\sup_{x,z} \big|\partial_x^\alpha\partial^\beta_z a(x,\xi)\big| \le C_k \epsilon^{-|\alpha|} \delta^{-|\beta|}. 
\]

In the sequel, by \emph{constants}, we mean positive numbers, e.g.~$\Cst,\cst$,  independent of $\hbar$. Such \emph{constants} are allowed to depend on the potential $V$ and other auxiliary parameters. Sometimes, we denote $\cst_\alpha$ to emphasize that the constant $\cst$ depends on a parameter $\alpha$. 

\smallskip

We declare right now the following objects and conventions, which we
fix for the rest of the article.
\begin{enumerate}[leftmargin=*]
\item Replacing $V$ with $V-\mu$, we assume that $\mu=0$. The droplet
  is $\mathcal{D} = \{V<0\}$.
  \item $\Omega \Subset \mathcal{D}$ is a fixed relatively compact open set, with a
  smooth boundary. $\mathrm{w} : \R^n \to \R$ is a smooth, tends to
  $+\infty$ at
  $\infty$, such that $\Omega :=\{\mathrm{w}< 0\}$ and
  $\partial_x\mathrm{w} \neq 0$ on $\partial\Omega =
  \{\mathrm{w}=0\}$. $\Omega' \Subset \{V<0\}$ is an open
  neighbourhood of $\Omega$ inside the droplet.
\item $\underline{\tau}, \underline{\ell}, \underline{\cst}$ are small positive constants. The constant $\underline{\ell}$ is chosen much
  smaller than $\underline{\cst}$, and the constant $\underline{\tau}$
  is chosen much smaller than $\underline{\ell}$, and all constants
  are small, in a way which depends on $V$ and $\Omega$. In
  particular, we assume that $V$ is $C^\infty$ on $\{V<2\underline{\cst}\}$, that 
  $\partial_x V(x) \neq 0$ for all $x\in \{|V| \le
  \underline{\cst}\}$, and that $\Omega'\subset \{V\leq -2\underline{c}\}$.
\item  $\rho\in \mathcal{S}(\R,\R_+)$ is even with  $\int_\R
\rho(\lambda)\dd \lambda = 1$. Its Fourier transform $\hat{\rho}$ is supported inside
$[-\underline{\tau},\underline{\tau}]$.
\item Given $f\in L^{\infty}(\R^d)$, for $\delta\in(0,1]$, we denote
$
f_{\delta}:x\mapsto f(\delta^{-1}x). 
$
This notation admits one exception: we will denote 
\(
\rho_{\hbar}: u \mapsto \hbar^{-1}\rho(\hbar^{-1}u),
\)
that is, $\rho$ is scaled as a density rather than a function.
\item Given $\vartheta  \in \Co^\infty_c(\R)$ and given $\varkappa\in L^{\infty}(\R)$ with compact support,
we define
\(
\widetilde{\varkappa}:=\vartheta\cdot (\varkappa*\rho_{\hbar}).
\)
Note that if $\varkappa \in S^\eta$ with  $\eta\in[\hbar,1]$ (see~Definition~\ref{def:Sdelta}), then $\widetilde{\varkappa}\in S^\eta$. 
\end{enumerate}

\subsection{Pseudodifferential calculus}

The goal of this section is to gather classical results on spectral functions of $H$. 
The main ingredient is that spectral functions involving the propagator, of the form $\vartheta(H)e^{-i\frac{tH}{\hbar}}$ where $\vartheta\in \Co^{\infty}_c(\R)$ and $t\in \R$ is small, are given by \emph{Fourier
Integral operators}, up to negligible errors. 
Such strong approximations go back to \cite{chazarain_spectre_1980,helffer_comportement_1981}, but for
consistency, we extract these results from our previous work \cite[Proposition 2.11]{deleporte_universality_2024} as they fit the exact hypotheses of our main claims (see also  \cite{robert_autour_1987,dimassi_spectral_1999} for further reference). 
Recall that the potential $V$ satisfies Assumptions~\ref{asymp:V} with $\mu=0$ and  $\underline{\cst}>0$ is a small constant such that $V\in \Co^\infty$ on $\{V < \underline{\cst}\}$. 

\begin{prop}\label{prop:propag_as_FIO}
Let $\vartheta\in\Co^{\infty}_c((-\infty,\underline{\cst}])$. There
exists a smooth function $\varphi : \R^{2n+1}\to \R$ (independent of
$\hbar$) and $s\in S^1(\R^{3n+1})$, such that, uniformly for $t\in[-\underline{\tau},\underline{\tau}]$,
\[
\vartheta(H)e^{i\frac{tH}{\hbar}} : (x,y)\mapsto \frac{1}{(2\pi\hbar)^n}\int
e^{i\tfrac{\varphi(t,x,\xi)-\xi\cdot
y}{\hbar}}s(t,x,y,\xi)\dd \xi +\O_{\J^1}(\hbar^{\infty}).
\]
For every $t\in[-\underline{\tau},\underline{\tau}]$, the amplitude
$s$ is supported in a neighbourhood of
\[\{|x-y|\le \underline{\ell}, \;\vartheta(V(x)+|\xi|^2)\neq 0\}.\] Moreover, its principal part is given on the diagonal by, at $t=0$, 
\[
s(0,x,x,\xi)|_{\hbar=0}=\vartheta(|\xi|^2+V(x)) , \qquad (x,\xi) \in \R^{2n}. 
\]

The phase $\varphi$ is the solution of the following initial value
problem for the Hamilton-Jacobi equation on a neighbourhood of the support of $s$:
\begin{equation*}\label{eq:HJ}
\partial_t\varphi(t,x,\xi)=V(x)+|\partial_x\varphi(t,x,\xi)|^2 \qquad
\qquad \varphi|_{t=0}:(x,\xi)\mapsto x\cdot \xi.
\end{equation*}
\end{prop}
\begin{rem}Note that $s$ silently depends on $\hbar$; in fact it can be taken to
be a smooth function of $\hbar$ on $[0,\hbar_0]$. Usually, $s$ is
represented by its Taylor expansion at $\hbar=0$, and the first term
of this expansion (the value at $0$) is called the \emph{principal
  symbol}. We will use this notion throughout this article, without
writing down explicitly the dependence of $s$ (and other amplitudes)
on $\hbar$.
\end{rem}
Using the Fourier inversion formula, 
Proposition \ref{prop:propag_as_FIO}  provides an approximation of
spectral functions of $H$ which are smooth on scales larger than
$\hbar$. More precisely we obtain the following corollary.
\begin{corr} \label{cor:propag}  
Let $\varphi$ be the phase from Proposition \ref{prop:propag_as_FIO}.
Let $\vartheta  \in \Co^\infty_c((-\infty,\underline{\cst}])$,  let $\varkappa\in L^{\infty}(\R)$ with compact support and let
\(
\widetilde{\varkappa}=\vartheta\cdot (\varkappa*\rho_{\hbar})
\)
as in the notations of Section \ref{sec:not}. Then, there exists $a\in S^1(\R^{3n+1})$ so that
\begin{equation*}
\widetilde{\varkappa}(H) : (x,y) \mapsto \frac{1}{(2\pi\hbar)^{n+1}}\int
e^{i\tfrac{\varphi(t,x,\xi)-y\cdot
\xi-t\lambda}{\hbar}}b(t,x,y,\xi)\varkappa(\lambda) \dd
\xi \dd t \dd \lambda+\O_{\J^1}(\hbar^{\infty}) .
\end{equation*}
The amplitude $b$  is supported in $\big\{ t\in
[\underline{\tau},\underline{\tau}] , |x-y| \le \underline{\ell}\}$
and on a small neighbourhood of \[\{(x,\xi) :\vartheta(V(x)+|\xi|^2)\neq 0\}.\]
\end{corr}
Up to a negligible error, controlled in the
trace-norm, the kernel of the operator $\widetilde{\varkappa}(H)$ is
given by an oscillatory integral. In fact, this kernel can be further
simplified if $\varkappa \in S^\eta$ with $\eta \ge \hbar^{\frac 12}$; see Proposition~\ref{prop:pseudo}. 
Let us also record two important observations regarding the support of the amplitude and the approximation of the phase. 

Since $ \underline{\tau}$ is much smaller than $\underline{\ell}$,
which itself is much smaller than $\underline{\cst}$, we obtain that,
if $\vartheta$ is supported in $[-\underline{\cst},\underline{\cst}]$,
one has
\begin{equation} \label{supp_bulk}
x\in \Omega',\;y\in\Omega',\; (t,x,y,\xi) \in \supp(a) \quad \Longrightarrow\quad |\xi| \ge \underline{\cst}. 
\end{equation}

Moreover, using the Hamilton-Jacobi equation, the phase has the following expansion as  $t\to 0$, 
\begin{equation} \label{phasephi}
\varphi(t,x,\xi)=x\cdot \xi + t( V(x) +|\xi|^2) + \O(t^2) 
\end{equation}
uniformly over the support of $a$, with a smooth error term, so that 
$\partial_t\varphi(t,x,\xi)=  V(x) +|\xi|^2  + \O(t)$ as  $t\to 0$.

\subsection{Preliminary commutator estimates}

We are interested in estimates for the Hilbert-Schmidt and trace-norm of the commutator $[\Pi,f|_\Omega]$ where $\Pi=\1_{\R_-}(H)$ is a spectral projector and $\Omega \Subset \mathcal{D}$ a subset of the droplet.  
A preliminary step to prove Theorem~\ref{thm:J2} consists in replacing $\Pi$ by a Fourier integral operator $\widetilde{\Pi}$, up to a negligible error when computing  $[\Pi,f|_\Omega]$. 
This operator is obtained by \emph{regularizing} $\Pi$ on scale $\hbar$
using Proposition~\ref{prop:propag_as_FIO}.

\begin{prop}\label{prop:regularised_commutator}
Let $\varphi$ be as in Proposition~\ref{prop:propag_as_FIO} and let $f\in C^{\infty}(\R^n)$. 
There exists $a \in  S^1(\R^{3n+2})$ such that
\[
[\Pi,f|_\Omega] = [\widetilde{\Pi}, f|_\Omega]+ \O_{\J^1}(\hbar^{1-n}) 
\]
where $\widetilde{\Pi}$ is a trace-class operator with a kernel 
\begin{equation} \label{kern_reg}
\widetilde{\Pi}: (x,y) \mapsto \frac{1}{(2\pi\hbar)^{n+1}}\int
e^{i\tfrac{\varphi(t,x,\xi)-y\cdot
\xi-t\lambda}{\hbar}}a(t,x,y,\xi,\lambda)\1\{\lambda\le0\} \dd
\xi \dd t \dd \lambda . 
\end{equation}
The amplitude $a$ is supported in
$\big\{t\in[-\underline{\tau},\underline{\tau}], |x-y|\le
\underline{\ell} ,|V(x)+|\xi|^2|\le \underline{\cst}, |\lambda| \le
\underline{\ell}\}$, satisfies \eqref{supp_bulk} and its principal part satisfies at $t=0$,
\begin{equation*}
a(0,x,x,\xi,\lambda)|_{\hbar=0}=\vartheta(|\xi|^2+V(x)) \chi(\lambda). 
\end{equation*}
\end{prop}
We expect  the error~$\O_{\J^1}(\hbar^{1-n})$ to be sharp; it is negligible compared to the asymptotics of Theorem~\ref{thm:J2}.
Note also that the kernel \eqref{kern_reg} has a discontinuous
amplitude, which makes its analysis non-trivial.

The proof of Proposition~\ref{prop:regularised_commutator} occupies
the rest of this section, the operator $\widetilde{\Pi}$ being constructed using Corollary~\ref{cor:propag} with $\varkappa = \chi\1_{\R_-}$ and two cutoffs: 
\begin{itemize}[leftmargin=*]
\item $\vartheta :\R \to[0,1]$ supported in
  $[-\underline{\cst},\underline{\cst}]$ and equal to $1$ on
  $[-\underline{\cst}/2, \underline{\cst}/2]$.
\item $\chi:\R \to [0,1]$, compactly support in $[-\underline{\ell},\underline{\ell}]$ and equal to 1 on a neighborhood of 0.
\end{itemize}
Then, we have
\begin{equation} \label{Pitilde}
\widetilde{\varkappa}(H) = \vartheta(H) (\chi\1_{\R_-}*\rho_{\hbar})(H) = \widetilde{\Pi}+\O_{\J^1}(\hbar^{\infty}) 
\end{equation}
where $\widetilde{\Pi}$ is the Fourier integral operator \eqref{kern_reg} and the amplitude 
\(
a : (t,x,y,\xi,\lambda) \mapsto \sqrt{2\pi}\widehat{\rho}(t)s(t,x,y,\xi)\chi(\lambda)
\)
with $s$ as in Proposition~\ref{prop:propag_as_FIO}. 
This yields both the support condition for $a$ and its principal part.

\smallskip

It remains to relate $\widetilde{\varkappa}(H)$ to the true projection $\Pi = \1_{\R_-}(H)$. 
The first (simple) step is to replace $\Pi$ by $\vartheta(H)\1_{\R_-}(H)$ where $\vartheta  \in \Co^\infty_c(\R)$ is supported in an arbitrary small neighbourhood of $\mu=0$.  

\begin{prop} \label{prop:J1commsmooth}
Let $\vartheta:\R \to\R$ be smooth and compactly supported in $(-\infty,\underline{\cst}]$. 
Let $\Omega \Subset \R^n$ be a set with a smooth boundary  and  $f \in \Co^\infty(\R)$, then as $\hbar\to0$, 
\[
\big\| [\vartheta(H), f|_\Omega] \big\|_{\J^1} =\O(\hbar^{1-n}) . 
\]
\end{prop}
\noindent
Since it requires some auxiliary estimates, we postpone the proof of
Proposition~\ref{prop:J1commsmooth} to the end of this Section.

\smallskip

For the second step, going from $\vartheta(H)\1_{\R_-}(H)$ to $\widetilde{\Pi}$, we need the following estimates for the eigenvalue counting function on microscopic scale. The proof follows the reasoning from \cite[Theorem 4]{helffer_comportement_1981}. 
We emphasize that the argument is similar in spirit, albeit technically much more straightforward, to the application of the stationary phase method we will use to prove Theorems~\ref{thm:J2} and~\ref{thm:J1}. 

\begin{lem} \label{lem:Weyl}
Uniformly for $\lambda\in [-\underline{\cst},\underline{\cst}]$,
\[
\tr \1\{|H-\lambda| \le \hbar\} = \O(\hbar^{1-n}).
\]
\end{lem}

\begin{proof}
First observe that it follows directly from Proposition~\ref{prop:propag_as_FIO} with $t=0$ that for any smooth function $\vartheta : \R\to\R_+$ supported on $(-\infty,\underline{\cst}]$, 
\begin{equation} \label{Weyl_global}
\|\vartheta(H)\|_{\J^1} = \int_{\R^n} \vartheta(H)(x,x) \dd x
= \frac{1}{(2\pi\hbar)^n}\int b(0,x,x,\xi)\dd x \dd \xi +\O(\hbar^{\infty})
= \O(\hbar^{-n}) .
\end{equation}
Here, we used that the operator $\vartheta(H)\ge 0$ is trace-class so that $\|\vartheta(H)\|_{\J^1} = \tr\vartheta(H)$. 

\smallskip

Let  $\mathcal{H}(x,\xi) := V(x)+|\xi|^2$ for $(x,\xi)\in\R^{2n}$.
Using the Notation of Section \ref{sec:not}, with $\widetilde{\delta_\lambda} := \vartheta\cdot (\delta_\lambda*\rho_{\hbar})$ for $\rho\in \mathcal{S}(\R,\R_+)$ and $\vartheta$ a cutoff equal to 1 on $[-\underline{\cst}/2, \underline{\cst}/2]$, one has $\rho(\frac{\cdot-\lambda}{\hbar}) =\hbar \widetilde{\delta_\lambda}   + \O(\hbar^\infty\vartheta)$ if $\lambda \in [-\underline{\cst}/2, \underline{\cst}/2]$. Then, by Proposition~\ref{prop:propag_as_FIO} and \eqref{Weyl_global}, 
\[
\rho\big(\tfrac{H-\lambda}{\hbar}\big): (x,y) \mapsto \frac{1/2\pi}{(2\pi\hbar)^{n}}\int
e^{i\tfrac{\varphi(t,x,\xi)-y\cdot
\xi-t\lambda}{\hbar}}b(t,x,y,\xi) \dd
\xi \dd t  +\O_{\J^1}(\hbar^{\infty}) 
\]
where $(t,x,\xi) \mapsto b(t,x,x,\xi) $ is supported on $t\in[-\underline{\tau},\underline{\tau}]$ and a small  neighbourhood of $\{\mathcal{H}(x,\xi)=\lambda\}$, and the error term is independent of $\lambda\in[-\underline{\cst}/2, \underline{\cst}/2]$.
In addition, this holds for any $\rho\in \mathcal{S}(\R,\R_+)$, not necessarily a probability density.
Thus, 
\[
\tr \rho\big(\tfrac{H-\lambda}{\hbar}\big)
\le \frac{1}{(2\pi\hbar)^{n}}\bigg| \int
e^{i\tfrac{\varphi(t,x,\xi)-x\cdot
\xi-t\lambda}{\hbar}}b(t,x,x,\xi) \dd x \dd
\xi \dd t \bigg| +(\hbar^{\infty}) .
\]
We now estimate the previous integral using the stationary phase method.  
Observe that according to \eqref{phasephi}, the phase is
$t(\mathcal{H}(x,\xi)-\lambda +\O(t))$ and, by
Assumptions~\ref{asymp:V}, the sets $\{\mathcal{H}(x,\xi)=\lambda\}$
for $\lambda\in[-\underline{\cst},\underline{\cst}]$ are diffeomorphic
to spheres. In particular, letting $\eta=\mathcal{H}(x,\xi)$, we can
make a (non-degenrate) change of coordinates $(x,\xi) \in
\R^{2n}\rightarrow (\eta,\omega) \in \R\times \S^{2n-1}$ such that
\[
\int e^{i\tfrac{\varphi(t,x,\xi)-x\cdot
\xi-t\lambda}{\hbar}}b(t,x,x,\xi) \dd\xi \dd t 
= \int e^{i\tfrac{t (\varpi(t,\eta,\omega)-\lambda)}{\hbar}} a(t,\eta,\omega) \dd \omega \dd \eta  \dd t 
\]
where $\varpi(t,\eta,\omega) = \eta +\O(t)$ with a smooth error and the amplitude $a\in S^1(\R^{2n+1})$ is supported in 
$\{t\in[-\underline{\tau},\underline{\tau}]\}$. 
Now, we apply the stationary phase method in the variable $(t,\eta)$, keeping $\omega\in \S^{n-1}$ fixed. 
There is a unique critical point $(t,\eta) = (0,\lambda)$ and the Hessian is the identity at the critical point. 
By Proposition~\ref{lem:statphase}, we obtain
\[
\bigg| \int
e^{i\tfrac{\varphi(t,x,\xi)-x\cdot
\xi-t\lambda}{\hbar}}b(t,x,x,\xi) \dd x \dd
\xi \dd t \bigg| =\O(\hbar)
\]
where the error term is uniform for $\{\lambda\in [-\underline{\cst},\underline{\cst}]\}$. By positivity, we conclude that there is a constant $C$ so that 
\[
\tr \1\{|H-\lambda| \le \hbar\} \le  \tr \rho\big(\tfrac{H-\lambda}{\hbar}\big)\le C \hbar^{1-n} . \qedhere
\]
\end{proof}

We are now ready to finish the proof of Proposition \ref{prop:regularised_commutator}.

\begin{proof}[End of the proof of Proposition~\ref{prop:regularised_commutator}]  
By Proposition~\ref{prop:J1commsmooth}, since the operator $H$ is bounded from below
\[
[\Pi,\1_\Omega] = [\vartheta(H)\Pi, \1_\Omega] +[(1-\vartheta(H))\Pi,\1_\Omega]=[\vartheta(H)\Pi, \1_\Omega] +\O_{\J^1} (\hbar^{1-n}) . 
\]
Recall that that $\widetilde{\varkappa}:=\vartheta\cdot (\chi\1_{\R_-}*\rho_{\hbar})$ where the mollifier $\rho$ is a Schwartz function and $\vartheta,\chi$ are appropriate cutoffs. By construction, for every $k\in\N$, there is a function $\chi_k : \R\to[-1,1]$, smooth with compact support so that 
\[
\widetilde{\varkappa} = \vartheta\1_{\R_-} + \chi_k (\hbar^{-1}\cdot)
+\O(\hbar^k\vartheta) . 
\]
Consequently, by Lemma~\ref{lem:Weyl} (see also \eqref{Weyl_global}), 
\[
\widetilde{\varkappa}(H)= \vartheta(H)\Pi +\O_{\J^1} (\hbar^{1-n})
\]
By \eqref{Pitilde}, this concludes the proof. 
\end{proof}

It remains to prove Proposition \ref{prop:J1commsmooth}. We start with
the following estimate from \cite{dimassi_spectral_1999}, formula (9.2): Let $K\in L^1(\R^{2n})$ be an integral kernel, then the corresponding operator satisfies
\begin{equation} \label{DS}
\|K\|_{\J^1} \le \sum_{\alpha\in\N^{2n}_0:|\alpha|\le 2n+1} \|\partial^\alpha_{x,y} K\|_{L^1(\R^{2n})} . 
\end{equation}

\begin{lem} \label{lem:J1commsmooth}
Let $\vartheta:\R \to\R$ be smooth and supported in $(-\infty,\underline{\cst}]$ and let $f:\R^n\to\R$ with compact support and $\|f\|_{\Co^1} \le 1$. Then, as $\hbar\to0$,
\[
[\vartheta(H),f] = O_{\J^1}(\hbar^{1-n})  . 
\]
\end{lem}

\begin{proof}
By Proposition~\ref{prop:propag_as_FIO} with $t=0$, there is $a\in S^1(\R^{3n})$ so that 
\begin{equation} \label{pseudoop}
\vartheta(H) : (x,y)\mapsto \frac{1}{(2\pi\hbar)^n}\int
e^{i\tfrac{\xi\cdot
(x-y)}{\hbar}}a(x,y,\xi)\dd \xi +\O_{\J^1}(\hbar^{\infty}).
\end{equation}
Hence,
\[
[\vartheta(H),f] = K+\O_{\J^1}(\hbar^{\infty}) , \qquad 
K : (x,y)\mapsto \frac{1}{(2\pi\hbar)^n}\int
e^{i\tfrac{\xi\cdot
(x-y)}{\hbar}}a(x,y,\xi)(f(x)-f(y))\dd \xi . 
\]
Up to a unitary scaling (changing $x\leftarrow x\hbar$), which fixes the $\J^1$-norm, the kernel $K$ is equivalent to 
\[
\widetilde{K} : (x,y) \mapsto  \frac{1}{(2\pi)^{n}} \int
e^{i{(x-y)\cdot \xi}}a(\hbar x,\hbar y,\xi)\big(f(\hbar x)-f(\hbar y)\big)\dd \xi 
= \widehat{a} (\hbar x, \hbar y, x-y)\big(f(\hbar x)-f(\hbar y)\big) .
\]
where $\widehat{a}(x,y,\cdot)$ denotes the Fourier transform (appropriately normalized) of $\xi \mapsto a(x,y,\xi)$ for $(x,y)\in\R^{2n}$ -- in particular, $\widehat{a} \in \mathcal{S}(\R^{3n})$ with norms independent of $\hbar$.

Hence, if $\|f\|_{\Co^1} \le 1$, we can bound for any $k\in\N_0$,
\[
|\widetilde{K}(x,y)| \le   \hbar \frac{ \varphi_k(\hbar x)| |x-y|}{1+|x-y|^{k}} 
\]
where $\varphi_k \in \mathcal{S}(\R^{n})$. 
This estimate implies that $\|\widetilde{K}\|_{L^1(\R^{2n})} = \O(\hbar^{1-n})$. 

In addition, observe that by a similar computation for any index $\alpha \in \N^{2n}_0$ and $k\in\N_0$, 
\[
\partial_{x,y}^\alpha \widetilde{K}(x,y) 
= \O_\alpha\bigg( \hbar \frac{ \varphi_k(\hbar x)| |x-y|}{1+|x-z|^{k}} \bigg)
\]
Thus, using \eqref{DS}, this shows that $\|K\|_{\J^1} = \O(\hbar^{1-n})$, which completes the proof.
\end{proof}

In order to prove Proposition~\ref{prop:J1commsmooth}, we need a similar estimates for (smooth) function which are constant except on a $\hbar$-neighborhood of the boundary $\partial \Omega$. 
Let $\mathrm{w}$ be as in Notation~\ref{def:g} below; this includes functions of the type
$\mathrm{f}= \chi(\hbar^{-1}\mathrm{w})$ with $\chi :\R\to \R_+$ is smooth and compactly supported.  

\begin{lem} \label{lem:J1commnonsmooth}
Let $\vartheta:\R \to\R$ be smooth and supported in
$(-\infty,\underline{\cst}]$. Let $\mathrm{f} \in S^\hbar(\R^n)$ and
suppose that $\mathrm{f}$ is constant on $\{ x\in\R^n : \mathrm{w}(x)
> c \hbar \} $ and constant on $\{  x\in\R^n : \mathrm{w}(x)
<-c \hbar \} $.
Then, as $\hbar\to0$,
\[
\|[\vartheta(H),\mathrm{f}] \|_{\J^1} =O(\hbar^{1-n})  .
\]
\end{lem}

\begin{proof} 
Let $\triangle := \{x\in\R^n : |\mathrm{w}(\hbar x)| \le c \hbar \}$ be a rescaled $\hbar$-neighborhood of $\partial\Omega$. 
In particular, since $\Omega$ is a smooth compact set, 
$|\triangle| =\O(\hbar^{1-n})$. 
Let $f(x) := \mathrm{f}(\hbar x)$ so that all the norms $\|f\|_{\Co^k}$ are all controlled independently of $\hbar$; since $\mathrm{f}\in S^\hbar(\R^n)$. 
Like in the proof of Lemma~\ref{lem:J1commsmooth}, one  has
$\| [\vartheta(H),\mathrm{f}] \|_{\J^1} =  \|\widetilde K\|_{\J^1} +O(\hbar^{1-n})  $ where the operator $\widetilde K$ has kernel 
\[
\widetilde K : (x,y) \mapsto  \widehat{a} (\hbar x, \hbar y, x-y) \big\{f(x)- f(y)\big\} 
\]
where  $\widehat{a} \in \mathcal{S}(\R^{3n})$ with norms independent of $\hbar$. 
In particular, by a simple volume estimate
\[
\| \widetilde K\|_{L^1} = \int_{\triangle^c\times\triangle^c} \hspace{-.5cm} |\widetilde  K(x,y)| \dd y \dd x  +\O(|\triangle|) , \qquad |\triangle|=\O(\hbar^{1-n}) .
\]
Then, by assumptions, for $x,y\in\triangle^c=\R^{n}/\triangle$, 
\[
\big| \mathrm{f}(x)-\mathrm{f}(y)\big|
\le \1\{\mathrm{w}(x) \ge c\hbar, - \mathrm{w}(y) \ge c\hbar\}+\1\{\mathrm{w}(y) \ge c\hbar , - \mathrm{w}(x) \ge c\hbar\} . 
\]
If $|\varkappa(z)| \le \frac{C_k}{(1+|z|)^{2k}}$ for $z\in\R^n$, $k\in\N$, using that $\mathrm{w}(x) \asymp \dist(x, \partial\Omega)$, we also have 
\[\begin{aligned}
\int \1\{\mathrm{w}(\hbar x) \ge c\hbar, - \mathrm{w}(\hbar y) \ge c\hbar\} 
\varkappa(x-y) \dd x\dd y
\le C |\triangle|
=\O(\hbar^{1-n}) .
\end{aligned}\] 
Since $|\widehat{a}(x,y,z)|\le \varkappa(z)$ for $(x,y,z)\in\R^{3n}$ with $\varkappa$ as above, this shows that 
\[
\int_{\triangle^c\times\triangle^c} \hspace{-.5cm} |\widetilde  K(x,y)| \dd y \dd x=\O(\hbar^{1-n}) .
\]
This establishes that $\| \widetilde K\|_{L^1}=\O(\hbar^{1-n})$. 
In addition, since all derivatives of $\widetilde K$ are controlled independently of $\hbar$, we can similarly bound for any $\alpha \in \N_0^n$, $\|\partial^\alpha_{x,y}\widetilde K\|_{L^1(\R^{2n})} = \O_\alpha(\hbar^{1-n})$ using the same argument. 
By \eqref{DS}, we conclude that
$\|\widetilde K\|_{\J^1} = \O(\hbar^{1-n})$.
\end{proof}

\begin{rem} \label{rem:J2norm}
The same argument shows that if $\mathrm{f}$ is constant on $\{ x\in\R^n : |\mathrm{w}(x)| > c \hbar \} $, without any smoothness condition (such as $\mathrm{f}\in S^\hbar$),
\[
\|[\vartheta(H),\mathrm{f}] \|_{\J^2}^2 = \|\widetilde K\|_{L^2(\R^{2n})}^2 =O(\hbar^{1-n})  .
\]
Indeed, there is no need to differentiate the kernel $K$ in order to obtain this estimate. 
In particular, this shows that 
\[
\|[\vartheta(H),\1_\Omega] \|_{\J^2}^2 
=O(\hbar^{1-n})  . 
\]
\end{rem}

Equipped with these estimates, we are ready to proceed to prove Proposition~\ref{prop:J1commsmooth}. 

\begin{proof}[Proof of  Proposition~\ref{prop:J1commsmooth}]
Recall that 
$\Omega \Subset \R^n$ be a smooth compact set and let $\chi\in C^\infty_c(\R)$.
Let $\mathrm{f}= \chi(\hbar^{-1}\mathrm{w})$ where $\chi :\R \to [0,1]$ smooth, compactly supported, so that $\mathrm{f}$ satisfies the assumptions Lemma~\ref{lem:J1commnonsmooth}. 
The first step is to expand
\[\begin{aligned}[]
&[\vartheta(H),\1_{\Omega}] = \vartheta(H)\1_{\Omega}(1-\vartheta(H)) -(1-\vartheta(H))\1_{\Omega}\vartheta(H) \\
&=\mathrm{f} \vartheta(H)\1_{\Omega}(1-\vartheta(H))   -(1-\vartheta(H))\1_{\Omega}\vartheta(H)\mathrm{f}+ (1-\mathrm{f}) \vartheta(H)\1_{\Omega}(1-\vartheta(H)) -(1-\vartheta(H))\1_{\Omega}\vartheta(H)(1-\mathrm{f}) . 
\end{aligned}\]
Let $\mathrm{h} = (1-\mathrm{f})\1_\Omega$. The function
$\mathrm{h}$ also satisfies the assumptions of Lemma~\ref{lem:J1commnonsmooth} so that 
\[\|  [\vartheta(H),\mathrm{f}] \|_{\J^1} , \|  [\vartheta(H),\mathrm{h}] \|_{\J^1}=\O(\hbar^{1-n}).\]
Hence, we have 
\[\begin{aligned}[]
&(1-\mathrm{f}) \vartheta(H)\1_{\Omega}(1-\vartheta(H)) -(1-\vartheta(H))\1_{\Omega}\vartheta(H)(1-\mathrm{f}) \\
&\qquad=\vartheta(H)\mathrm{h}(1-\vartheta(H)) - (1-\vartheta(H))\mathrm{h}\vartheta(H)
+[\vartheta(H),\mathrm{f}]\1_{\Omega}(1-\vartheta(H)) + (1-\vartheta(H))\1_{\Omega} [\vartheta(H),\mathrm{f}]  \\ 
&\qquad= [\vartheta(H),\mathrm{h}]+\O_{\J^1}\big([\vartheta(H),\mathrm{f}]\big) \\
&\qquad= \O_{\J^1}(\hbar^{1-n}) .
\end{aligned}\]

Moreover, we can decompose $\1_\Omega= \mathrm{k} +\theta$ where 
$\mathrm{k}$ satisfies the assumptions of Lemma~\ref{lem:J1commnonsmooth} with $\mathrm{k} =0$ on $ \supp(\mathrm{f})$ and $\theta$ is discontinuous, supported in a $\hbar$-neighborhood of $\partial\Omega$. 
In particular,  $\|\mathrm{f}\|_{L^2}^2, \|\theta\|_{L^2}^2=\O(\hbar)$ and
$\|  [\vartheta(H),\mathrm{k}] \|_{\J^1}=\O(\hbar^{1-n}) $ so that 
\[\begin{aligned}[]
\mathrm{f} \vartheta(H)\1_{\Omega}(1-\vartheta(H))  
&= \mathrm{f} \vartheta(H)\mathrm{k}(1-\vartheta(H)) + \mathrm{f} \vartheta(H)\theta(1-\vartheta(H))  \\
& = \mathrm{f}[\vartheta(H),\mathrm{k}](1-\vartheta(H))
+ \mathrm{f} \vartheta(H)[\vartheta(H),\theta] + \mathrm{f} \psi(H)^2\theta  
\end{aligned}\]
where $\psi=\sqrt{\vartheta(1-\vartheta)}$ -- $\psi\in \Co^\infty_c(\R)$.
By Proposition \ref{prop:J2_estimates_dyadic} and Remark~\ref{rem:J2norm}, 
\[
\|\mathrm{f} \vartheta(H)[\vartheta(H),\theta]\|_{\J^1} \le \|\mathrm{f} \vartheta(H) \|_{\J^2} \|[\vartheta(H),\theta]\|_{\J^2} =\O(\hbar^{1/2-n}\|\mathrm{f}\|_{L^2}) =\O(\hbar^{1-n})
\]
and 
\[
\|\mathrm{f} \psi(H)^2\theta |_{\J^1} \le \|\mathrm{f} \psi(H) \|_{\J^2} \|\theta\psi(H) \|_{\J^2}
=\O(\hbar^{-n}\|\mathrm{f}\|_{L^2}\|\theta\|_{L^2}) =\O(\hbar^{1-n}). 
\]
This shows that 
\[
\|\mathrm{f} \vartheta(H)\1_{\Omega}(1-\vartheta(H))\|_{\J^1} \le \|[\vartheta(H),\mathrm{k}]\|_{\J^1} + \O(\hbar^{1-n})= \O(\hbar^{1-n}) .
\]
Similarly, $\|(1-\vartheta(H))\1_{\Omega}\vartheta(H)\mathrm{f}\|_{\J^1} =\O(\hbar^{1-n})   $, so we conclude that 
\(
\|[\vartheta(H),\Omega]\|_{\J^1}= \O(\hbar^{1-n})  
\).
\end{proof}
\section{Hilbert-Schmidt estimates for mild spectral functions}
\label{sec:hilb-schm-estim}

This section is devoted to preliminary estimates that we will use in
the proof of Theorem \ref{thm:J1}: these estimates are Hilbert-Schmidt
norms of products or commutators between spectral functions of
$H_{\hbar}$ and multiplication operators, with relatively small
supports: the spectral functions will be supported on small
neighbourhoods of $0$, and the multiplication operators will be
supported on small neighbourhoods of $\partial \Omega$.

Even though they are relatively independent, these estimates also
serve as a warm-up for Section \ref{sec:comm} where we will compute
the limit of the Hilbert-Schmidt norm of the ``full'' commutator
$[\Pi_{\hbar},\1_{\Omega}]$ for Theorem \ref{thm:clt}. The same kind of stationary phase
arguments are used in both cases and they are simpler to
apply in the context of this section.

Recall the notations from Section~\ref{sec:not} and assume that $\vartheta  \in \Co^\infty_c(\R)$ is supported in  $[-\underline{\cst},\underline{\cst}]$. 

\begin{prop}\label{prop:J2_estimates_dyadic}
Let $f,g : L^{\infty}(\R,\R)$, be two functions supported in $[-\underline{\cst},\underline{\cst}]$. 
For every $\hbar\in (0,1]$, for every $\eta,\varepsilon\in [\hbar,1]$,
\begin{equation}
\label{eq:J2_dyadic_product}
\|\widetilde{f}_{\eta}(H)g_{\varepsilon}(\mathrm{w})\|_{\J^2}^2 +  \|f_{\eta}(H)g_{\varepsilon}(\mathrm{w})\|_{\J^2}^2\leq
C\hbar^{-n}\eta\varepsilon.
\end{equation}
Let $f,g ,k \in C^\infty_c(\R)$ be supported in $[-\underline{\cst},\underline{\cst}]$. For every $\hbar\in (0,1]$, for every $\eta,\varepsilon\in [\hbar,1]$ with $\eta\varepsilon\geq \hbar$, 
\begin{align}
\label{eq:J2_dyadic_comm}
\|[\widetilde{f}_{\eta}(H),g_{\varepsilon}(\mathrm{w})]\|_{\J^2}^2&\leq
                                                          C\frac{\hbar^{n-2}}{\eta\varepsilon}\\
\label{eq:J2_dyadic_dbcomm}
\|[[\widetilde{f}_{\eta}(H),g_{\varepsilon}(\mathrm{w})],k_{\varepsilon}(\mathrm{w})]\|_{\J^2}^2&
                                                                               \leq
C\frac{\hbar^{n-4}}{\eta^3\varepsilon^3}.
\end{align}
\end{prop}

Observe that $\widetilde{f}_{\eta}$ is a smooth function in the class $S^\eta(\R)$, according to Definition~\ref{def:Sdelta}, then using Corollary~\ref{cor:propag}, one can replace (up to a negligible error) $\widetilde{f}_{\eta}(H)$ by a Fourier integral operator whose amplitude is in $S^1_{t,x,y,\xi} \times S^\eta_\lambda$. This allows us to apply the stationary phase method (Appendix~\ref{sec:stat-phase}) in order to prove the above Hilbert-Schmidt estimates. The proof of Proposition
\ref{prop:J2_estimates_dyadic}  occupies the rest of this section and it is organized as follows.
\begin{itemize}[leftmargin=*]
\item In Section~\ref{sec:hilbert-schmidt-norm}, we proceed with the
  estimate for
  $\|\widetilde{f}_{\eta}(H)g_{\varepsilon}(\mathrm{w})\|_{\J^2}^2$. This
  quantity is given (up to a negligible error) by the oscillatory integral
  \eqref{eq:HS_norm_product_dyadic_as_osc_int}. After an appropriate
  change of coordinates, we can simplify
  \eqref{eq:HS_norm_product_dyadic_as_osc_int} by a
  stationary phase argument. This reduces the problem to the
  computation of
  another oscillatory integral \eqref{comm_std}, with a mild
  amplitude, but with the \emph{classical phase} $(x,\xi)\mapsto
  x\cdot \xi$. 
This method is important and we will use variants of it throughout the rest of this article. Then, we can deduce a similar estimate for $\|f_{\eta}(H)g_{\varepsilon}(\mathrm{w})\|_{\J^2}^2$ by monotonicity. 

\item In Section~\ref{sec:commutators-at-small}, we adapt the method
  from Section~\ref{sec:hilbert-schmidt-norm} to estimate the
  Hilbert-Schmidt norms of the commutators \eqref{eq:J2_dyadic_comm}
  and\eqref{eq:J2_dyadic_dbcomm}, if the spectral scale $\eta$ is
  smaller than $\hbar^{\frac 12}$. Some difficulty comes from
  controlling the \emph{support of the amplitude} after a first
  application of
  stationary phase. 
Note that the proofs of \eqref{eq:J2_dyadic_comm} and \eqref{eq:J2_dyadic_dbcomm} are analogous. 

\item In Section~\ref{sec:commutators-at-large}, we prove
  \eqref{eq:J2_dyadic_comm} and \eqref{eq:J2_dyadic_dbcomm} in the
  complementary regime where the spectral scale $\eta$ is larger than $\hbar^{\frac 12}$. The proof relies on the fact that in this regime, $\widetilde{f}_{\eta}(H)$ is  (up to a negligible error) a pseudo-differential operator (Proposition~\ref{prop:pseudo}). 
Then, using again the stationary phase method in a more direct way, we complete the proofs of  \eqref{eq:J2_dyadic_comm}--\eqref{eq:J2_dyadic_dbcomm}. 
\end{itemize}

\subsection{Hilbert-Schmidt norm of products; proof of \eqref{eq:J2_dyadic_product}}
\label{sec:hilbert-schmidt-norm}

This subsection focuses on the proof of
\eqref{eq:J2_dyadic_product}.
Starting from Corollary \ref{cor:propag},  we have
\begin{equation} \label{eq:HS_norm_product_dyadic_as_osc_int} 
\begin{aligned} 
\|\widetilde{f}_{\eta}(H_{\hbar})g_{\varepsilon}(\mathrm{w})\|_{\J^2}^2=
\frac{1}{(2\pi\hbar)^{2n+2}}\int &e^{i\tfrac{\Psi_0(x,y,\xi_1,\xi_2,t_1,t_2,\lambda_1,\lambda_2)}{\hbar}}  B_0(x,y,\xi_1,\xi_2,t_1,t_2) \\
&\times f_{\eta}(\lambda_1) f_{\eta}(\lambda_2)  g_{\varepsilon}(\mathrm{w}(x))^2\dd t_1\dd t_2\dd \lambda_1\dd
\lambda_2\dd \xi_1\dd \xi_2\dd x \dd y +\O(\hbar^{\infty})
\end{aligned}
\end{equation}
where the error is independent of $(\eta,\varepsilon)$, the amplitude
$B_0$ belongs to $S^1(\R^{3n+2})$, and we have introduced the phase
\begin{equation} \label{phase0}
\Psi_0:(x,y,\xi_1,\xi_2,t_1,t_2,\lambda_1,\lambda_2)\mapsto
\varphi(t_1,x,\xi_1)-\varphi(t_2,x,\xi_2)-y\cdot(\xi_1-\xi_2)-t_1\lambda_1+t_2\lambda_2  . 
\end{equation}
The amplitude
\begin{equation} \label{symbol0}
B_0 : (x,y,\xi_1,\xi_2,t_1,t_2) \mapsto b(t_1,x,y,\xi_1)\overline{b(t_2,x,y,\xi_2)} \chi(x,y)
\end{equation}
where  $\chi$ is a smooth cutoff supported in $\Omega^{\prime 2}$.
We can introduce this cutoff since  $g_{\varepsilon}(\mathrm{w})$ is supported in $\{|\mathrm{w}|<\underline{\cst}\}$
and $b$ is supported in $  |x-y| \le \underline{\ell}$ with $\underline{\ell}\ll \underline{\cst}$. 
In particular, according to \eqref{supp_bulk}, the amplitude $B_0$ is supported in 
\[
\big\{ t_1,t_2\in [\underline{\tau},\underline{\tau}] , x,y \in \Omega', |x-y| \le \underline{\ell}, |\xi_1|,|\xi_2| \ge \underline{\cst} \big\}. 
\]
This allows to make the following change of coordinates in the integral \eqref{eq:HS_norm_product_dyadic_as_osc_int}:
\begin{equation} \label{change_coord}
\begin{cases}
t= \tfrac{t_1+t_2}{2} , & \lambda=\frac{\lambda_1+\lambda_2}{2} ,\qquad\xi=  \tfrac{\xi_1+\xi_2}{2}=r\omega \quad\text{with $r>0$, $\omega\in \S^{n-1}$} \\
s=t_1-t_2  , &  \sigma=\lambda_2-\lambda_1 , \quad 
\zeta= \xi_1-\xi_2
\end{cases}.
\end{equation}
We write 
\[
\Psi_0(x,y,\xi_1,\xi_2,t_1,t_2,\lambda_1,\lambda_2)
= \Psi_1(x,y,r,t,\lambda,\omega,\zeta,s) + \sigma t
\]
so that  $\Psi_1$ is independent of $\sigma$ and  
\begin{equation}  \label{eq:HS_norm_product} 
\begin{aligned}
\|\widetilde{f}_{\eta}(H_{\hbar})g_{\varepsilon}(w)\|_{\J^2}^2=
\frac{1}{(2\pi\hbar)^{2n+2}}\int &e^{i\tfrac{ \Psi_1(x,y,r,t,\lambda,\omega,\zeta,s) + \sigma t}{\hbar}}  B_1(x,y,\zeta,\omega,r,s,t)\\
&\times f_{\eta}(\lambda+\tfrac{\sigma}{2})f_{\eta}(\lambda-\tfrac{\sigma}{2}) g_{\varepsilon}(w(x))^2\dd y\dd\zeta  \dd r \dd s\dd \omega  \dd t \dd \sigma \dd \lambda \dd x +\O(\hbar^{\infty}).
\end{aligned}
\end{equation}
At this stage, we apply the stationary phase method to the previous
integral in the variable $(y,\zeta,r,s)$ keeping
$(t,\omega,\sigma,\lambda,x)$ fixed. The equations for the  critical point(s) are (using the notation \eqref{change_coord})
\[
\begin{cases}
\partial_y\Psi_1 =0 \\
\partial_\zeta\Psi_1 =0 \\
\partial_r\Psi_1 =0\\
\partial_s\Psi_1 =0
\end{cases}
\quad\Longleftrightarrow\qquad
\begin{cases}
\zeta=0 \\
y =\tfrac12 \big(\partial_{\xi}\varphi(t_1,x,\xi_1)+ \partial_{\xi}\varphi(t_2,x,\xi_2)\big) \\
0 =\tfrac12 \omega\cdot\big(\partial_{\xi}\varphi(t_1,x,\xi_1)- \partial_{\xi}\varphi(t_2,x,\xi_2)\big) \\
\lambda= \tfrac12 \big(\partial_t\varphi(t_1,x,\xi_1)+\partial_t\varphi(t_2,x,\xi_2)\big). 
\end{cases}
\] 
The first equation is equivalent to $\xi_1=\xi_2 =r \omega$ with $r\ge c_2$.  
Then, using that $\partial_{\xi}\varphi(t_1,x,\xi) = x + t\xi +\O(t^2)$ 
and  $\partial_{t}\varphi(t,x,\xi) = V(x) + |\xi|^2 +\O(t)$, these
equations reduce to (uniformly in $\omega,\sigma,\lambda,x$)
\[\begin{cases}
\zeta=0 \\
y = x+\O(rt)\\
0 = s r +\O(st)\\
\lambda=  \partial_t\varphi(t,x,r\omega)
\end{cases}
\quad\Longleftrightarrow\qquad
\begin{cases}
\zeta=0 \\
y = x+\O(rt)\\
s=0\\
r = \sqrt{\lambda-V(x)+\O(t)}.
\end{cases}
\]
Since we restrict our attention to $|t|,|s|<\underline{\tau}$, there
is a unique critical point $(y,\zeta,r,s) = (y_\star,0,r_\star,0)$. At the critical point,  $\Psi_1 =0 $ and 
\begin{equation} \label{critpt}
{\rm Hess} \Psi_1(x,t,\lambda,\omega) =
\begin{pmatrix} 0 &\mathrm{I} & 0  &0 \\\mathrm{I}& * & * & 0
\\ 0 & * &0 &2r_\star+\O(t) \\ 0 &0 & 2r_\star+\O(t) &0
\end{pmatrix}; \qquad 
\begin{cases} r_\star(x,t,\lambda,\omega) =  \sqrt{\lambda-V(x)+\O(t)} \\
y_\star(x,t,\lambda,\omega) = x+\O(t) 
\end{cases}.
\end{equation}
In particular $\det {\rm Hess} \Psi_1 = 4(\lambda-V(x))+\O(t)$ is
non-degenerate. As far as the phases are concerned, we are in position to apply stationary phase, and
$y_{\star},r_{\star}$ are smooth functions of $\omega,\sigma,\lambda,x$.

Since $B_1$ belongs to $S^1$, by Proposition~\ref{lem:statphase},
there exists a symbol $B_2\in S^1(\R^{2n+2})$ such that
\[
\frac{1}{(2\pi\hbar)^{n+1}}\int
e^{i\tfrac{\Psi_1(x,y,r,t,\lambda,\omega,\zeta,s)}{\hbar}}B_1(x,y,\zeta,\omega,r,s,t) \dd y\dd\zeta  \dd r \dd s
=B_2(x,\omega,\lambda,t) . 
\]
Let $\displaystyle B_3 :(x,\lambda,t) \mapsto \int_{\S^{n-1}}
B_2(x,\omega,\lambda,t)  \dd \omega$, so that $B_3\in
S^1(\R^{n+2})$. We have simplified
\eqref{eq:HS_norm_product_dyadic_as_osc_int} into
\begin{equation} \label{comm_std}
\|\widetilde{f}_{\eta}(H_{\hbar})g_{\varepsilon}(\mathrm{w})\|_{\J^2}^2=
\frac{1}{(2\pi\hbar)^{n+1}}\int e^{i\tfrac{\sigma t}{\hbar}}  B_3(x,\lambda,t) f_{\eta}(\lambda+\tfrac{\sigma}{2})f_{\eta}(\lambda-\tfrac{\sigma}{2}) g_{\varepsilon}(\mathrm{w}(x))^2  \dd t \dd \sigma \dd \lambda \dd x +\O(\hbar^{\infty}) .
\end{equation}
Now, we compute the   integral \eqref{comm_std} by applying the stationary phase method in the variables $(t,\sigma)$ keeping $(x,\lambda)$ as parameters. Note that the amplitude is in $S^1_t \times S^\eta_\sigma$ with $\eta \ge \hbar$ and it satisfies 
\[
\int   B_3(x, \lambda,t) f_{\eta}(\lambda+\tfrac{\sigma}{2})f_{\eta}(\lambda-\tfrac{\sigma}{2}) g_{\varepsilon}(\mathrm{w}(x))^2\dd \lambda\dd x
= \O_{S^1_t \times S^\eta_\sigma}\big(\eta\varepsilon\big)
\]
since $\|g_{\varepsilon}(\mathrm{w})\|_{L^2(\R^n)}^2=\O(\varepsilon)$
and $\|f_{\eta}\|_{L^2(\R)}^2=\O(\eta)$. Hence, by Corollary~\ref{corr:phase_stat_dyadic},  we conclude that 
\[
\|\widetilde{f}_{\eta}(H_{\hbar})g_{\varepsilon}(\mathrm{w})\|_{\J^2}^2 = \O\big(\hbar^{-n}\eta\varepsilon\big) . 
\]  

This proves the first estimate in \eqref{eq:J2_dyadic_product}.
To deduce the second estimate, we can assume that $f, g\ge 0$ and we choose the function  $\vartheta \in C^\infty_c(\R\to\R_+)$ so that $\widetilde{f}_{\delta} \ge f_\delta$. Then, as operators, for any $\varkappa : \R^n\to\R_+ $ bounded with compact support,
\[
\varkappa\widetilde{f}_{\eta}(H_{\hbar})^2\varkappa \ge  \varkappa f_{\eta}(H_{\hbar})^2\varkappa . 
\]
With $\varkappa = \sqrt{g_{\varepsilon}(\mathrm{w})}$,  taking traces on both sides, we conclude that  
\(
\|\widetilde{f}_{\eta}(H_{\hbar})g_{\varepsilon}(\mathrm{w})\|_{\J^2}^2 \ge \|f_{\eta}(H_{\hbar})g_{\varepsilon}(\mathrm{w})\|_{\J^2}^2.
\)
\qed 

\subsection{Hilbert-Schmidt norm of commutators; proof of \eqref{eq:J2_dyadic_comm}--\eqref{eq:J2_dyadic_dbcomm} at small spectral scales}
\label{sec:commutators-at-small}

This section is devoted to the proof of \eqref{eq:J2_dyadic_comm} and
\eqref{eq:J2_dyadic_dbcomm} 
under the assumption that the spectral scale $\eta$ is smaller than $\hbar^{\frac 12}$. 
The arguments follow the strategy from the proof of
\eqref{eq:J2_dyadic_product} in the previous section, but they rely
on the stationary phase with mild symbols
(Proposition~\ref{prop:stat_tempered}); the hypotheses will be
satisfied because $\varepsilon\eta\geq \hbar$ and in particular
$\varepsilon\geq \hbar^{\frac 12}$, since we have assumed that $\eta
\le \hbar^{\frac 12}$.

We begin with \eqref{eq:J2_dyadic_comm}. Introducing a dyadic sequence
of open sets
\[
  \mho_k : = \{  |\mathrm{w}| <  \varepsilon_k\}\text{ with }\varepsilon_k =2^k
  \varepsilon \text{ for }k\ge 0,\]
since without loss of generality $\supp(g) \subset [-1,1]$, one has $\supp(g_\varepsilon(\mathrm{w})) \subset \mho_0$.

As in
\eqref{eq:HS_norm_product_dyadic_as_osc_int}--\eqref{eq:HS_norm_product},
our starting point is the change of variables \eqref{change_coord}
leading to the expression
\begin{equation} \label{eq:HS_norm_comm_int} 
\begin{aligned}
\|[\widetilde{f}_{\eta}(H_{\hbar}),g_{\varepsilon}(\mathrm{w})]\|_{\J^2}^2=
\frac{1}{(2\pi\hbar)^{2n+2}}\int & e^{i\tfrac{ \Psi_1(x,y,r,t,\lambda,\omega,\zeta,s) + \sigma t}{\hbar}}  B_{1,\varepsilon}(x,y,\zeta,\omega,r,s,t)\\
&\times f_{\eta}(\lambda+\tfrac{\sigma}{2})f_{\eta}(\lambda-\tfrac{\sigma}{2})\dd y\dd\zeta  \dd r \dd s\dd \omega  \dd t \dd \sigma \dd \lambda \dd x +\O(\hbar^{\infty}) .
\end{aligned} 
\end{equation}
Here
\[
B_{1,\varepsilon}(x,y,\zeta,\omega,r,s,t) = B_0(x,y,\zeta,\omega,r,s,t)\big(g_\varepsilon(\mathrm{w}(x))-g_\varepsilon(\mathrm{w}(y))\big)^2
\]
The change of coordinates \eqref{change_coord} is still
well-defined, and as in the previous subsection, with respect to the
variables $(y,\zeta,r,s)$, there is a unique critical point
$(y_{\star},0,r_{\star},0)$, satisfying \eqref{critpt}. The amplitude $B_{1,\varepsilon}$ belongs to
$S^\varepsilon(\R^{4n+2})$ with $\varepsilon\geq \hbar^{\frac 12}$, so
we can apply the stationary phase method to the integral \eqref{eq:HS_norm_comm_int},
with respect to the variables $(y,\zeta,r,s)$, the other variables
being treated as parameters. All in all, by Proposition~\ref{prop:stat_tempered}, there exists a  symbol $B_{3,\varepsilon}\in S^\varepsilon(\R^{n+2})$ such that (integrating also over $\omega\in \S^{n-1}$)
\[
\frac{1}{(2\pi\hbar)^{n+1}}\int
e^{i\tfrac{\Psi_1(x,y,r,t,\lambda,\omega,\zeta,s)}{\hbar}}B_{1,\varepsilon}(x,y,\zeta,\omega,r,s,t) \dd y\dd\zeta  \dd r \dd s \dd \omega
=B_{3,\varepsilon}(x,\lambda,t) . 
\]
Moreover, this symbol has the following properties: $B_{3,\varepsilon}$ is supported in $\{x\in\Omega\}$, 
\begin{equation} \label{B2_exp}
B_{3,\varepsilon}(x,\lambda,0) = \O\big(\hbar\varepsilon^{-2} \big) 
\end{equation}
and, for any fixed $\ell \in \N$, there is a small constant $c>0$ so that for every $k\ge 2$, 
\begin{equation} \label{B2_supp}
B_{3,\varepsilon}(x,\lambda,t)=\O_{S^\varepsilon_t}\big(\big(\tfrac{\hbar}{\varepsilon\varepsilon_k}\big)^\ell\big), \qquad \text{for } x\in (\mho_k\setminus \mho_{k-1}),\,  |t|\le c \varepsilon_k ,\, |\lambda|\le c .
\end{equation}
Here, \eqref{B2_exp} follows directly from \eqref{eq:phase_stat} with
$\ell=1$, using that $L^0$ is a multiplication operator and
$y_\star|_{t=0} =x$. Then, \eqref{B2_supp} follows from
\eqref{eq:supp_stat}. Indeed, since
$ y_\star(x,t,\lambda,\omega) = x+\O(t)$ and we assume that
$\partial_x \mathrm{w} \neq 0$ on $\{\mathrm{w}=0\}$, we have for every $k\ge 2$, 
\[
\dist(y_\star(x,t,\lambda,\omega),\mho_0) \ge \varepsilon_k , \qquad \text{for } x\in (\mho_k\setminus \mho_{k-1}),\,  |t|\le c \varepsilon_k ,\, |\lambda|\le c ,\, \omega\in \S^{n-1} .
\]

By \eqref{eq:HS_norm_comm_int}, this computation implies that 
\begin{equation} \label{eq:HS_norm_comm} 
\|[\widetilde{f}_{\eta}(H_{\hbar}),g_{\varepsilon}(\mathrm{w})]\|_{\J^2}^2=
\frac{1}{(2\pi\hbar)^{n+1}}\int e^{i\tfrac{\sigma t}{\hbar}}  B_{3,\varepsilon}(x,\lambda,t) f_{\eta}(\lambda+\tfrac{\sigma}{2})f_{\eta}(\lambda-\tfrac{\sigma}{2})   \dd t \dd \sigma \dd \lambda \dd x +\O(\hbar^{\infty}) .
\end{equation}

Since $|\mho_1| \le C \varepsilon$ and $\|f_{\eta}\|_{L^2(\R)}^2=\O(\eta)$, we have 
\[
\int_{\{x\in \mho_1\}}  B_{3,\varepsilon}(x, \lambda,t) f_{\eta}(\lambda+\tfrac{\sigma}{2})f_{\eta}(\lambda-\tfrac{\sigma}{2}) \dd \lambda\dd x
= \O_{S^\varepsilon_t \times S^\eta_\sigma}\big(\eta\varepsilon\big) ,
\]
thus by Corollary~\ref{corr:phase_stat_dyadic} (with $\ell=2$),  
\begin{multline} \label{eq:HS_norm_comm_dyadic} 
\frac{1}{2\pi\hbar}
\int_{\{x\in \mho_1\}}  e^{i\tfrac{\sigma t}{\hbar}}  B_{3,\varepsilon}(x,\lambda,t) f_{\eta}(\lambda+\tfrac{\sigma}{2})f_{\eta}(\lambda-\tfrac{\sigma}{2})   \dd t \dd \sigma \dd \lambda \dd x 
\\= \int_{\{x\in \mho_1\}}  B_{3,\varepsilon}(x,\lambda,0) f_{\eta}(\lambda)^2 \dd x\dd \lambda+ \O\big(\tfrac{\hbar^{2}}{\eta\varepsilon}\big) . 
\end{multline}
Here, we used that the amplitude of the previous integral is in
$S^\varepsilon_t \times S^\eta_\sigma$ with $\varepsilon\eta\geq
\hbar$, that $B_{3,\varepsilon}$ is independent of $\sigma$, and that $ \sigma \mapsto  f_{\eta}(\lambda+\tfrac{\sigma}{2})f_{\eta}(\lambda-\tfrac{\sigma}{2})$ is a odd function. 
Then, by \eqref{B2_exp}, the previous integral is $ \O\big(\hbar\varepsilon^{-1}\eta\big) $. 
Since by assumption $ \eta^2\le \hbar$, we conclude that
\[
\int \1\{x\in \mho_1\}  e^{i\tfrac{\sigma t}{\hbar}}  B_{3,\varepsilon}(x,\lambda,t) f_{\eta}(\lambda+\tfrac{\sigma}{2})f_{\eta}(\lambda-\tfrac{\sigma}{2})   \dd t \dd \sigma \dd \lambda \dd x 
=  \O\big(\tfrac{\hbar^{3}}{\eta\varepsilon}\big) . 
\]

In addition, according to  \eqref{B2_supp}, by Corollary~\ref{corr:phase_stat_perturbation}, for  every $\ell\ge 2$, 
\[
\int \1\{x\in \mho_k\setminus \mho_{k-1}\} e^{i\tfrac{t\sigma}{\hbar}}B_{3,\varepsilon}(x,\lambda,t)  f_{\eta}(\lambda+\tfrac{\sigma}{2})f_{\eta}(\lambda-\tfrac{\sigma}{2}) \dd t\dd\sigma\dd \lambda\dd x 
=\O\big(\tfrac{\hbar^{\ell+2}}{\eta^\ell\varepsilon_k^\ell}\big) 
\]
using that $|\mho_k| \le C \varepsilon_k$ and $(\lambda,\sigma)
\mapsto
f_{\eta}(\lambda+\tfrac{\sigma}{2})f_{\eta}(\lambda-\tfrac{\sigma}{2})$
is in $S^\eta$ and supported in $\{|\lambda|,|\sigma| \le C
\eta\}$. Since $B_{3,\varepsilon}$ is supported in $\{x\in\Omega\}$
and $\sum_{k\ge 2} \varepsilon_k^{-\ell} = \O(\varepsilon^{-\ell})$, we deduce
\begin{equation*}
\int \1\{x\notin \mho_1\} e^{i\tfrac{t\sigma}{\hbar}} B_{3,\varepsilon}(x, \lambda,t) f_{\eta}(\lambda+\tfrac{\sigma}{2})f_{\eta}(\lambda-\tfrac{\sigma}{2}) \dd t\dd\sigma\dd \lambda\dd x 
=\O\big(\tfrac{\hbar^{\ell+2}}{\eta^\ell\varepsilon^\ell}\big) . 
\end{equation*}
Going back to \eqref{eq:HS_norm_comm}, by combing these estimates, we obtain 
\[
\|[\widetilde{f}_{\eta}(H_{\hbar}),g_{\varepsilon}(\mathrm{w})]\|_{\J^2}^2= \O\big(\tfrac{\hbar^{2-n}}{\eta\varepsilon}\big) . 
\]
This completes the proof of \eqref{eq:J2_dyadic_comm} in the case
where the spectral scale $\eta$ is smaller than $\hbar^{\frac 12}$. 

\smallskip

Now, we turn to the proof of \eqref{eq:J2_dyadic_dbcomm} in the same regime. 
The Hilbert-Schmidt norm of the double commutator $\|[[\widetilde{f}_{\eta}(H_{\hbar}),g_{\varepsilon}(\mathrm{w})],k_{\varepsilon}(\mathrm{w})]\|_{\J^2}^2$ is expressed as in \eqref{eq:HS_norm_comm_int} with, instead of $B_{1,\varepsilon}$, the symbol
\[
\widetilde{B}_{1,\varepsilon}(x,y,\zeta,\omega,r,s,t) = B_0(x,y,\zeta,\omega,r,s,t)\big(g_\varepsilon(\mathrm{w}(x))-g_\varepsilon(\mathrm{w}(y))\big)^2\big(k_\varepsilon(\mathrm{w}(x))-k_\varepsilon(\mathrm{w}(y))\big)^2 .
\]
We can again apply the same stationary phase method; the amplitude
$\widetilde{B}_{1,\varepsilon}$ belongs to $S^\varepsilon(\R^{4n+2})$ with $\varepsilon\geq \hbar^{\frac 12}$, so by Proposition~\ref{prop:stat_tempered}, there exists a  symbol $\widetilde{B}_{3,\varepsilon}\in S^\varepsilon(\R^{n+2})$ such that
\[
\|[[\widetilde{f}_{\eta}(H_{\hbar}),g_{\varepsilon}(\mathrm{w})],k_{\varepsilon}(\mathrm{w})]\|_{\J^2}^2
= \frac{1}{(2\pi\hbar)^{n+1}}\int e^{i\tfrac{\sigma t}{\hbar}}  \widetilde{B}_{3,\varepsilon}(x,\lambda,t) f_{\eta}(\lambda+\tfrac{\sigma}{2})f_{\eta}(\lambda-\tfrac{\sigma}{2})   \dd t \dd \sigma \dd \lambda \dd x +\O(\hbar^{\infty}), 
\]
and $\widetilde{B}_{3,\varepsilon}$ also satisfies \eqref{B2_exp}--\eqref{B2_supp}. 
In particular, using \eqref{B2_supp} and Corollary~\ref{corr:phase_stat_perturbation} as above, for any fixed $\ell\in\N$,
\[
\int \1\{x\notin \mho_1\} e^{i\tfrac{t\sigma}{\hbar}} \widetilde{B}_{3,\varepsilon}(x, \lambda,t) f_{\eta}(\lambda+\tfrac{\sigma}{2})f_{\eta}(\lambda-\tfrac{\sigma}{2}) \dd t\dd\sigma\dd \lambda\dd x 
=\O\big(\tfrac{\hbar^{\ell+2}}{\eta^\ell\varepsilon^\ell}\big) . 
\]

In fact, $\widetilde{B}_{3,\varepsilon}$ satisfies stronger estimates than \eqref{B2_exp}. 
Since $y_\star|_{t=0} =x$, we have for any smooth differential operator $L=L_{(y,\zeta,r,s)}$ of degree 2, 
\[
L\widetilde{B}_{1,\varepsilon}(x,y,\zeta,\omega,r,s,t) \big|_{(y,\zeta,r,s)=(y_\star,0,r_\star,0),\, t=0} =0
\]
so that, according to \eqref{eq:phase_stat} with $\ell=2$,
\[
B_{3,\varepsilon}(x,\lambda,0) = \O\big(\hbar^2\varepsilon^{-4}\big) .
\]
Similarly, by  \eqref{eq:phase_stat} with $\ell=1$,
\[
\partial_t^2 B_{3,\varepsilon}(x,\lambda,t)\big|_{t=0} = \O\big(\hbar\varepsilon^{-4}  \big) . 
\]

Hence, instead of \eqref{eq:HS_norm_comm_dyadic}, we apply
Corollary~\ref{corr:phase_stat_dyadic} (with $\ell=4$),  and obtain
\[\begin{aligned}
\frac{1}{2\pi\hbar}
&\int \1\{x\in \mho_1\}  e^{i\tfrac{\sigma t}{\hbar}}  \widetilde{B}_{3,\varepsilon}(x,\lambda,t) f_{\eta}(\lambda+\tfrac{\sigma}{2})f_{\eta}(\lambda-\tfrac{\sigma}{2})   \dd t \dd \sigma \dd \lambda \dd x \\ 
&= \int \1\{x\in \mho_1\} \Big( B_{3,\varepsilon}(x,\lambda,0) f_{\eta}(\lambda)^2 + 
\big(\tfrac{\hbar}{2\eta}\big)^2\partial_t^2 B_{3,\varepsilon}(x,\lambda,0)
\big(f_{\eta}'(\lambda)^2 -f_{\eta}''(\lambda)f_{\eta}(\lambda)\big)\Big)
\dd x\dd \lambda+ \O\big(\tfrac{\hbar^{4}}{\eta^3\varepsilon^3}\big) . 
\end{aligned}\]
Since $ \eta^2\le \hbar$ in this regime and $\|f^{(j)}_\eta\|_{L^2(\R)}^2 = \O_j(\eta)$ for $j\in\N$, we conclude that
\[
\int \1\{x\in \mho_1\}  e^{i\tfrac{\sigma t}{\hbar}}  \widetilde{B}_{3,\varepsilon}(x,\lambda,t) f_{\eta}(\lambda+\tfrac{\sigma}{2})f_{\eta}(\lambda-\tfrac{\sigma}{2})   \dd t \dd \sigma \dd \lambda \dd x 
=\O\big(\tfrac{\hbar^{5}}{\eta^3\varepsilon^3}\big) .
\]
By combining the previous estimates, we conclude that 
\[
\|[[\widetilde{f}_{\eta}(H_{\hbar}),g_{\varepsilon}(\mathrm{w})],k_{\varepsilon}(\mathrm{w})]\|_{\J^2}^2 = \O\big(\tfrac{\hbar^{4}}{\eta^3\varepsilon^3}\big) . 
\]
This completes the proof of \eqref{eq:J2_dyadic_dbcomm} in the case
where the spectral scale $\eta$ is smaller than $\hbar^{\frac 12}$. 

\subsection{Hilbert-Schmidt norm of commutators; proof of \eqref{eq:J2_dyadic_comm}--\eqref{eq:J2_dyadic_dbcomm} at large spectral scales}
\label{sec:commutators-at-large}

This section is concerned with the proof of \eqref{eq:J2_dyadic_comm}
and \eqref{eq:J2_dyadic_dbcomm}
under the assumption that $\eta\geq \hbar^{\frac 12}$. In this regime,
$f_{\eta}(H_{\hbar})$ is a pseudodifferential operator with
symbol in $S^{\eta}$ and small support.

\begin{prop} \label{prop:pseudo}
Let  $\mathcal{H}(x,\xi) := V(x)+|\xi|^2$ for $(x,\xi)\in\R^{2n}$ and
$\chi :\R^n\to[0,1]$ be a smooth function with compact support, and
equal to 1 on the ball of radius $\underline{\ell}$.
Let $\widetilde{f}_\eta$ be as in Proposition \ref{prop:J2_estimates_dyadic} for $\eta\in [\sqrt\hbar, 1]$. 
Then there exists a symbol $p_{\eta} \in S^{\eta}(\R^{2n})$ such that 
\[
\widetilde{f}_\eta(H_{\hbar}) : (x,y) \mapsto \frac{\chi(x-y)}{(2\pi\hbar)^{n}}\int
e^{i\tfrac{(x-y)\cdot \xi}{\hbar}}p_{\eta}(x,\xi)\dd \xi  +\O_{\J^1}(\hbar^{\infty})
\]
and for any  $\delta\in[\eta,1]$, 
\begin{equation}\label{Lyap2}
p_{\eta}  =\O_{S^\eta}\big(\big(\tfrac{\hbar}{\eta\delta}\big)^{\infty}\big)  
\qquad\text{uniformly on }\{|\mathcal{H}(x,\xi)|> \mathrm{C} \delta\} .
\end{equation}
In particular, $\|p_{\eta}\|_{L^1_\xi\times L^\infty_x} = \O(\eta)$
. \end{prop}

\begin{proof}
Recall that according to Proposition \ref{prop:propag_as_FIO},
\[
\widetilde{f}_\eta(H_{\hbar}) : (x,y)\mapsto\frac{1}{(2\pi\hbar)^{n+1}}\int
e^{i\tfrac{\varphi(t,x,\xi)-y\cdot\xi-t\lambda}{\hbar}}a(x,y,\xi,t)f_\eta(\lambda)\dd t \dd \lambda \dd \xi+O_{\J^1}(\hbar^{\infty}).
\]
We define
\begin{equation} \label{pseudob}
b : (x,y,\xi) \mapsto \frac{1}{2\pi\hbar}\int
e^{i\tfrac{\Phi(t,\lambda,x,\xi)}{\hbar}}a(x,y,\xi,t)f_\eta(\lambda)\dd t \dd \lambda , \qquad 
\Phi(t,\lambda,x,\xi):=\varphi(t,x,\xi)-x\cdot\xi-t\lambda.
\end{equation}
This oscillatory integral satisfies the hypotheses of Proposition
\ref{prop:stat_tempered}: the amplitude belongs to $S^{\eta}_{t,\lambda}(\R^2)$, the phase
$\Phi$ has a unique critical point $(t,\lambda)=(0,V(x)+|\xi|^2)$, and
the Hessian at this point is non-degenerate. Moreover, the amplitude
is in $S^1_y$, the phase is independent of $y$. Consequently, there is a symbol $b\in
S^{\eta}_{x,\xi} \times S^1_y$ such that 
\[
\text{$b$ satisfies  \eqref{Lyap} with }\mathcal{A}=\{(x,\xi)\in \R^{2n}, f_{\eta}(V(x)+|\xi|^2)\neq 0\} ,
\]
and we have 
\[
\widetilde{f}_\eta(H_{\hbar})(x,y)=\frac{1}{(2\pi\hbar)^{n}}\int
e^{i\frac{( x-y)\cdot \xi}{\hbar}}b(x,y,\xi)\dd \xi+O_{\J^1}(\hbar^{\infty}) .
\]
Observe also that since $a$ is supported in a neighborhood of size $\underline{\ell}$
of the diagonal $\{x=y\}$, we can freely include the cutoff
$\chi(x-y)$ in the previous kernel. Now, it remains to eliminate the
fact that $b$ depends on $y$. To this end, let $\varkappa : \R^n\to[0,1]$ be a smooth cutoff, with a fixed compact support, such that $\varkappa(\xi)=1$ on a neighborhood of $\mathcal B := \bigcup_{(x,y) \in\R^{2n}}\supp\big(\xi\mapsto b(x,y,\xi)\big)$, and introduce 
\[
p_{\eta} : (x,\xi) \mapsto \frac{\varkappa(\xi)}{(2\pi\hbar)^n} \int e^{-i\tfrac{u\cdot \zeta}{\hbar}} b(x,x+u,\xi+\zeta)\, \dd u \dd\zeta .
\]
Note that we introduce the cutoff $\varkappa$ simply so that $p_\eta$
has a fixed compact support. Then, applying again Proposition
\ref{prop:stat_tempered}, $p_{\eta}\in S^{\eta}(\R^{2n})$ and, according to Remark~\ref{rk:decay}, $p_{\eta}$  satisfies \eqref{Lyap2}. 
Finally, by Proposition~\ref{prop:stat_tempered_classic} (the saddle point is $(z,\xi)=(y,\zeta)$), since $b(x,z,\zeta) \varkappa(\xi)$ is in $S^{0}_{z,\xi} \times S^{\eta}_{x,\xi}$, we obtain for any $y\in \R^n$,
\[\begin{aligned}
\int e^{i\frac{(x-y)\cdot
\xi}{\hbar}}p_{\eta}(x,\xi)\dd \xi 
&= \frac{1}{(2\pi\hbar)^{n}} \int e^{i\frac{(x-y)\cdot\xi+(x-z)\cdot(\zeta-\xi)}{\hbar}}b(x,z,\zeta) \varkappa(\xi)\dd z\dd \xi \dd \zeta \\
&= \int e^{i\frac{(x-y)\cdot\zeta}{\hbar}} \big\{ b(x,y,\zeta) + \O_{S^\eta_{x,\zeta}}(\hbar^\infty)\big\} \dd\zeta
\end{aligned}\]
where we used that $\varkappa(\xi)=1$ on $\mathcal B$, so that $\partial_\xi^k \varkappa(\xi)=0$ for every $k\ge 1$ on the support of $b$. 
In particular, the error (times the cutoff $\chi(x-y)$) is a kernel in $\O_{S^\eta_{x,y}}(\hbar^\infty)$. Using the estimate \eqref{DS}, this corresponds to an operator $\O_{\J^1}(\hbar^\infty)$ as claimed.

Finally, to estimate $\|p_{\eta}\|_{L^1_\xi\times L^\infty_x} = \O(\eta)$, we observe that by assumption, $\partial_{x,\xi} \mathcal H \neq 0$ for $x \in \{\mathcal H = 0\}$, so that $\big| \xi \in\R^n : |\mathcal{H}(x,\xi)| \le \mathrm{C} \delta  \big| = \O(\delta)$ for any $0<\delta \le \underline{\cst}$ for some small constant $\underline{\cst}$. 
Then, by considering dyadic scales $\delta_k = \eta 2^k$ for $k\ge 0$, we deduce from the estimate \eqref{Lyap2} that uniformly for $x\in\R^n$, 
\[
\int |p_{\eta}(x,\xi)|\dd \xi =\O \bigg( \eta+  \sum_{k\ge 1} \frac{\hbar^2}{\eta^2\delta_k} \bigg) = \O(\eta)
\]
since $\hbar^2 \eta^{-4} \le 1$.
This concludes the proof. 
\end{proof} 

From Proposition \ref{prop:pseudo}, by applying the stationary phase
for mild symbols (in the form of
Proposition~\ref{prop:stat_tempered_classic}), one can prove the
relevant commutator estimates when the spectral scale $\eta$ is larger
than $\hbar^{\frac 12}$, in which case we only need to assume that
$\varepsilon\eta \ge \hbar$. We focus on \eqref{eq:J2_dyadic_comm}.

Since $f_{\eta}(H_{\hbar})$ and $g_{\epsilon}(\mathrm{w})$ are
self-adjoint, from the expression of the kernel of $f_{\eta}(H_{\hbar})$ in Proposition
\ref{prop:pseudo}, we obtain
\begin{multline}
\|[f_{\eta}(H_{\hbar}),g_{\varepsilon}(\mathrm{w})]\|_{\J^2}^2\\=\frac{1}{(2\pi\hbar)^{2n}}\int
e^{i\frac{(x-y)\cdot(\xi-\zeta)}{\hbar}}p_{\eta}(x,\xi)p_{\eta}(y,\zeta)\chi(x-y)^2(g_\varepsilon(\mathrm{w}(x))-g_{\varepsilon}(\mathrm{w}(y)))^2\dd
\zeta\dd y \dd x \dd \xi +\O(\hbar^{\infty}). 
\end{multline}

We split this integral in two parts by introducing a cutoff. Let $\chi_1:\R\to[0,1]$ be a smooth compactly supported function, let $\chi_\varepsilon : = \chi(\varepsilon^{-1}\mathrm{w})$ and $ \chi^\dagger_\varepsilon := 1-\chi_\varepsilon$ be such that $ \chi^\dagger_\varepsilon \cdot g_{\varepsilon}(\mathrm{w})=0$. 
We have 
\begin{multline}\label{J2comm1}
\|[f_{\eta}(H_{\hbar}),g_{\varepsilon}(\mathrm{w})]\|_{\J^2}^2=\frac{1}{(2\pi\hbar)^{2n}}\int
p_{\eta}(x,\xi) \bigg( \int e^{i\frac{(x-y)\cdot(\xi-\zeta)}{\hbar}} p_{\eta}(y,\zeta)Q(x,y)\dd
\zeta\dd y \bigg) \dd x \dd \xi \\
+ \frac{1}{(2\pi\hbar)^{n}}\int  p_{\eta}(y,\zeta)  g_{\varepsilon}(\mathrm{w}(y))^2 R(y,\zeta) \dd
\zeta\dd y +\O(\hbar^{\infty}) 
\end{multline} 
where
\begin{align*}
  Q : (x,y)&\mapsto \chi(x-y)^2\chi_\varepsilon(x)
             (g_\varepsilon(\mathrm{w}(x))-g_{\varepsilon}(\mathrm{w}(y)))^2\\
  R: (y,\zeta) &\mapsto  \frac{1}{(2\pi\hbar)^{n}} \int
e^{i\frac{(x-y)\cdot(\xi-\zeta)}{\hbar}}\chi(x-y)^2p_{\eta}(x,\xi)
  \chi_\varepsilon^\dagger(x) \dd x \dd \xi.
\end{align*}

The relevant property is that, with $\ell =2$,
\begin{equation} \label{Qondition}
(y,\zeta) \mapsto \varepsilon^{-1} p_{\eta}(y,\zeta) \int Q (x,y) \dd x
\text{ is in $S^{\varepsilon}_y \times S^{\eta}_\zeta$  with $\varepsilon\eta \ge \hbar$  and $\partial_y^k Q(x,y)|_{y=x} =0$ for $k<\ell$}.
\end{equation}
Then, applying Corollary~\ref{corr:phase_stat_dyadic} with $\ell=2$,  we obtain (uniformly for $\xi \in\R^{n}$)
\[
\frac{1}{(2\pi\hbar)^{n}}\int e^{i\frac{(x-y)\cdot(\xi-\zeta)}{\hbar}} p_{\eta}(y,\zeta)Q(x,y)\dd \zeta\dd y \dd x= \O\big(\tfrac{\hbar^2}{\eta^2\varepsilon}\big) .
\]
Since
\[
  \sup_{x\in \R^n}\int |p_{\eta}(x,\xi)|\dd \xi = \O(\eta),
\]this implies that 
$\eqref{J2comm1}= \O\big(\tfrac{\hbar^{2-n}}{\eta\varepsilon}\big)$. 

Then, the amplitude of the integral $R$ is in $S^{\delta}_x \times S^{\eta}_\xi$ with $\delta=\min\{\varepsilon,\eta\}$, so applying either Proposition~\ref{prop:stat_tempered_classic}  in case $\delta= \varepsilon$ or Lemma~\ref{prop:nonstatphase} in case $\delta=\eta$ with $\eta\geq \hbar^{\frac 12}$ (here, $\chi_\varepsilon^\dagger(x) =0$ for $x\in B(y,\varepsilon)$ if $y\in \{g_{\varepsilon}(\mathrm{w}) \neq 0\}$),we obtain in both cases 
\[
R(y,\zeta)
= \O\big(\big(\tfrac{\hbar}{\eta\varepsilon}\big)^{\infty}\big) 
\]
with the required uniformity. 
Thus, this implies that the second line of \eqref{J2comm1} is $\O\big(\big(\tfrac{\hbar}{\eta\varepsilon}\big)^{\infty}\hbar^{-n}\eta\varepsilon\big)$. Combining these estimates, this concludes the proof of \eqref{eq:J2_dyadic_comm}. 

Now, we turn to the estimate \eqref{eq:J2_dyadic_dbcomm}; the argument
being the same. Using Proposition \ref{prop:pseudo}, we also have
\begin{multline}\label{J2comm2}
\|[[f_{\eta}(H_{\hbar}),g_{\varepsilon}(\mathrm{w})],k_{\varepsilon}(\mathrm{w})]\|_{\J^2}^2=\frac{1}{(2\pi\hbar)^{2n}}\int
p_{\eta}(x,\xi) \bigg( \int e^{i\frac{(x-y)\cdot(\xi-\zeta)}{\hbar}} p_{\eta}(y,\zeta)\widetilde{Q}(x,y)\dd
\zeta\dd y \bigg) \dd x \dd \xi \\
+ \frac{1}{(2\pi\hbar)^{n}}\int  p_{\eta}(y,\zeta)  g_{\varepsilon}(\mathrm{w}(y))^2 k_{\varepsilon}(\mathrm{w}(y))^2\widetilde{R}(y,\zeta) \dd
\zeta\dd y +\O(\hbar^{\infty}) 
\end{multline} 
where
\begin{align*}
  \widetilde{Q} : (x,y)&\mapsto \chi(x-y)^2\chi_\varepsilon(x)
             (g_\varepsilon(\mathrm{w}(x))-g_{\varepsilon}(\mathrm{w}(y)))^2(k_\varepsilon(\mathrm{w}(x))-k_{\varepsilon}(\mathrm{w}(y)))^2\\
  \widetilde{R}: (y,\zeta) &\mapsto  \frac{1}{(2\pi\hbar)^{n}} \int
e^{i\frac{(x-y)\cdot(\xi-\zeta)}{\hbar}}\chi(x-y)^2p_{\eta}(x,\xi)
  \chi_\varepsilon^\dagger(x) \dd x \dd \xi.
\end{align*}
Now $\widetilde{Q}$ satisfies the condition \eqref{Qondition} with $\ell=4$. 
Thus, the first term of the right-hand-side of \eqref{J2comm2}
is $\O\big(\tfrac{\hbar^{4-n}}{\eta^3\varepsilon^3}\big)$ and the integral involving $\widetilde{R}$ is again  $\O\big(\big(\tfrac{\hbar}{\eta\varepsilon}\big)^{\infty}\hbar^{-n}\eta\varepsilon\big)$.
This yields the estimate \eqref{eq:J2_dyadic_dbcomm}.

\section{Hilbert-Schimdt norm of commutators: Proof of Theorem \ref{thm:J2}.}
\label{sec:comm}

Throughout this section, we use the notation from Section~\ref{sec:not}.
In particular, $\Omega \Subset \mathcal{D}$ is a fixed relatively
compact open subset of the bulk with smooth boundary,
$f\in C^{\infty}(\R^n)$, and we denote $f|_\Omega =f\1_\Omega$ (this function is also  viewed as a bounded operator on $L^2(\R^n)$).
Using Proposition~\ref{prop:regularised_commutator}, we focus on analysing the commutator $[\widetilde{\Pi}, f|_\Omega] $ where the regularized kernel $\widetilde{\Pi}$ is given by the oscillatory integral \eqref{kern_reg}.
Then, with the phase \eqref{phase0}, one has 
\begin{equation} \label{comm1}
\begin{aligned}
\big\| [\widetilde{\Pi}, f|_\Omega] \big\|^2_{\J^2} = \tr\big( [\widetilde{\Pi},f|_\Omega]  [f|_\Omega,\widetilde{\Pi}^*] \big)
=&\frac{1}{(2\pi\hbar)^{2n+2}}\int
e^{i\frac{\Psi_0(x,y,\xi_1,\xi_2,t_1,t_2,\lambda_1,\lambda_2)}{\hbar}}A_0(x,y,\xi_1,\xi_2,t_1,t_2,\lambda_1,\lambda_2) \\
&\qquad(f|_\Omega(x)-f|_\Omega(y))^2  \1_{\lambda_1\leq
\mu}\1_{\lambda_2\leq \mu}\dd t_1\dd t_2\dd \lambda_1\dd
\lambda_2\dd \xi_1\dd \xi_2\dd x \dd y , 
\end{aligned}
\end{equation}
where the amplitude $A_0 $ belongs to $S^1(\R^{6n+4})$ and its principal part at $(t_1,t_2,y) = (0,0,x)$ is given by  
\begin{equation} \label{symbolA}
A_0(x,x,\xi_1,\xi_2,0,0,\lambda_1,\lambda_2)|_{\hbar=0}=  \vartheta(V(x)+|\xi_1|^2)  \vartheta(V(x)+|\xi_2|^2) \chi(\lambda_1)\chi(\lambda_2) .
\end{equation}
Moreover, since $\Omega \Subset \{V<0\}$, the function $(x,y) \mapsto (f|_\Omega(x)-f|_\Omega(y))^2 \1\{|x-y| \le \underline{\ell}\}$ is supported in $\Omega'\times\Omega'$ where $\Omega' \subset \{V<-2\underline{\cst}\}$ is a neighbourhood of $\Omega$ in the bulk. Consequently, by \eqref{supp_bulk}, $A_0$ is supported in 
\begin{equation}\label{suppA0}
\big\{t\in[-\underline{\tau},\underline{\tau}], (x,y)\in \Omega'\times\Omega',  |x-y|\le \underline{\ell},  |\xi_1|, |\xi_2|\ge \underline{\cst}, |\lambda_1|,|\lambda_2| \le \underline{\ell}\} .
\end{equation}
for small constants $\underline{\tau}\ll\underline{\ell}\ll \underline{\cst}$.

\smallskip

The proof consists of the following steps.
\begin{itemize}[leftmargin=*]
\item In Section~\ref{sec:est}, we gather some preliminary estimates for singular integrals with discontinuous amplitudes. 
\item In Section~\ref{sec:red}, using the stationary phase method, we perform a series of reductions to write the oscillatory integral \eqref{comm1}  as
\begin{equation} \label{red}
\| [\widetilde{\Pi}, f|_\Omega] \|^2_{\J^2} =  \frac{-1}{(2\pi\hbar)^{2n}} \int
e^{i\frac{(x-y)\cdot\xi}{\hbar}} F(x,y,\xi)(f|_\Omega(x)-f|_\Omega(y))^2\dd \xi \dd x\dd y +  \O(\hbar^{1-n}) .
\end{equation}
where the function $F$ has compact support, is smooth on $\R^3\setminus\{\xi=0\}$ and $|F(x,y,\xi)| \le C|\xi|$. 
\item In Section~\ref{sec:int}, we study integrals of the form
  \eqref{red} in the particular case where the set $\Omega$ is $C^\infty$-diffeomorphic to
  the ball $B = \{x\in\R^n : |x|< 1\}$.
After a change of variables and further reduction steps, we show that  
\[
\eqref{red} = \frac{1}{(2\pi\hbar)^{2n}} \int e^{i\frac{(x-y)\cdot\xi}{\hbar}} g(\tfrac{x+y}{2},\xi) \1_{B}(x)\1_{B}(y)\dd \xi \dd x\dd y +  \O(\hbar^{1-n}) ,
\]
where $g$ is smooth on $
\R^{2n}\setminus\{\xi=0\}$  and  $g(v,\xi) = |\xi| R(v) +\O(|\xi|^2)$ as $\xi\to0$.
Using the asymptotics of the Fourier transform $\widehat\1_B$, which is related to Bessel functions, we obtain that the leading term of \eqref{red} is given by
\(
(2\pi\hbar)^{1-n} \frac{\log \hbar^{-1}}{\pi^2}  \int_{\partial B} R(\hat x) f(\hat x)^2 \dd \hat x .
\)

\item Finally, in Section \ref{sec:partition-unity-end}, we use a partition of unity type argument to conclude the proof of
Theorem \ref{thm:J2} without the topological assumption.

\item In Section \ref{sec:clt}, we prove Theorem \ref{thm:clt}. The Gaussian asymptotic fluctuations for counting statistics is a direct consequence of Theorem \ref{thm:J2}  (the variance of $\X(\Omega)$ diverges for any smooth set $\Omega$) and we only need to estimate covariances. 
Using the off-diagonal decay of the regularized kernel
(Proposition~\ref{prop:kerdiag}), we argue that for two smooth sets
$\Omega_1, \Omega_2 \Subset \mathcal D$ such that $|\partial
\Omega_1\cap \partial \Omega_2|=0$, the covariance satisfies $\tr\big( [\Pi, \1_{\Omega_1}]  [\1_{\Omega_2},\Pi] \big) = o(\log(\hbar^{-1})\hbar^{1-n})$ as $\hbar\to0$.
\end{itemize}

\subsection{Preliminary estimates}
\label{sec:est}

In this section we gather several estimates that will be important for
the proof of Theorem \ref{thm:clt}.

\begin{lem} \label{lem:key_estimate}
There exists a constant $C_n$ such that the following is true: let
$\Omega \Subset \R^n$ be a bounded open set with a smooth
boundary. Denote by $|\partial \Omega|$ the $(n-1)$-Hausdorff measure of
the smooth, compact hypersurface $\partial \Omega$. Then for every
$\hbar>0$ sufficiently small,
\[
\int_{|x-y|\leq \hbar}\frac{|\1_\Omega(x) - \1_\Omega(y)|^2}{2} \dd x \dd y 
=  \int_{|x-y|\leq \hbar} \1_\Omega(x)\1_{\Omega^c}(y) \dd x \dd y 
\sim C_n|\partial \Omega|\hbar^{n+1} . 
\]
For any $\varkappa : \R^{2n}\to [0,1]$ with compact support, 
\begin{equation} \label{bdsub}
\int\frac{|\1_\Omega(x) - \1_\Omega(y)|^2}{2|x-y|^n} \varkappa(x,y) \dd x \dd y 
=  \int\frac{\1_\Omega(x)\1_{\Omega^c}(y)}{|x-y|^n} \varkappa(x,y)\dd x \dd y 
<\infty .
\end{equation}
Moreover, given $\overline\chi :\R^n \to [0,1]$ smooth such that $\overline\chi(x)=0$ for $|x| \le \underline{\cst}$, there exists a constant $C$ so that for any $\hbar>0$,
\begin{equation} \label{bdsup}
\int  \overline\chi(\tfrac{x-y}{\hbar})  \frac{|\1_\Omega(x) - \1_\Omega(y)|^2}{2|x-y|^{n+2}}  \dd x \dd y 
=  \int  \overline\chi(\tfrac{x-y}{\hbar})  \frac{\1_\Omega(x)\1_{\Omega^c}(y)}{|x-y|^{n+2}} \dd x \dd y 
\le C/\hbar . 
\end{equation}
\end{lem}

\begin{proof}
Let 
\[
\Omega_{\hbar}=\{x \in \Omega : \dist(x,\partial
\Omega)<\hbar\} 
\]
be an $\hbar$-neighborhood of $\partial\Omega$ inside $\Omega$.
Since $\partial \Omega$ is smooth and compact, if $\hbar$ is small enough, then the following change of variables is well defined 
\[\begin{aligned}
\Omega_{\hbar} & \mapsto  \partial\Omega \times (0,1) \\
x&\mapsto  (\hat{x},t) =  \big(\underset{q\in \partial \Omega}{\arg\min}|q-x|^2, 
\hbar^{-1}  \dist(x, \partial \Omega)\big)  .
\end{aligned}\]
This corresponds to the (orthogonal) projection of $x \in  \Omega_{\hbar}$ onto 
$ \partial\Omega$ and the volume form $\dd x \sim  \hbar \dd t \dd \hat{x} $ as $\hbar\to0$ where $\dd \hat{x}$ denotes the volume measure on $\partial\Omega$.
Hence, as $\hbar\to0$
\[
\int \1_{\Omega}(x) \1_{|x-y|\le \hbar}\1_{\Omega^c}(y) \dd x \dd y 
\sim \hbar\int_{\partial \Omega \times (0,1)}  \bigg(\int_{\Omega^c} \1_{|\hat{x} - \hbar t \nu_\Omega(\hat{x}) -y| \le \hbar}  \dd y \bigg) \dd \hat{x}  \dd t 
\]
where $\nu_\Omega(\hat{x})$ denotes the (exterior) unit normal vector at $\hat{x} \in\Omega$.

For a fixed $\hat{x} \in \partial\Omega$ and $t>0$, 
\[
\int_{\Omega^c} \1_{|\hat{x} - \hbar t \nu_\Omega -y| \le \hbar}  \dd y
\sim \hbar^n\int_{\{y \cdot  \mathrm{e}_1 \ge 0 \}} \1_{|y+t \mathrm{e}_1|\le 1 } \dd y
\]
which follows by rescaling and using normal coordinates around $\hat{x}$.  
This shows that there exists a constant $C_n$ so that as $\hbar\to 0$,

\[
\int_{|x-y|\leq \hbar} \1_\Omega(x)\1_{\Omega^c}(y) \dd x \dd y  \sim C_n \hbar^{n+1}\int_{\partial \Omega} \dd\hat{x} .
\]

To obtain \eqref{bdsub}, observe that for $s\in\R_+$ 
\[
\int \frac{\1_\Omega(x)\1_{\Omega^c}(y)}{|x-y|^s}  \varkappa(x,y) \dd x \dd y
\le \sum_{k\in\Z : 2^k \le C/\hbar} (\hbar 2^k)^{-s} \int_{|x-y|\leq \hbar 2^{k+1}} \hspace{-.3cm}\1_\Omega(x)\1_{\Omega^c}(y) \dd x \dd y  
\]
where the constant $C$ depends only on $\varkappa$. 
The pervious argument shows that the integrals on the RHS are  $\O\big((\hbar 2^k)^{n+1}\big)$  
so that 
\[
\int\frac{\1_\Omega(x)\1_{\Omega^c}(y)}{|x-y|^s} \varkappa(x,y) \dd x \dd y
\le  C \hbar^{n+1-s}  \sum_{k\in\Z : 2^k\le C/\hbar}  2^{k(n+1-s)} =O(1)
\]
provided that $s<n+1$; in which case the sum is geometrically growing.
On the other hand, for $s>n+1$, upon adding a different cutoff to exclude that diagonal, we obtain
\[
\int  \overline\chi(\tfrac{x-y}{\hbar})  \frac{\1_\Omega(x) \1_{\Omega^c}(y)}{|x-y|^{s}}  \dd x \dd y \le \sum_{k\in\N_0} (\hbar 2^k)^{-s} \int_{|x-y|\leq \hbar 2^{k+1}} \hspace{-.3cm}\1_\Omega(x)\1_{\Omega^c}(y) \dd x \dd y  \le  C \hbar^{n+1-s} 
\]
since the series is convergent. This completes the proof of \eqref{bdsup}. 
\end{proof}

Lemma~\ref{lem:key_estimate} has the two following consequences.

\begin{prop} \label{prop:key}
Let $\Omega \Subset \R^n$ be a relatively compact open set with a smooth boundary, $f\in C^\infty(\R^n)$ and recall that $f|_\Omega= f \1_\Omega$.
For any $\varkappa : \R^{2n}\to [0,1]$ with compact support,
\[
\int\frac{|f|_\Omega(x) - f|_\Omega(y)|^2}{2|x-y|^n} \varkappa(x,y) \dd x \dd y 
<\infty .
\]
Moreover, for any cutoff $\chi \in C^\infty_c$,
\begin{equation} \label{bdsubmod}
\int  \chi(\tfrac{x-y}{\hbar}) (f|_\Omega(x)-f|_\Omega(y))^2  \dd x \dd y = \O(\hbar^{n+1}) .
\end{equation}
\end{prop}

\begin{proof}
We have
\[
|f|_\Omega(x) - f|_\Omega(y)|^2 = 
\big(f(x)-f(y)\big)\big(f|_\Omega(x)-f|_\Omega(y) \big)+ f(x)f(y) \big( \1_\Omega(x) \1_{\Omega^c}(y)+ \1_{\Omega^c}(x) \1_{\Omega}(y) \big).
\]
Then according to Lemma~\ref{lem:key_estimate} \eqref{bdsub},
\[
\int\frac{|f|_\Omega(x) - f|_\Omega(y)|^2}{|x-y|^n} \varkappa(x,y) \dd x \dd y 
= \int \frac{\big(f(x)-f(y)\big)\big(f|_\Omega(x)-f|_\Omega(y) \big)}{|x-y|^n} \varkappa(x,y) \dd x \dd y +\O(1) .
\]
Since $f$ is smooth, the first term is controlled by 
$\displaystyle\int \frac{\varkappa(x,y)}{|x-y|^{n-1}} \dd x \dd y<\infty$.

Similarly, by  Cauchy-Schwarz's inequality,
\[\begin{aligned}
\int  \chi(\tfrac{x-y}{\hbar}) &(f|_\Omega(x)-f|_\Omega(y))^2  \dd x \dd y \\
&\le \sqrt{\int  \chi(\tfrac{x-y}{\hbar}) (f(x)-f(y))^2  \dd x \dd y }
\sqrt{\int  \chi(\tfrac{x-y}{\hbar}) (f|_\Omega(x)-f|_\Omega(y))^2  \dd x \dd y}
+\O(\hbar^{n+1}) \\
& =  \sqrt{\O(\hbar^{n+2})\int  \chi(\tfrac{x-y}{\hbar}) (f|_\Omega(x)-f|_\Omega(y))^2  \dd x \dd y}
+\O(\hbar^{n+1}) 
\end{aligned}\]
where we used that $ \int  \chi(\tfrac{x-y}{\hbar}) (f(x)-f(y))^2  \dd x \dd = \O(\hbar^{n+2})$ for $f$ Lipchitz-continuous.
This inequality implies \eqref{bdsubmod}, otherwise upon dividing by $\sqrt{\int  \chi(\tfrac{x-y}{\hbar}) (f|_\Omega(x)-f|_\Omega(y))^2  \dd x \dd y}$ we obtain a contradiction. 

Let us also record that by \eqref{bdsup} and an analogous argument, given a cutoff $\overline\chi$, we have
\[
\int  \overline\chi(\tfrac{x-y}{\hbar})  \frac{|f|_\Omega(x) - f|_\Omega(y)|^2}{|x-y|^{n+2}}  \dd x \dd y \le
\sqrt{\int \frac{|f(x) - f(y)|^2}{|x-y|^{n+2}}  \dd x \dd y}
\sqrt{\int  \overline\chi(\tfrac{x-y}{\hbar})  \frac{|f|_\Omega(x) - f|_\Omega(y)|^2}{|x-y|^{n+2}}  \dd x \dd y} +\O(\hbar^{-1}) .
\] 
Since $\displaystyle\int \frac{|f(x) - f(y)|^2}{|x-y|^{n+2}}  \dd x \dd y <\infty$, we conclude that
\begin{equation} \label{bdsupmod}
\int  \overline\chi(\tfrac{x-y}{\hbar})  \frac{|f|_\Omega(x) - f|_\Omega(y)|^2}{|x-y|^{n+2}}  \dd x \dd y =\O(\hbar^{-1}) . \qedhere
\end{equation}
\end{proof}

\begin{lem}\label{lem:cov}
Let $\Omega_1,\Omega_2 \Subset \R^n$ be two relatively compact open
sets with smooth boundary such that the $(n-1)$-Hausdorff measure of the
intersection $\partial \Omega_1\cap \partial \Omega_2$ is zero. Then as $\hbar\to 0$,
\begin{equation} \label{covlog}
\int \frac{|(\1_{\Omega_1}(x)-\1_{\Omega_1}(y))(\1_{\Omega_2}(x)-\1_{\Omega_2}(y))|}{(\hbar +|x-y|)^{n+1}}\dd
x \dd y=o(\log(\hbar^{-1})).
\end{equation}
If in fact there exists $\beta<n-1$ such that $H^{\beta}(\partial
\Omega_1\cap \partial \Omega_2)<+\infty$, then
\begin{equation} \label{covlogHaus}
\int \frac{|(\1_{\Omega_1}(x)-\1_{\Omega_1}(y))(\1_{\Omega_2}(x)-\1_{\Omega_2}(y))|}{(\hbar +|x-y|)^{n+1}}\dd
x \dd y=\O(1).
\end{equation}
\end{lem}
\begin{proof}
We claim that as $\eta\to0$, the integral 
\begin{equation} \label{cov1}
\int_{|x-y|\leq\eta}(\1_{\Omega_1}(x)-\1_{\Omega_1}(y))(\1_{\Omega_2}(x)-\1_{\Omega_2}(y))\dd x \dd y = o(\eta^{1+n}) .
\end{equation}
Indeed, since the integrand is supported on $\{x \in \R^n :
,\dist(x,\partial\Omega_1\cap \partial\Omega_2)<\eta\}$ and
$H^{n-1}(\partial \Omega_1\cap \partial \Omega_2)=0$, by a volume
estimate we obtain
\[
\big|\{(x,y) \in \R^{2n} : \dist(x,\partial\Omega_1\cap \partial
\Omega_2), \dist(x,y)<\eta\}\big| = o(\eta^{1+n}) .
\]
Moreover, for every $\hbar_0>0$, 
\[
\int_{|x-y|>\hbar_0}\frac{(\1_{\Omega_1}(x)-\1_{\Omega_1}(y))(\1_{\Omega_2}(x)-\1_{\Omega_2}(y))}{|x-y|^{n+1}}\dd
x \dd y \le C/\hbar_0
\]

The result follows directly form these estimates by a dyadic decomposition; one has for $\hbar \le \hbar_0$, 
\[
\eqref{covlog} \le \sum_{k \le \log(1/\hbar)} 
\int_{\hbar2^{k-1}\leq |x-y|\leq \hbar2^{k}}\frac{|(\1_{\Omega_1}(x)-\1_{\Omega_1}(y))(\1_{\Omega_2}(x)-\1_{\Omega_2}(y))|}{|x-y|^{n+1}}\dd x \dd y +\O(\hbar_0^{-1})
\]
Using \eqref{cov1}, with $\eta = \hbar 2^{k}$, these integrals are all
$o(1)$ uniformly for $k \le \log(\hbar^{-1})$, $\hbar\le \hbar_0$, as $\hbar_0\to 0$. 
This shows that 
\(
\eqref{covlog} = o(\log(\hbar^{-1}))
\)
by choosing $\hbar_0$ sufficiently small.

Under any positive improvement on the Hausdorff dimension of $\partial
\Omega_1\cap \partial \Omega_2$, one has instead, for some $\epsilon>0$,
\[
\big|\{(x,y) \in \R^{2n} : \dist(x,\partial\Omega_1\cap \partial
\Omega_2), \dist(x,y)<\eta\}\big| = o(\eta^{1+n+\epsilon}),
\]
and now the estimates of the dyadic decomposition form a convergent series. This
concludes the proof.
\end{proof}

As we already alluded to, the following class of functions plays a key
role when bounding the integral~\eqref{comm1}. 

\begin{defn}[Class $\mathcal{F}$] \label{class:F}
Let $m,n\in \N_0$, define the class 
\[\begin{aligned}
\mathcal{F}(\R^m,\R^n) = \big\{ 
f : \R^m \times \R^n \to \R \text{ continuous with compact support} ;
x\mapsto f(x,\xi) \text{ is smooth for }\xi\in\R^n , \\
\text{and for every }j,k\in\N_0, \, 
\big\|\partial_\xi^k \partial_x^j f(x,\xi)\big\| \le C_{j,k} |\xi|^{1-k} \text{ for all } (x,\xi) \in \R^m \times \R^n \big\} . 
\end{aligned}\]
\end{defn}

We record the following two lemmas without proofs as they are direct consequences of this definition.

\begin{lem}
Functions in $\mathcal{F}$ are Lipschitz continuous with respect to $\xi$ and there are smooth on $\R^{n}\setminus\{0\}$.  
If $g \in C^\infty_c(\R^m\times \R^n\times\R)$ with $g(x,\xi,0) =0$,  then 
$f(x,\xi) = g(x,\xi,|\xi|)$ is in the class~$\mathcal{F}$. 
We also emphasize that the class $\mathcal{F}$ is stable under smooth change of variables on $\R^m\times \R^n$ which coincide with the identity outside of a compact set.
\end{lem}

\begin{lem} \label{lem:trunc}
Let $a \in \mathcal{F}(\R^m,\R^n)$. 
For any cutoff $\chi\in C^\infty_c(\R^n)$, as $\hbar\to 0$,
\[
\int  \chi(\tfrac{\xi}{\hbar}) |a(x,\xi)| \dd \xi \dd x  =\O(\hbar^{n+1}) . 
\]   
\end{lem}
The last lemma is complemented by the following claim on the zone
$\{|\xi|\geq \hbar\}$.

\begin{lem} \label{lem:1}
Let $n,k\in \N$ with $k\leq n+2$. Let $b: \R^{3n} \to \R$ have compact
support, smooth with respect to $\xi$ on $\R^{n}\setminus\{0\}$ and
such that for every $j\in\N_0$, and $\xi\in \R^n$,
\begin{equation} \label{conda1}
\int \frac{\| \partial_\xi^j b(x,y,\xi)\|}{|x-y|^{n+2-k}} \dd x \dd y \le  |\xi|^{1-j-k}   , \qquad \text{for }\xi\in\R^n.
\end{equation}
Let $\chi\in C^{\infty}(\R^{3n})$ be such that, for every $x,y,\xi\in \R^n$,
\[
  \1_{|\xi|\leq 1}\leq \chi(x,y,\xi) \leq \1_{|\xi|\leq 2}.
\]
Then as $\hbar\to0$,
\[
\int e^{i\frac{(x-y)\cdot \xi}{\hbar}}(1-\chi(x,y,\tfrac{\xi}{\hbar}))b(x,y,\xi) \dd \xi \dd x \dd y  =\O(\hbar^{n+1-k}) .
\]
\end{lem}
\begin{proof}
  We will prove this claim by decreasing order of the value of $k$: at
  $k=n+2$, the claim follows from the simple fact that, for any fixed
  $R>0$, as $\hbar\to 0$, one has
  \[
    \int_{\hbar\leq |\xi|\leq R}|\xi|^{-n-1}=\O(\hbar^{-1}).
  \]
  Suppose now that the claim holds for some value of $k$ and let us
  prove it for $k-1$; to this end we integrate by parts
  \begin{multline*}
    \int e^{i\frac{(x-y)\cdot
        \xi}{\hbar}}(1-\chi(x,y,\tfrac{\xi}{\hbar}))b(x,y,\xi) \dd \xi
    \dd x \dd y \\
    =i\hbar\int e^{i\frac{(x-y)\cdot
        \xi}{\hbar}}(1-\chi(x,y,\tfrac{\xi}{\hbar}))\frac{(x-y)\cdot \partial_{\xi}b(x,y,\xi)}{\|x-y\|^2} \dd \xi
    \dd x \dd y\\
    -i\int e^{i\frac{(x-y)\cdot
        \xi}{\hbar}}(\partial_{\xi}\chi)(x,y,\tfrac{\xi}{\hbar})\cdot \frac{(x-y)b(x,y,\xi)}{|x-y|^2} \dd \xi
    \dd x \dd y
  \end{multline*}
  Note that the first term in the equation is exactly of the form
  given before, where now $B=\frac{(x-y)\cdot
    \partial_{\xi}b(x,y,\xi)}{\|x-y\|^2}$ satisfies \eqref{conda1};
  thus
  \[
    i\hbar\int e^{i\frac{(x-y)\cdot
        \xi}{\hbar}}(1-\chi(x,y,\tfrac{\xi}{\hbar}))\frac{(x-y)\cdot \partial_{\xi}b(x,y,\xi)}{|x-y|^2} \dd \xi
    \dd x \dd y=\O(\hbar^{n+2-k}).
  \]
  It remains to bound
  \[
    i\int e^{i\frac{(x-y)\cdot
        \xi}{\hbar}}(\partial_{\xi}\chi)(x,y,\tfrac{\xi}{\hbar})\cdot \frac{(x-y)b(x,y,\xi)}{|x-y|^2} \dd \xi
    \dd x \dd y =\O\left(\int_{\hbar\leq |\xi|\leq
        2\hbar}|\xi|^{2-k}\right)=\O(\hbar^{n+1-k}).
    \]
  This concludes the proof.
\end{proof}

Finally, we also need an estimate for the commutator between a pseudodifferential operators with smooth, compactly supported symbols and a function with jump discontinuities.
This is to be compared with Proposition~\ref{prop:J1commsmooth}, albeit simpler.

\begin{lem} \label{lem:commsmooth}
Let $f\in C^{\infty}(\R^n)$, $a\in S^1(\R^{3n})$
and 
\(\displaystyle
K: (x,y)\mapsto \frac{1}{(2\pi\hbar)^n}\int
e^{i\tfrac{\xi\cdot
(x-y)}{\hbar}}a(x,y,\xi)\dd \xi .
\)
Then, for any bounded open set $\Omega\Subset \R^n$ with smooth boundary, 
\[
\big\| [f|_\Omega, K] \big\|^2_{\J^2} =O(\hbar^{1-n}).
\]
\end{lem}

\begin{proof}
By definition, the Hilbert-Schmidt norm, 
\begin{align*}
\| [f|_\Omega, K] \big\|^2_{\J^2} =\frac{1}{(2\pi\hbar)^{2n}}\int 
e^{i\frac{(x-y)\cdot(\xi_1-\xi_2)}{\hbar}}  (f|_\Omega(x)-f|_\Omega(y))^2a\left(x,y,\xi_1\right)\overline{a}\left(y,x,\xi_2\right) \dd \xi_1 \dd \xi_2  \dd x \dd y .
\end{align*}
We make a change of variable $\xi = \xi_1-\xi_2$ and $\zeta = \frac{ \xi_1+\xi_2}{2}$ so that 
\[
\| [f|_\Omega, K] \big\|^2_{\J^2}
= \frac{1}{(2\pi\hbar)^{2n}}\int 
e^{i\frac{(x-y)\cdot \xi}{\hbar}}  (f|_\Omega(x)-f|_\Omega(y))^2 b\left(x,y,\xi,\zeta\right) \dd \xi \dd \zeta  \dd x \dd y 
\]
where $b\in S^1(\R^{4n})$.
Using the bound \eqref{bdsubmod}, up to an error  $\O(\hbar^{1-n})$, we can introduce a  cutoff excluding the diagonal in the previous integral. Then, we perform integrations by parts with respect to $\xi$;  we obtain for any $k\in\N_0$
\[
\| [f|_\Omega, K] \big\|^2_{\J^2}
= \frac{(i \hbar)^k}{(2\pi\hbar)^{2n}}\int  \overline\chi(\tfrac{x-y}{\hbar})
e^{i\frac{(x-y)\cdot\xi}{\hbar}}  (f|_\Omega(x)-f|_\Omega(y))^2 
\frac{(x-y)^{\otimes k}\cdot \partial_\xi^k b\left(x,y,\xi,\zeta\right) }{|x-y|^{2k}} \dd \xi \dd \zeta  \dd x \dd y +\O(\hbar^{1-n}) .
\]
We can bound for $(x,y,\xi,\zeta) \in\R^{4n}$ with $k = n+2$,
\[
\left|\frac{(x-y)^{\otimes k}\cdot \partial_\xi^k b\left(x,y,\xi,\zeta,\hbar\right) }{|x-y|^{2k}}\right| \le  \frac{\varkappa(\xi,\zeta)}{|x-y|^{n+2}}
\]
where $\varkappa \in C^\infty_c(\R^{2n})$. Thus, by \eqref{bdsupmod}, we conclude that $\big\| [f|_\Omega, K] \big\|^2_{\J^2} =O(\hbar^{1-n})$.
\end{proof}

\subsection{Reduction to an oscillatory integral} \label{sec:red}

Starting from the expression \eqref{comm1}, using the stationary phase
method and certain well-chosen changes of variables, we can prove the following result. 

\begin{prop} \label{prop:red}
Under the assumptions from Section~\ref{sec:not}, there is a symbol $F\in\mathcal{F}(\R^{2n},\R^n)$ (independent of $\Omega$) such that as $\hbar\to0$,
\[
\| [\widetilde{\Pi}, f|_\Omega]  \|^2_{\J^2}= \frac{-1}{(2\pi\hbar)^{2n}}\int
e^{i\tfrac{(x-y)\cdot\xi}{\hbar}} F(x,y,\xi)(f|_\Omega(x)-f|_\Omega(y))^2\dd \xi \dd x\dd y +  \O(\hbar^{1-n}) . 
\] 
Moreover, the principal part of $F$ satisfies uniformly as $\xi \to 0$, 
\begin{equation} \label{Fpp}
F(x,x,\xi)|_{\hbar=0} =\cst_{n-1}  |\xi| |V(x)|^{\frac{n-1}2} + \O(|\xi|^2)
\end{equation}
with the constant $\cst_{n-1}$ as in Theorem~\ref{thm:J2} for $n\in\N$. 
\end{prop}

Starting from formula \eqref{comm1}, let us explain the main steps of the proof of Proposition~\ref{prop:red}.
\begin{itemize}[leftmargin=*] 
\item {\bf Step 1.} As in Section~\ref{sec:hilbert-schmidt-norm}, we perform a change of coordinates similar to \eqref{change_coord} 
\begin{equation} \label{cv}
\begin{cases}
\lambda=\frac{\lambda_1+\lambda_2}{2} ,\qquad\xi_1=  r_1\omega_1\\
\sigma=\lambda_2-\lambda_1 , \quad  \xi_2=r_2 \omega_2
\end{cases}
\quad\text{with $r_1,r_2>\underline{\cst}$, $\omega_1,\omega_2\in \S^{n-1}$}. 
\end{equation}
This is justified because of the support condition \eqref{suppA0},
which ensures that $\xi_1$ and $\xi_2$ are bounded away from zero. Observe that $ \1_{\lambda_1\leq
0}\1_{\lambda_2\leq 0}=\1_{|\sigma|\le -2\lambda}$ under this change
of variables.
Then, we can perform a stationary phase with respect to the variables
$(t_1,r_1,t_2,r_2)$, keeping $(x,y,\omega_1,\omega_2,\lambda,\sigma)$
fixed, and we obtain
an integral of the type
\begin{equation} \label{intred1}
\frac{1}{(2\pi\hbar)^{2n}}\int
e^{i\frac{\Psi_1(x,y,\omega_1,\omega_2,\lambda,\sigma)}{\hbar}}(f|_\Omega(x)-f|_\Omega(y))^2 A_1(x,y,\omega_1,\omega_2,\lambda,\sigma)\1_{\{|\sigma|\le -2\lambda\}}\dd x \dd y \dd \omega_1 \dd \omega_2\dd \lambda\dd\sigma
\end{equation}
where the amplitude $A_1 \in S^1(\R^{4n+2})$ is also supported in $\big\{t\in[-\underline{\tau},\underline{\tau}], (x,y)\in \Omega^{\prime2},  |x-y|\le \underline{\ell}, |\lambda| \le \underline{\ell}\}$. 

\item {\bf Step 2.} We study the phase $\Psi_1$.
This phase vanishes
along the diagonal $\{x=y\}$ and it can be factorised into
\[
\Psi_1(x,y,\omega_1,\omega_2,\lambda,\sigma)
= (x-y)\cdot \zeta(x,y,\omega_1,\omega_1,\lambda,\sigma)   
\]
where $\zeta\in\R^{4n+2} \to \R^n$ is smooth and approximated by
$\zeta \simeq \sqrt{\lambda-V(x)} (\omega_1-\omega_2) $ for $x-y$ small.
This allows us to show that for $\omega_1\in \S^{n-1}$, $x\in
\Omega'$, $|x-y| \le \underline{\ell}$ fixed,  the map $(\omega_2,
\sigma) \mapsto \zeta$  is a (smooth) diffeomorphism from a
neighborhood of $(\omega_1,0)$ in $S^{n-1}\times \R$  onto $\{\zeta\in
\R^n  : |\zeta|\le \delta \}$, where $\delta$ can be chosen much
larger than $\underline{\ell}$ and much smaller than $\underline{\cst}$.

We will justify in step 4 that the main contribution to the integral \eqref{intred1} comes from this region.
Hence, by a change of coordinates, the main term in $\| [\widetilde{\Pi}, f|_\Omega]  \|^2_{\J^2}$ is 
\[
\frac{-1}{(2\pi\hbar)^{2n}}\int  
e^{i\frac{(x-y) \cdot \zeta}{\hbar}}(f|_\Omega(x)-f|_\Omega(y))^2 A_2(x,y,\omega ,\zeta,\lambda )\1_{\{|\sigma(x,y,\zeta,\omega,\lambda)| >-2\lambda>0 \}} \dd x \dd y \dd \zeta \dd \omega\dd \lambda
\]
where  $A_2 \in S^1(\R^{4n+1})$, $\sigma : \R^{4n+1}\to\R$ is smooth and $\sigma \simeq  2\sqrt{\lambda-V(x)}\, \zeta \cdot\omega$.  

\item {\bf Step 3.} 
At this stage, we argue that symbols of the type
\[
(x,y,\xi) \mapsto  \int_{\S^{n-1} \times\R} A_2(x,y,\omega,\xi,\lambda)\1_{\{|\sigma(x,y,\xi,\omega,\lambda)| >-2\lambda>0 \}} \dd \omega\dd \lambda
\] 
are in the class $\mathcal F$ and we prove \eqref{Fpp}.

\item {\bf Step 4.} It remains to show that the remaining part of the integral \eqref{intred1},
\begin{equation}\label{intremain}\begin{aligned}
\frac{1}{(2\pi\hbar)^{2n}}\int  e^{i\tfrac{(x-y)\cdot \zeta(x,y,\omega_1,\omega_1,\lambda,\sigma)   
}{\hbar}} 
&(1-\chi_\delta)\big( \zeta(x,y,\omega_1,\omega_1,\lambda,\sigma)    \big)(f|_\Omega(x)-f|_\Omega(y))^2\\
&A_2(x,y,\omega_1,\omega_2,\lambda,\sigma)\1_{\{|\sigma|\le -2\lambda\}}\dd x \dd y \dd \omega_1 \dd \omega_2\dd \lambda\dd\sigma,
\end{aligned}
\end{equation}
is $\O(\hbar^{1-n})$. Here, as usual $\chi_{\delta}(\cdot)=\chi(\delta
\cdot)$, and the function $\chi:\R^n\to [0,1]$ is smooth, equal to $1$ on
$B(0,\tfrac{\delta}{2})$ and to $0$ outside $B(0,\delta)$.

If $|x-y|\leq \underline{\ell}$ with $\underline{\ell}$ small with
respect to $\delta$, then the oscillating phase
\eqref{intremain}  has no stationary point. Thus we expect
\eqref{intremain} to be rather small. The issue is that the amplitude is not smooth with respect to $(x,y)$.
Nevertheless, we  can integrate by parts twice, once with respect to
$x$ and once with respect to $y$, and this reduces \eqref{intremain}
(up to an error  $\O(\hbar^{1-n})$)  to an integral over the (smooth)
boundary of $\Omega$. It remains to bound an integral of the type
\[
\frac{1}{(2\pi\hbar)^{2(n-1)}}\int_{\{ \hat x, \hat y \in \partial \Omega \}\times\{\omega_1,\omega_2 \in \partial B\}}
e^{i\tfrac{(\hat x-\hat y)\cdot \zeta(\hat x, \hat y,\omega_1,\omega_1,\lambda,\sigma)   
}{\hbar}} \overline{\chi_\delta}(\omega-\omega_2)) A_3(\hat x, \hat y,\omega_1,\omega_2,\lambda,\sigma) \dd \hat x \dd \hat y \dd \omega_1 \dd \omega_2
\]
where $A_3 \in S^1(\R^{4n+2})$.
To complete the proof, we argue that for a $\hat x \in
\partial\Omega$, the phase has non-singular Hessian in one of the pair $(\hat y, \omega_1)$ or
$(\hat y, \omega_2)$. 
Hence, by a stationary phase argument, this integral is at most of order $\O(\hbar^{1-n})$. This completes the proof. 
\end{itemize}
\subsubsection{Step 1: Stationary phase} \label{sec:statphase}

We make the change of coordinates \eqref{cv} (the Jacobian of this change of variable is $(r_1r_2)^{n-1}$ for $n\ge 1$) and study the critical point of the phase $\Psi_0$ in the variables $(t_1,r_1,t_2,r_2)$. According to \eqref{phase0} and \eqref{phasephi}, we verify that
\begin{equation} \label{devcrit} 
\begin{aligned}
\partial_{t_1} \Psi_1(x,y,\xi_1,\xi_2,t_1,t_2,\lambda_1,\lambda_2)
& = \partial_{\xi}\varphi(t_1,x,\xi_1)- \lambda_1 = V(x) + r_1^2 - \lambda_1 +  \O(t_1)  \\
\partial_{r_1} \Psi_1(x,y,\xi_1,\xi_2,t_1,t_2,\lambda_1,\lambda_2)
& = \omega_1 \cdot \big( \partial_{\xi}\varphi(t_1,x,\xi_1)-y \big)
= \omega_1 \cdot \big( (x-y) + \O(t_1^2)  \big) + 2 t_1 r_1 
\end{aligned}
\end{equation}
and similarly for $\partial_{t_2} \Psi_1, \partial_{r_2} \Psi_1$
switching $1\leftarrow 2$ and the sign.

Let us first study the first line of \eqref{devcrit}. Since $|t_1|\leq \underline{\tau}\ll
\underline{c}\leq \lambda-\sigma/2-V(x)$, for every $t_1$ there exists a unique
$r_1^c$ solving the first equation, it is a smooth function of all
other parameters, and
$r_1^c|_{t_1=0}=\sqrt{\lambda+\sigma/2-V(x)}$. In particular,
$r_1^c\geq \underline{c}$. Now we turn to the
second line of \eqref{devcrit}: with $r_1=r_1^c$ bounded away from
below, for every $|x-y|\leq \underline{\ell}\ll \underline{c}$ there
exists at most one solution for $t_1$ such that $|t_1|\leq
\underline{\tau}$; moreover at $y=x$ this solution exists and is equal
to $0$.

All in all, for $x-y$ small there exists a unique stationary point given by
\begin{equation} \label{critptrt}
\begin{aligned}
r_1^c&=\sqrt{\lambda+\sigma/2-V(x)}+O(|x-y|) \\
t_1^c &= - \frac{ \omega_1 \cdot (x-y) }{2 \sqrt{\lambda+\sigma/2-V(x)}}+ \O(|x-y|^2)
\end{aligned}
\end{equation}
and similarly for $(r_2^c, t_2^c)$ replacing $(\lambda+\sigma) \leftarrow (\lambda-\sigma)$ and $\omega_1 \leftarrow \omega_2$. 
Differentiating \eqref{devcrit} again, we obtain that the Hessian of $\Psi_1$ with respect to $(t_1,r_1,t_2,r_2)$ is of the form
\[
2 \left(\begin{smallmatrix}
\star &r_1+\O(t_1)&0&0\\
r_1+\O(t_1)&t_1+\O(t_1^2)&0&0\\
0&0& \star
& - r_2+\O(t_2)\\
0&0&r_2+\O(t_2)&t_2+\O(t_2^2)
\end{smallmatrix} \right) .
\] 
For $r_1,r_2 \geq \underline{\cst}/4$ and $|t_1|,|t_2|\leq \underline{\tau}$, this Hessian is non-degenerate and its determinant, evaluated at the critical point \eqref{critptrt}, is given by
\begin{equation} \label{discrit}
\J(x,y,\omega_1,\omega_2,\sigma,\lambda)
=  16\big((\lambda-V(x))^2-\sigma^2/4\big)+\O(|x-y|).
\end{equation}

We introduce a new phase
\[
\Psi_1:(x,y,\omega_1,\omega_2,\lambda,\sigma)\mapsto
\Psi_0(x,y,t_1^c,t_2^c,r_1^c\omega_1,r_2^c\omega_2,\lambda+\tfrac\sigma2,\lambda-\tfrac\sigma2) 
\] 
and we denote 
\begin{equation} \label{def:R}
R= R(x,\lambda) := \sqrt{\lambda-V(x)} . 
\end{equation}
By \eqref{critptrt}, if $x=y$ and $\sigma=0$ one has $R=r_j^c$ for
$j\in\{1,2\}$ as well as 
$\J=(2R)^4$. 

\smallskip

We are in position to apply Lemma~\ref{lem:statphase} to the integral
\eqref{comm1}, with the variables
$(x,y,\omega_1,\omega_2,\lambda,\sigma)$ as parameters. There is an amplitude $A_1 \in S^1(\R^{4n+2})$ so that
\begin{multline} \label{comm2} 
\big\| [\widetilde{\Pi}, f|_\Omega] \big\|^2_{\J^2}\\= \frac{1}{(2\pi\hbar)^{2n}}\int
e^{i\tfrac{\Psi_1(x,y,\omega_1,\omega_2,\lambda,\sigma)}{\hbar}}(f|_\Omega(x)-f|_\Omega(y))^2A_1(x,y,\omega_1,\omega_2,\lambda,\sigma) \1_{\{|\sigma|\le -2\lambda\}}\dd x \dd y \dd \omega_1\dd \omega_2\dd \sigma\dd \lambda  . 
\end{multline}
Moreover, the principal symbol of $A_1$ is given as follows, in terms of \eqref{critptrt} and \eqref{discrit}, 
\[
A_1(x,y,\omega_1,\omega_2,\lambda,\sigma)|_{\hbar=0}
=  \frac{(r_1^c r_2^c)^{n-1}}{\sqrt{\J}}  A_0(x,y,r_1^c\omega_1,r_2^c\omega_2,t_1^c,t_2^c,\lambda+\tfrac\sigma2,\lambda-\tfrac\sigma2) .
\]
The factor $(r_1^cr_2^c)^{n-1}$ comes from the Jacobian of the change
of coordinates.
In particular, according to \eqref{symbolA}, at $(y,\sigma) = (x,0)$, 
\begin{align}\notag
A_1(x,x,\omega_1,\omega_2,\lambda,0)
& = \frac{R^{2n-2} A_0(x,x, R \omega_1,R\omega_2,0,0,\lambda,\lambda)}{\sqrt{R^4}}
= R^{2n-4} \vartheta(V(x)+R^2)^2
\chi(\lambda)^2/4 \\
&\label{symbold}
=R^{2(n-2)} \mathrm{g}(\lambda)/4 , 
\qquad\qquad x\in\Omega' , \lambda\in\R, \omega_1,\omega_2 \in \S^{n-1} ,
\end{align}
where $R$ is given by \eqref{def:R} and $\mathrm{g} := \vartheta^2\chi^2$ satisfies $\mathrm{g}(0)=0$. 
Moreover,  $A_1$ is supported in
\[
\big\{(x,y)\in \Omega'\times\Omega',  |x-y|\le \underline{\ell}, |\lambda| \le \underline{\ell}\} .
\]

\subsubsection{Step 2: Study of the phase $\Psi_1$} \label{sec:phase}

By \eqref{critptrt}, $t_1^c=t_2^c=0$ if $x=y$, then with $\xi_j = r_j^c \omega_j$,  we have on the diagonal
\[
\Psi_1(x,x,\omega_1,\omega_2,\lambda,\sigma)= \varphi(0,x,\xi_1)-\varphi(0,x,\xi_2)-(\xi_1-\xi_2)\cdot x = 0 .
\]
Thus there exists a smooth function $\zeta : \R^{2n+2} \to \R^n$ such that we can write
\begin{equation} \label{phase2}
\Psi_1(x,y,\omega_1,\omega_2,\lambda,\sigma)= ( 
x-y)\cdot \zeta(x,y,\omega_1,\omega_2,\lambda,\sigma).
\end{equation}

Moreover, the original phase
$\Psi_1$ is anti-symmetric with respect to the exchange
$\xi_1\leftrightarrow \xi_2,t_1\leftrightarrow
t_2, \lambda_1\leftrightarrow \lambda_2$.
From the critical point equation, this implies that $(r_1^c,t_1^c)=(r_2^c,t_2^c)=(r^c,t^c)$ if $(\omega_1,\sigma) =(\omega_2,0)$ 
so that
\[
\Psi_1(x,y,\omega,\omega,\lambda,0)
= \Psi_0(x,y,t^c,t^c,r^c\omega,r^c\omega,\lambda,\lambda) 
= 0
\]
Hence
\[
\zeta(x,y,\omega,\omega,\lambda,0)=0 , \qquad \text{on }\big\{(x,y)\in \Omega'\times\Omega',  |x-y|\le \underline{\ell}, |\lambda| \le \underline{\ell}\} .
\]
In particular, with $\mathcal{Z}(x,\omega_1,\omega_2,\sigma,\lambda) = \zeta(x,x,\omega_1,\omega_2,\sigma,\lambda)$, 
\begin{equation} \label{Z2}
\zeta(x,y,\omega_1,\omega_2,\sigma,\lambda) = \mathcal{Z}(x,\omega_1,\omega_2,\sigma,\lambda) 
+\O\Big(\sqrt{ |\omega_1-\omega_2|^2 +  |\sigma|^2}|x-y|\Big) . 
\end{equation}

Recall that $|x-y| < \underline{\ell}$ in the integral in questions,
so we first study $\mathcal{Z}$ instead of $\zeta$. Note that
\begin{equation} \label{phaseZ}
\mathcal{Z}(x,\omega_1,\omega_2,\sigma,\lambda)=
- \partial_y \Psi_1  \big|_{y=x}  = -  \partial_y \Psi_0 \big|_{y=x,t_1=t_2=0,r_1=r_1^c,r_2=r_2^c}  . 
\end{equation}
Here we used the property of the critical point for which  $ \partial_{t_j} \Psi_1 |_{t_1^c,r_1^c,t_2^c,r_2^c} =   \partial_{r_j} \Psi_1 |_{t_1^c,r_1^c,t_2^c,r_2^c}=0 $ for $j\in\{1,2\}$ and $t_1^c=t_2^c=0$ if $x=y$. Since $ \partial_y \Psi_1= \xi_1-\xi_2 $, by \eqref{critptrt},  this yields
\begin{equation} \label{Z1}
\mathcal{Z}(x,\omega_1,\omega_2,\sigma,\lambda) = 
r_1^c \omega_1 - r_2^c \omega_2
= \omega_1\sqrt{\lambda-V(x)+\tfrac\sigma2}- \omega_2\sqrt{\lambda-V(x)-\tfrac\sigma2}.
\end{equation}
Linearizing this function for small~$\sigma$ (here, $|\sigma|\le 2\underline{\ell}$), we obtain
\[
\mathcal{Z}(x,\omega_1,\omega_2,\sigma,\lambda)  
=  ( \omega_1-\omega_2 )  \big(R+ \O(\sigma^2) \big) 
+ (\omega_1 +\omega_2)  \tfrac{\sigma}{4R}  \big(1+ \O(\sigma^2) \big)  
\]
with $R\ge \underline{\cst}$ as in \eqref{def:R}. 

We now consider the equation $\mathcal{Z}(x,\omega_1,\omega_2,\sigma,\lambda) = \xi$ for fixed $(x, \omega_1,\lambda) \in \Omega' \times \S^{n-1}\times [-\underline{\ell},\underline{\ell}]$ and $\xi\in\R^n$. 
We have 
\[
\omega_2 \big(R- \tfrac{\sigma}{4R} \big) =  \omega_1  \big(R+ \tfrac{\sigma}{4R} \big) - \xi  + \O(\sigma^2) . 
\] 
To solve this equation, we decompose $\omega_2 =  \alpha \omega_1
+\nu$ where $\nu \in \omega_1^\perp$ and $\alpha = \sqrt{1-|\nu|^2}  \ge 0$, we obtain
\[ \begin{cases}
\alpha   = 1+ \tfrac{\sigma}{2R^2} - \tfrac{\xi_1}{R-
  \tfrac{\sigma}{4R} }  + \O(\sigma^2)   ,  &\text{ where }\xi_1 = \xi\cdot \omega_1 \\
\nu =  - \tfrac{\xi^\perp}{R- \tfrac{\sigma}{4R} }  + \O(\sigma^2)  ,
&\text{ where }\xi  = \xi_1+ \xi^\perp . 
\end{cases}\] 
In particular, we have 
\[
1 = \alpha^2 + |\nu|^2 = 1+ \tfrac{|\xi|^2}{(R-\sigma/4R)^2}
+  \tfrac{\sigma}{R^2} - \tfrac{2 \xi_1}{R-\sigma/4R} + \O(\sigma^2) . 
\]

Since $\delta\ll  \underline{\cst}\leq R $, the last equation
determines $\sigma$ for $|\xi|<\delta$, and
\[
\sigma = 2 R \xi_1 + \O(|\xi|^2)
\qquad\text{and then}\qquad
\omega_2 = \omega_1 -\xi^\perp/R +  \O(|\xi|^2) .
\]

To summarize, choosing a small parameter $\delta$ with
$\underline{\ell}\ll \delta \ll \underline{\cst}$, given $(x, \omega_1,\lambda) \in \Omega'\times \S^{n-1}\times\R$ and $\xi\in \delta B_n$, the equation
$\mathcal{Z}(x,\omega_1,\omega_2,\sigma,\lambda) = \xi$ 
has a unique smooth solution $(\sigma,\omega_2) \in \R \times \S^{n-1}$.
Moreover, by \eqref{Z1}, we also compute 
\[
\partial_{\sigma} \mathcal{Z}
= \tfrac{\omega_1+\omega_2}{4R}  +\O\big(\sigma(\omega_1-\omega_2)\big) , \qquad
\partial_{\omega_2} \mathcal{Z} = R \mathrm{I}_n+ \O(\sigma) ,
\]
where $\mathrm{I}_n$ denotes the identity matrix. These derivatives are non-degenerate in a neighbourhood of the previous solution, so according  to \eqref{Z2}, the equation $\zeta(x,y,\omega_1,\omega_2,\sigma,\lambda) = \xi$,  given $(x, \omega_1,\lambda) \in \Omega'\times \S^{n-1}\times\R$ and $(y,\xi)\in\R^{2n}$ with $|x-y|<\underline{\ell}$, $|\xi|<\delta$, also has a  unique smooth solution  $(\sigma,\omega_2) \in \R \times \S^{n-1}$ with the expansions: 
\begin{equation}\label{sigma1}
\begin{cases}
\sigma = 2 R \xi_1 +  \O\big(|\xi|(|\xi|+ |x-y|)\big)  \\
\omega_2 =   \omega_1 -\xi^\perp/R +\O\big(|\xi|(|\xi|+ |x-y|)\big) 
\end{cases}
, \qquad \xi= (\xi_1,\xi^\perp) ,\, \xi_1 = \xi \cdot \omega_1 , \,  R=\sqrt{\lambda-V(x)} . 
\end{equation} 

Hence, we can make a change of variable $\xi\in \delta B_n \mapsto
(\sigma,\omega_2)$, after which the phase becomes
\[
\Psi_1(x,y,\omega_1,\omega_2,\sigma,\lambda)=  
(x-y)\cdot \xi . 
\]

To use this, we must split the integral \eqref{comm2} in two parts by introducing a cutoff in $\mathcal{Z}$ (equivalently in $\xi$ after the change of coordinates):
\[
\big\| [\widetilde{\Pi}, f|_\Omega] \big\|^2_{\J^2}
= I_{1,\hbar} + I_{2,\hbar}
\]
where 
\begin{equation} \label{comm3} 
\begin{aligned} 
I_{1,\hbar}   & = \frac{1}{(2\pi\hbar)^{2n}}\int
e^{i\tfrac{\Psi_1(x,y,\omega_1,\omega_2,\sigma,\lambda)}{\hbar}}(f|_\Omega(x)-f|_\Omega(y))^2 A_1'(x,y,\omega_1,\omega_2,\sigma,\lambda)  \1_{\{|\sigma|\le -2\lambda\}}\dd x \dd y \dd \omega_1\dd \omega_2\dd \sigma\dd \lambda , \\
I_{2,\hbar}   & = \frac{1}{(2\pi\hbar)^{2n}}\int
e^{i\tfrac{(x-y)\cdot \xi}{\hbar }}(f|_\Omega(x)-f|_\Omega(y))^2
A_2(x,y,\omega_1,\xi,\lambda) \1_{\{|\sigma(x,y,\xi,\omega_1,\lambda)|\le -2\lambda\}}\dd x \dd y \dd \omega_1\dd \xi\dd \lambda .
\end{aligned}
\end{equation}

We tune a smooth cutoff so that
\begin{itemize}[leftmargin=*] 
\item the amplitude $A_1'$ belongs to  $S^1(\R^{4n+2})$ with $A_1'(x,y,\omega_1,\omega_2,\sigma,\lambda) =0$ on the set 
$\big\{|\mathcal{Z}\big(x,\omega_1,\omega_2,\sigma,\lambda) |\le \delta \big\}$ 
and $A_1'$ is supported on the set $(x, \lambda) \in \Omega' \times[-\underline{\ell},\underline{\ell}]$  and $|x-y| \le \underline{\ell}$ with $\underline{\ell}\ll \delta$. 
\item the amplitude $A_2 \in S^1(\R^{4n+1})$ and it is supported in
\[(x, \lambda) \in \Omega' \times\R,\;(y,\xi) \in \R^{2n},\;|y-x|+|\lambda|+|\xi|\le 2\delta\ll
  \underline{\cst}.
  \]
\end{itemize}

Finally, the principal part of $A_2$ satisfies on the diagonal $x=y$
and at $\xi=0$ (where $\mathcal{Z}=0$ also)
\begin{equation*}
A_2(x,x,0,\omega_1,\lambda)|_{\hbar=0} = A_1(x,x,\omega_1,\omega_1,\lambda,0) \bigg(\frac{\dd \sigma \dd\omega_2}{\dd \xi_1 \dd \xi^\perp}\bigg)\bigg|_{y=x, \xi=0}
\end{equation*}
since $(\sigma,\omega_2)=(0,\omega_1)$ when $(y,\xi) =(x,0)$, cf.~\eqref{sigma1}. Moreover, we compute the Jacobian
\[
\frac{\dd \sigma \dd\omega_2}{\dd\xi_1\dd\xi^\perp} \bigg|_{y=x, \xi=0}=   \left| \det \left(\begin{smallmatrix} 
2R &  \hdots & 0 &\hdots \\
\vdots & \ddots \\
0& &  \mathrm{I}_{n-1}/R  \\
\vdots & & & \ddots
\end{smallmatrix}\right) \right| = 2R^{2-n} . 
\]
Hence, by \eqref{symbold} with $R(x,0) = |V(x)|^{1/2}$, we obtain 
\begin{equation} \label{symbolg1}
A_2(x,x,0,\omega,0) = R(x,0)^{n-2}  \mathrm{g}(0)/2 = |V(x)|^{n/2-1} /2 , \qquad 
x\in\Omega' , \omega\in \S^{n-1}.
\end{equation}

\subsubsection{Step 3: the main symbol is in the class $\mathcal F$} \label{sec:F}

We now prove the following fact.
\begin{prop} \label{prop:symbolF}
The symbol
\[
F: (x,y,\xi) \in \R^{3n} \times (0,1]  \mapsto \int_{\S^{n-1}\times\R}
\hspace{-.5cm} A_2(x,y,\xi,\omega,\lambda)
\1_{\{|\sigma(x,y,\xi,\omega,\lambda)| >-2\lambda>0 \}} \dd \omega\dd \lambda .
\]
is in the class $\mathcal F(\R^{2n},\R^n)$ (cf.~Definition~\ref{class:F}) and it satisfies \eqref{Fpp}. 
\end{prop}

According to \eqref{comm3} and the previous definition, we can rewrite
\[ 
I_{2,\hbar}  = \frac{-1}{(2\pi\hbar)^{2n}}\int
e^{i\tfrac{(x-y)\cdot \xi}{\hbar }} (f|_\Omega(x)-f|_\Omega(y))^2 F(x,y,\xi)  \dd x \dd y \dd \xi  + \|[f|_\Omega,K]\|^2
\]
where
\[ 
K : (x,y) \mapsto \frac{1}{(2\pi\hbar)^{n}}\int
e^{i\tfrac{(x-y)\cdot \xi}{\hbar }}
a(x,y,\xi) \dd x \dd y \dd \xi  \quad\text{and}\quad
a :  \R^{3n} \mapsto \int A_2(x,y,\xi,\omega,\lambda) \1_{\{\lambda<0\}} \dd \lambda\dd \omega  . 
\]
The amplitude $a$ belongs to $S^1(\R^{3n})$ ($\|a\|_{\Co^k}$ are
bounded uniformly in $\hbar$ by differentiating under the integral). 
Consequently, by Lemma~\ref{lem:commsmooth}, 
\(
\|[f|_\Omega,K]\|^2=\O(\hbar^{1-n})
\)

\smallskip

We will show in the next section (Proposition~\ref{prop:err1}) that the integral $I_{1,\hbar} = \O(\hbar^{1-n})$  as $\hbar\to0$. This allows to conclude that
\[
\big\| [\widetilde{\Pi}, f|_\Omega] \big\|^2_{\J^2}  = \frac{1}{(2\pi\hbar)^{2n}}\int
e^{i\tfrac{(x-y)\cdot \xi}{\hbar }}(f|_\Omega(x)-f|_\Omega(y))^2  F(x,y,\xi)  \dd x \dd y \dd \xi  + \O(\hbar^{1-n})  
\]
which completes the proof of Proposition~\ref{prop:red}. 

\begin{proof}[Proof of Proposition~\ref{prop:symbolF}]~
  
First observe that according to \eqref{sigma1}, there exist two smooth functions 
$\nu : (x,y,\xi,\omega,\lambda) \in \R^{4n+1} \mapsto \S^{n-1}$ and $r : (x,y,\xi,\omega,\lambda) \in \R^{4n+1} \mapsto \R_+$  such that  
\begin{equation} \label{sigma2}
\sigma = 2 r \nu \cdot \xi , \qquad 
(r,\nu) = (R,\omega) + \O\big(|x-y|+|\xi|\big). 
\end{equation}
Then, we claim that for $\big\{(x,y,\xi)\in\R^{3n}: x\in\Omega' , |x-y|, |\xi|<2\delta\big\}$, we can make  a (smooth) change of coordinates
\[
(\lambda,\omega) \in\R \times \S^{n-1} 
\mapsto (\varsigma, \nu)  \in \R \times \S^{n-1} , \qquad \varsigma= -\lambda/r
\]
in the integral $F$. 
We denote by $\J: (x,y,\xi,\nu,\varsigma) \in \R^{4n+1} \to \R_+$ the corresponding Jacobian. 
Observe that if $(y,\xi) = (x,0)$, $\varsigma= -\lambda/\sqrt{\lambda-V(x)}$ and, by solving this quadratic equation and choosing the appropriate root, we obtain
\[
\lambda= -\varsigma\sqrt{-V(x)} +\O(\varsigma^2) ,\qquad \omega=\nu .
\]
Thus, in this case, 
\begin{equation} \label{J}
\J(x,x,0,\nu,\varsigma) = \sqrt{-V(x)} +\O(\varsigma) .
\end{equation}

Generally, $|\varsigma| \le \underline{\cst}$ (as $r \ge \underline{\cst}$ for $x\in \Omega'$ and $|y-x| , |\lambda|, |\xi|\le 2\delta$ with $\delta\ll \underline{\cst}$), so a similar computation using \eqref{sigma2} shows that $\J(x,y,\xi,\nu,\varsigma)$ is non-degenerate.
Hence, this change of coordinates is admissible and there is an amplitude $B \in S^1(\R^{4n+1})$ so that
\begin{equation} \label{F1}
F   = \int_{\S^{n-1}\times\R}
\hspace{-.5cm} 
B(x,y,\xi,\nu,\varsigma)  \1_{\{|\nu\cdot\xi| > \varsigma >0\}} \dd \nu\dd \varsigma .
\end{equation}

Moreover, by \eqref{sigma2}, the principal part of $B$ is given by, if $(y,\xi,\varsigma) = (x,0,0)$ with $x\in\Omega'$, $\nu\in\S^{n-1}$, 
\begin{equation} \label{Bprinc}
\begin{aligned}
B(x,x,0,\nu,0)|_{\hbar=0} &= A_2(x,x,0,\nu,0)|_{\hbar=0} \J(x,x,0,\nu,0) \\
&= |V(x)|^{\frac{n-1}{2}}/2
\end{aligned}
\end{equation}
using \eqref{symbolg1} and  \eqref{J}. 
\smallskip

Using spherical coordinates, writing  $\xi = R \theta$ with $(R, \theta)\in\R_+ \times\S^{n-1}$ and making a change of variable $\varsigma \leftarrow R \varsigma$ in \eqref{F1}, we obtain for $(x,y,\xi) \in\R^{3n}$
\begin{equation} \label{F2}
F(x,y,\xi)   = R \int_{\S^{n-1}\times\R}
\hspace{-.5cm} 
B(x,y,\xi,\nu,R\varsigma)  \1_{\{|\nu\cdot\vartheta| > \varsigma >0 \}} \dd \nu\dd \varsigma  .
\end{equation}
Then, since $B$ is smooth, by \eqref{Bprinc}, the principal part of $F$ satisfies on the diagonal $\{x=y\}$,
\[\begin{aligned}
F(x,x,\xi) &= R \int_{\S^{n-1}\times\R}
\hspace{-.5cm} \big(
B(x,x,0,\nu,0)+\O(R) \big) \1_{\{|\nu\cdot\vartheta| > \varsigma >0 \}} \dd \nu\dd \varsigma  \\ 
&= \frac{|\xi||V(x)|^{\frac{n-1}{2}}}2 \left\{  \int_{\S^{n-1}} |\nu_1| \dd \nu + \O(|\xi|) \right\} \\
&=\cst_{n-1} |\xi||V(x)|^{\frac{n-1}{2}} + \O(|\xi|^2)
\end{aligned}\]
uniformly for $x\in\Omega'$. Here, we used that for $n \ge 2$, 
\[
\frac 12 \int_{\S^{n-1}} |\nu_1|  \dd \nu  = \frac{|\S^{n-2}|}{n-1} = |B_{n-1}| 
\]
given the relationship between the volume of the unit sphere $\S^{n-1} =\partial B_n$ and the unit ball $B_n\subset \R^{n}$ (in particular, $|\S^0| = |B_1| =2$). 
In dimension $n=1$, the situation is special as $\nu\in\{\pm1\}$, so that $\cst_0 = 1$.
This establishes \eqref{Fpp}.

\smallskip

For $\theta\in\S^{n-1}$, let $\mathcal R_\theta : \S^{n-1} \to \S^{n-1}$ be the rotation so that $\mathcal R_\theta^*(\theta) =\mathrm{e}_1$ (the first coordinate vector in $\R^n$).
By \eqref{F2}, using the invariance under rotation of the Haar measure $\dd \nu$ on the sphere $\S^{n-1}$, we obtain 
\begin{equation} \label{F3}
F (x,y,\xi) = R \int_{\S^{n-1}\times\R}
\hspace{-.5cm} 
B(x,y,\xi,\mathcal R_\theta \nu,R\varsigma)  \1_{\{|\nu_1| > \varsigma >0 \}} \dd \nu\dd \varsigma  , \qquad \xi = R\theta , \, (R, \theta)\in\R_+ \times\S^{n-1}. 
\end{equation}
This function is clearly smooth in $(x,y,\xi)\in\R^{2n}\times\{\R^n\setminus 0\}$. 
In fact, for any $k\in\N_0$, as $R(\xi) \to0$, 
\[
\big\|\partial_\xi^k \mathcal R_{\theta(\xi)}  \big\|= \O_k(R^{-k}) , \qquad \big\| \partial_\xi^kR(\xi) \big\| = \O_k(R^{1-k}) .
\]
Indeed, $\theta \in \S^{n-1} \mapsto \mathcal R_\theta$ is smooth,
with  $\partial_\xi R(\xi) = \theta(\xi)$ for $\xi \neq 0$  and $\big\| \partial_\xi^k\theta(\xi) \big\| = \O_k(R^{-k})$.
Thus, by differentiating \eqref{F3} under the integral. we conclude that for any $k\in\N_0$,
\(
\big\| \partial_\xi^k F(x,y,\xi)\big\|=  \O_k(R^{1-k}) .
\)
as $R(\xi) \to0$. 
This completes the proof that $F\in\mathcal{F}(\R^{2n},\R^n)$.
\end{proof}

\subsubsection{Step 4: Control of  $I_{1,\hbar}$} \label{sec:error}

At this stage, it remains to show that  $I_{1,\hbar} =
\O(\hbar^{1-n})$, where we recall that $I_{1,\hbar}$ is given by \eqref{comm3}. By \eqref{comm3}, we can  work with $(\sigma,\lambda)$ fixed and it suffices to obtain the following bounds (We drop the dependency of $\Psi_1, \mathcal{Z},$ etc on $(\sigma,\lambda)$ for notational convenience. The proof of Proposition~\ref{prop:err1} relies on the stationary phase as formulated in Proposition~\ref{lem:statphase} which allows to deal with the parameters  $(\sigma,\lambda)$ in a uniform way.).

\begin{prop} \label{prop:err1}
Let $(\Psi_1,\mathcal{Z})$ be as in Section~\ref{sec:phase} and assume that $A\in S^1(\R^{4n})$ is supported on the set
$\big\{x\in\Omega', |x-y| \le \underline{\ell} , |\mathcal{Z}\big(x,\omega_1,\omega_2)| > \delta \big\}$ with  $\underline{\ell} \ll \delta \ll 1$. 
Then, as $\hbar\to0$, 
\[
\frac{1}{(2\pi\hbar)^{2n}}\int_{\omega_1,\omega_2\in\S^{n-1}}
\hspace{-1cm}e^{i\tfrac{\Psi_1(x,y,\omega_1,\omega_2)}{\hbar}}(f|_\Omega(x)-f|_\Omega(y))^2 A(x,y,\omega_1,\omega_2) \dd x \dd y \dd \omega_1\dd \omega_2 = \O(\hbar^{1-n}) . 
\]
\end{prop}

\begin{proof} 
Expanding the square, we need to deal with two integrals of the type
\begin{align}
\label{int1}
&\frac{1}{(2\pi\hbar)^{2n}}\int
e^{i\tfrac{\Psi_1(x,y,\omega_1,\omega_2)}{\hbar}} A_1(x,y,\omega_1,\omega_2) \1_\Omega(x)\dd x \dd y \dd \omega_1\dd \omega_2 \\
\label{int2}
&\frac{1}{(2\pi\hbar)^{2n}}\int
e^{i\tfrac{\Psi_1(x,y,\omega_1,\omega_2)}{\hbar}} A_2(x,y,\omega_1,\omega_2) \1_\Omega(x)\1_\Omega(y) \dd x \dd y \dd \omega_1\dd \omega_2 
\end{align}
where the amplitudes $A_1,A_2$ belong to $S^1(\R^{4n})$. There is a
minus sign, but we do not expect a cancelation in this regime.

According to \eqref{phaseZ},
\begin{equation} \label{partialPhase2}
\partial_y \Psi_1(x,y,\omega_1,\omega_2)=
- \mathcal{Z}(x,\omega_1,\omega_2)+ \O(x-y)  .
\end{equation}
Hence, if $A\in S^1(\R^{4n})$  is supported in 
$\big\{ |x-y| \le \underline{\ell} , |\mathcal{Z}\big(x,\omega_1,\omega_2)| > \delta \big\}$ with $\underline{\ell} \ll \delta \ll \underline{\cst}$, 
$\partial_y \Psi_1$ does not vanish on $\supp(A)$
and we can write 
\begin{equation*}
A(x,y,\omega_1,\omega_2) = 
\tfrac{1}{2\pi}J(x,y,\omega_1,\omega_2)  \cdot \partial_y\Psi_1(x,y,\omega_1,\omega_2) ,
\end{equation*}
where $J: \R^{4n}\to\R^n$ is also in $S^1$ with the same support as
$A$.

In these circumstances, we can repeat this procedure and integrate by parts with respect to $y$
arbitrarily often in \eqref{int1}, and obtain
\[
  \eqref{int1}=\O(\hbar^{\infty}).
\]
For \eqref{int2}, we integrate by parts once and find
that there are $B_1,B_2\in S^1(\R^{4n})$ so that
\begin{align}
\label{int3}
\eqref{int2}&=\frac{1}{(2\pi\hbar)^{2n-1}}\int_{\{\hat y\in\partial\Omega\}}
\hspace{-.5cm} e^{i\tfrac{\Psi_1(x,\hat y,\omega_1,\omega_2)}{\hbar}} B_1(x,\hat y,\omega_1,\omega_2) \1_\Omega(x)\dd x \dd \hat{y} \dd \omega_1\dd \omega_2 \\
&\label{int4}\quad
+\frac{1}{(2\pi\hbar)^{2n-1}}\int
e^{i\tfrac{\Psi_1(x,y,\omega_1,\omega_2)}{\hbar}} B_2(x,y,\omega_1,\omega_2) \1_\Omega(x)\1_\Omega(y) \dd x \dd y \dd \omega_1\dd \omega_2 .
\end{align}
Here, $B_1(\cdot) = \nu_\Omega(\hat{y}) \cdot J(\cdot) A_2(\cdot)$
where $\nu_\Omega$ is the exterior normal to  $\partial\Omega$, and we
integrate the
variable $\hat{y}$ with respect to the volume
measure over $\partial\Omega$ -- it is worth remembering here that
$\Omega$ has a smooth boundary.

The integral \eqref{int4} has the same form as the original integral
\eqref{int2} with an extra power of $\hbar$. 
Thus, if we show that $\eqref{int3} = \O(\hbar^{-n+1})$, by induction, we can conclude that also $\eqref{int2} =\O(\hbar^{-n+1})$.

\smallskip

Then we consider the integral \eqref{int3}. Similarly to \eqref{phaseZ}, the phase $\partial_x \Psi_1$ does not vanish on $\supp(B_1)$. Thus, integrating by part with respect to $x$, we obtain
\begin{equation} 
\label{int5}
\eqref{int3}=\frac{1}{(2\pi\hbar)^{2(n-1)}}\int_{\{\hat x,\hat y\in\partial\Omega\}}
\hspace{-.5cm} e^{i\tfrac{\Psi_1(\hat x,\hat y,\omega_1,\omega_2)}{\hbar}} B_3(\hat x,\hat y,\omega_1,\omega_2)\dd \hat{x} \dd \hat{y} \dd \omega_1\dd \omega_2  
+\frac{Q_\hbar}{2\pi\hbar}
\end{equation}
where $B_3\in S^1(\R^{4n})$ and $Q_\hbar$ is an integral of the same type as \eqref{int3}. 
Thus, by the same reasoning as above, it suffices to show that the
boundary integral in \eqref{int5} is $\O(\hbar^{1-n})$ to conclude that the original integral is also of order $\O(\hbar^{1-n})$. 

In dimension $n=1$, obviously $\eqref{int5}=\O(1)$, so this reasoning already shows that the integral in question is bounded, as required.

In what follows, we assume that $n\ge 2$ and focus on the boundary
integral \eqref{int5} for a fixed $\hat{x}\in \partial\Omega$. Recall
that the amplitude $B_3$ belongs to $S^1$ and it is supported on 
\(
\big\{\hat{x} , \hat{y} \in \partial\Omega :  |\hat{x} - \hat{y}| \le \underline{\ell} , |\mathcal{Z}\big(\hat{x},\omega_1,\omega_2)| > \delta \big\}.
\)
Let $\Pi_{\hat x}$ denote the orthogonal projection on $T_{\hat x}(\partial \Omega)$. On the above set ($\underline{\ell}$ is small), 
\[
(\hat{x}-\hat{y}) = \Pi_{\hat x}(\hat{x}-\hat{y}) +  \O\big(|\hat{x}-\hat{y}|^2\big) 
\]
then, by \eqref{phase2}, the phase satisfies
\begin{equation} \label{phase3}
\begin{aligned}
\Psi_1(\hat{x},\hat{y},\omega_1,\omega_2)&  = 
(\hat{x}-\hat{y}) \cdot  \mathcal{Z}(\hat{x},\omega_1,\omega_2) + \O\big(|\hat{x}-\hat{y}|^2\big) \\
& = \Pi_{\hat x}(\hat{x}-\hat{y}) \cdot  \Pi_{\hat x}\mathcal{Z}(\hat{x},\omega_1,\omega_2) +  \O\big(|\hat{x}-\hat{y}|^2\big) .
\end{aligned}
\end{equation}

We split the  integral \eqref{int5} in several parts depending on $\Pi_{\hat x}^\perp(\omega_j) $ for $j\in\{1,2\}$. 
Let
\[
\mathcal{S}_{\hat x,\underline{\ell}}^{\pm} := \big\{ \omega \in \S^{n-1}  :  \pm \Pi_{\hat x}^\perp(\omega)  > \ell \big\}  
\]
and  $\chi_{\hat x,\underline{\ell}}^{j} :  \S^{n-1} \to [0,1] $ for
$j\in\{0,+,-\}$ be smooth cutoffs such that \[\1_{\mathcal{S}_{\hat x,
      2\underline{\ell}}^{\pm}} \le \chi_{\hat x,
    \underline{\ell}}^{\pm} \le \1_{\mathcal{S}_{\hat x,
      \underline{\ell}}^{\pm}}\qquad \qquad 
\chi_{\hat x, \underline{\ell}}^{0} +\chi_{\hat x, \underline{\ell}}^{-} +\chi_{\hat x, \underline{\ell}}^{+} =\1_{\S^{n-1}}.\] 
We first consider the integrals
\begin{equation} \label{comm9}
(\hat{x},\omega_2) \mapsto \frac{1}{(2\pi\hbar)^{2(n-1)}}
\int \chi_{\hat x, \underline{\ell}}^{\pm}(\omega_1) e^{i\tfrac{\Psi_1(\hat{x},\hat{y},\omega_1,\omega_2)}{\hbar}} B_3(\hat{x},\hat{y},\omega_1,\omega_2)  \dd \hat y\dd \omega_1. 
\end{equation}
Since the map $ \S^{n-1}\ni\omega \mapsto \Pi_{\hat x}\omega$ is a diffeomorphism 
from $\mathcal{S}_{\hat x, \underline{\ell}}^{\pm}$ to its image, we can make a change of variables
\[
\hat y \to u= \Pi_{\hat x}(\hat y - \hat x) , \qquad \omega_1 \to v= \Pi_{\hat x}(\omega_1) , 
\]
in \eqref{comm9}. 
We obtain 
\[
\eqref{comm9} =\frac{1}{(2\pi\hbar)^{2(n-1)}}  \int e^{i\tfrac{\Psi_2(\hat{x}, \omega_2, u,v)}{\hbar}} B_4(\hat x ,\omega_2, u,v)  \dd u \dd v
\]
where the amplitude $B_4$ belongs to $S^1$ and the new phase $\Psi_2$ satisfies, by \eqref{phase3} and \eqref{Z1},
\[
\Psi_2(\hat{x}, \omega_2, u,v)=  u \cdot \big(R_+  v - R_- \Pi_{\hat x}(\omega_2) \big)  +\O(u^2) ,
\qquad R_{\pm} = \sqrt{\lambda-V(\hat{x}) \pm \sigma} .
\]
Any critical point of $\Psi_2$ with respect to $(u,v)$ satisfies $u=0$ and the Hessian is non-degenerate with determinant $R_+^{2(n-1)}$ and $R_\pm \ge \sqrt{\underline{\cst}}$  
(there is at most one critical point given by $v=\Pi_{\hat x}(\omega_2)R_-/R_+$). 
Thus, by Proposition~\ref{lem:statphase} with $d=2(n-1)$, one has the
following estimate, uniformly in $ (\hat x, \omega_2) \in \partial\Omega\times\S^{n-1}$ (as well as the auxiliary parameters  $(\sigma,\lambda)\in\R^2$):
\[
\eqref{comm9} = \O(\hbar^{n-1}) . 
\]
By the same argument, with the same uniformity, 
\[
\int \chi_{\hat x, \underline{\ell}}^{\pm}(\omega_2) e^{i\frac{\Psi_1(\hat{x},\hat{y},\omega_1,\omega_2)}{\hbar}} B_3(\hat{x},\hat{y},\omega_1,\omega_2)  \dd \hat y\dd \omega_2  = \O(\hbar^{n-1}) .
\]

Hence, we are left to deal with the integral
\[
\hat x \mapsto \frac{1}{(2\pi\hbar)^{2(n-1)}}  \int \chi_{\hat x, \underline{\ell}}^{0}(\omega_1)  \chi_{\hat x, \underline{\ell}}^{0}(\omega_2) e^{i\frac{\Psi_1(\hat{x},\hat{y},\omega_1,\omega_2)}{\hbar}} B_3(\hat{x},\hat{y},\omega_1,\omega_2)  \dd \hat y\dd \omega_1 \dd \omega_2 .
\]
Since ${\hat x} \mapsto \Pi_{\hat x}$ is smooth (we assume that $\partial\Omega$ is smooth) and $\partial_{\hat y}=  \Pi_{\hat y}(\partial_y)+\O(x-y)$, according to \eqref{partialPhase2}, 
\[
\partial_{\hat y} \Psi_1(\hat{x},\hat{y},\omega_1,\omega_2)
= \Pi_{\hat x}\mathcal{Z}(\hat{x},\omega_1,\omega_2)  + \O(\underline{\ell}) 
\qquad\text{for $\big\{ \hat y \in \partial\Omega :  |\hat y- \hat x| \le \underline{\ell} \big\}$}.
\]

Thus, on the support of the previous integral,  $\big\{ | \Pi_{\hat x}^\perp(\omega_1)| , |\Pi_{\hat x}^\perp(\omega_2)| \le \underline{\ell} ,  |\mathcal{Z}\big(\hat x,\omega_1,\omega_2)| > \delta \big\}$ with $\underline{\ell}\ll\delta$.
Since $\mathcal{Z}$ is a linear combination of $(\omega_1,\omega_2)$, see~\eqref{Z2}, we have $ \Pi_{\hat x}  \mathcal{Z}(\hat x,\omega_1,\omega_2) = \mathcal{Z}(\hat x,\omega_1,\omega_2) + \O(\underline{\ell})$ and then
\[
\big|\partial_{\hat y} \Psi_1(\hat{x},\hat{y},\omega_1,\omega_2)\big| \ge \delta/2 . 
\]
This shows that the phase has no critical point in the previous integral. Hence, by Lemma~\ref{lem:statphase} (non-stationary phase version), we conclude that 
uniformly in $\hat{x}\in\partial\Omega$ (as well as the auxiliary parameters  $(\sigma,\lambda)\in\R^2$), 
\[
\int \chi_{\hat x, \underline{\ell}}^{0}(\omega_1)  \chi_{\hat x, \underline{\ell}}^{0}(\omega_2) e^{i\tfrac{\Psi_1(\hat{x},\hat{y},\omega_1,\omega_2)}{\hbar}} B_3(\hat{x},\hat{y},\omega_1,\omega_2)  \dd \hat y\dd \omega_1 \dd \omega_2 
= \O(\hbar^\infty).
\]

Altogether, this establishes that the integrals \eqref{int5}, \eqref{int3} as well as the original integral are all $\O(\hbar^{1-n})$ with the required uniformity in $(\sigma,\lambda)\in\R^2$. 
Going back to \eqref{comm3}, we conclude that $I_{1,\hbar} = \O(\hbar^{1-n})$ and this also finalizes the proof of Proposition~\ref{prop:red}. 
\end{proof}

\subsection{Case of a contractible open set} \label{sec:int}

It remains to study integrals of the form given by the right-hand-side
of Proposition \ref{prop:red}. We first do so in the case where
$\Omega$ is topologically simple: if it is diffeomorphic to the unit
ball, then we can proceed by a change of variables.

We denote $B = B_n= \{x\in\R^n : |x| <1\}$. 

\begin{prop} \label{prop:intasymp}
Let $\Omega \Subset \R^n$ be a  bounded open set with a smooth
boundary; suppose that there is a $C^\infty$-diffeomorphism $\varphi : \R^n \to \R^n$ with $\varphi^{-1}(\Omega)=B$.  
Let $F \in \mathcal F(\R^{2n},\R^n)$ according to~Definition~\ref{class:F} and assume that as $r\to0$,
\begin{equation} \label{condR}
F(x,x,r\omega) = rR(x) + \O(r^2)
\end{equation}
uniformly for $(x,\omega)\in \Omega'\times \S^{n-1}$.
Let $f\in C^\infty(\R^n)$ and denote $\mathrm{f}= f\1_{\Omega}$.
Then, as $\hbar\to0$, 
\[
\int e^{i\frac{(x-y)\cdot\xi}{\hbar}} F(x,y,\xi) (f|_\Omega(x)-f|_\Omega(y))^2 \dd \xi \dd x \dd y 
=  (2\pi\hbar)^{n+1} \bigg( \frac{\log \hbar^{-1}}{-\pi^2}  \int_{\partial\Omega} R(\hat x) f(\hat x)^2 \dd \hat x + \O(1) \bigg) 
\]
where $\dd\hat x$ denotes the volume measure on the boundary $\partial\Omega$. 
\end{prop}

Without loss of generality, we assume that $\varphi = \mathrm{I}$ on $\R^n \setminus \Omega'$. Throughout this section, we denote 
\begin{equation} \label{oscint}
L_\hbar  : = \int e^{i\frac{(x-y)\cdot\xi}{\hbar}} F(x,y,\xi) (f|_\Omega(x)-f|_\Omega(y))^2 \dd \xi \dd x \dd y . 
\end{equation}

The proof of Proposition~\ref{prop:intasymp} will be divided in three steps.
\begin{itemize}[leftmargin=*] 
\item In Section~\ref{sec:chv}, using the map $\varphi$ as a change of coordinates, we show that we can reduce \eqref{oscint} to the case where $\Omega = B$.
\item In Section~\ref{sec:ball}, we obtain the asymptotics of
  \eqref{oscint} when $\Omega = B$ by using the expression of the
  Fourier transform $\widehat{\1_B}$ in terms of a Bessel function, cf.~Lemma~\ref{lem:Bessel}.
\item Combining these results, we complete the proof in Section~\ref{sec:proofintasymp}.
\end{itemize}

These computations rely importantly on the estimates from Section~\ref{sec:est} for integrals involving $\1_\Omega$. 

\subsubsection{Change of variables} \label{sec:chv}

\begin{prop}\label{prop:smooth-diffeo}
Under the assumptions of Proposition~\ref{prop:intasymp}, define a map
\[
\Phi :(x,y,\xi) \in \R^{3n} \mapsto \big(\varphi(x),\varphi(y), D\varphi(\tfrac{x+y}{2})^{-*}\xi\big) 
\]
and let $G := \J_\Phi\,  F\circ\Phi$. Then $G$ is also in the class $\mathcal F$
and, as $\hbar\to 0$, 
\[\begin{aligned}
\int e^{i\frac{(x-y)\cdot\xi}{\hbar}} F(x,y,\xi) (f|_\Omega(x)-f|_\Omega(y))^2 \dd \xi \dd x \dd y 
= \int e^{i\frac{(x-y)\cdot\xi}{\hbar}} G(x,y,\xi) (\mathrm{g}_B(x)-\mathrm{g}_B(y))^2\dd \xi \dd x \dd y 
+ \O(\hbar^{n+1}) 
\end{aligned}\]
where $\mathrm{g}_B = f|_\Omega(\varphi)=f(\varphi)\1_B$.
\end{prop}

\begin{proof}
By construction, the map $\Phi :\R^{3n}\to\R^{3n}$ is a $C^\infty$-diffeomorphism and it is straightforward to check that $G$ is also in the class $\mathcal F$, cf.~Definition~\ref{class:F} (since $\varphi = \mathrm{I}$ on $\R^n \setminus \Omega'$, $G$ has compact support).
In particular, for every $k\in\N_0$ with $k\le n+3$,
\begin{equation} \label{GclassF}
\big\|\partial_\xi^k G(x,y,\xi)\big\| \le |\xi|^{1-k} \varkappa_1(\xi)  \varkappa_2(x,y)  
\end{equation}
where $\varkappa_1\in C^\infty_c(\R^n)$ and $\varkappa_2 \in C^\infty_c(\R^{2n})$.

By a Taylor expansion, we have for $(x,y,\xi) \in \R^{3n}$, 
\[
\big(\varphi(x)-\varphi(y)\big)\cdot D\varphi(\tfrac{x+y}{2})^{-*}  \xi =(x-y)\cdot \xi+Q(\tfrac{x+y}{2},x-y)\xi
\]
where $Q : \R^{2n} \to \R^{n}$ is $C^{\infty}_c$ and vanishes at zero
along with its two first derivatives:
\begin{equation}
  \label{eq:Qorder4}
  |Q(v,u)|\leq C|u|^3.
\end{equation}
Let us denote
\[
\breve{G}(\tfrac{x+y}{2},x-y,\xi) = G(x,y,\xi) ,
\qquad\text{and}\qquad
\Gamma(\tfrac{x+y}{2},x-y) = (\mathrm{g}_B(x)-\mathrm{g}_B(y))^2. 
\]

Let $\chi\in C^{\infty}_c(\R^n)$ be equal to 1 near $0$. By a change of variable using $\Phi$.
\begin{align} \label{intchvar}
\int e^{i\frac{(x-y)\cdot\xi}{\hbar}} F(x,y,\xi) (f|_\Omega(x)-f|_\Omega(y))^2 \dd \xi \dd x \dd y 
 =& \int e^{i\frac{u\cdot \xi+Q(v,u)\xi}{\hbar}} \breve{G}(v,u,\xi)\Gamma(v,u)\dd \xi \dd v\dd u  \\
\notag
=& \int e^{i\frac{(x-y)\cdot\xi}{\hbar}} G(x,y,\xi)
   (\mathrm{g}_B(x)-\mathrm{g}_B(y))^2\dd \xi \dd x \dd y \\ \notag &+ I_\hbar +\O(\hbar^{n+1}) 
\end{align}
where 
\[
I_{\hbar} = \int  (1-\chi(\tfrac{\xi}{\hbar})) e^{i\frac{u\xi}{\hbar}}\Big(e^{i\frac{Q(v,u)\xi}{\hbar}}-1\Big)  \breve{G}(v,u,\xi)\Gamma(v,u)\dd \xi \dd v\dd u .
\]
Here, we used Lemma~\ref{lem:trunc} together with the fact that $G \in
\mathcal F$ to  introduce the cutoff $1-\chi(\tfrac{\xi}{\hbar})$ in
the integral \eqref{intchvar}, up to an error $\O(\hbar^{n+1})$. 

Now, our goal is now to show that $I_\hbar = \O(\hbar^{n+1})$ as $\hbar\to0$.
The argument is similar to the proof of Lemma~\ref{lem:1}, albeit the
induction process is more technical. We will integrate by parts
repeatedly in the variable $u$, stopping every time a derivative hits
either $\chi(\tfrac{\xi}{\hbar})$ or $e^{i\frac{Q}{\hbar}}$. To this
end, we introduce the following integrals for $k\in\N_0$ with $k\le n+2$:
\[\begin{aligned}
J_{k,\hbar} &=2i \int e^{i\frac{u\xi}{\hbar}}  u\cdot\partial_\xi \chi(\tfrac{\xi}{\hbar})\Big(e^{i\frac{Q(v,u)\xi}{\hbar}}-1\Big) \frac{u^{\otimes k}\cdot\partial_\xi^{k} \breve{G}(v,u,\xi)}{|u|^{2(k+1)}} \Gamma(u,v)\dd\xi\dd u \dd v\, ,  \\
I_{k,\hbar}  &=  \int (1-\chi(\tfrac{\xi}{\hbar})) e^{i\frac{u \xi+Q(v,u)\xi}{\hbar}}
\frac{Q(v,u)u}{|u|^2} \frac{u^{\otimes k}\cdot\partial_\xi^{k} \breve{G}(v,u,\xi)}{|u|^{2k}} \Gamma(u,v)\dd\xi\dd u \dd v \, .
\end{aligned} \]
Provided we can show that all these integrals are finite, one has
\begin{equation} \label{Imain}
\begin{aligned}
&I_{\hbar}\\ & =  i\hbar  \int e^{i\frac{u\xi}{\hbar}} \frac{u\cdot\partial_\xi}{|u|^2} \left\{ (1-\chi(\tfrac{\xi}{\hbar}))\Big(e^{i\frac{Q(v,u) \xi}{\hbar}}-1\Big)  \breve{G}(v,u,\xi)\right\}\Gamma(u,v)\dd\xi\dd u \dd v  \\
& =  i\hbar \int(1 - \chi(\tfrac{\xi}{\hbar})) 
e^{i\frac{u\xi}{\hbar}}\Big(e^{i\frac{Q(v,u)
    \xi}{\hbar}}-1\Big)\frac{u\cdot\partial_\xi
  \breve{G}(v,u,\xi)}{|u|^2} \Gamma(u,v)\dd\xi\dd u \dd v  -
I_{0,\hbar}-J_{0,\hbar}\\
&=(i\hbar)^2 \int(1 - \chi(\tfrac{\xi}{\hbar})) 
e^{i\frac{u\xi}{\hbar}}\Big(e^{i\frac{Q(v,u)
    \xi}{\hbar}}-1\Big)\frac{u^{\otimes 2}\cdot\partial^2_\xi
  \breve{G}(v,u,\xi)}{|u|^4} \Gamma(u,v)\dd\xi\dd u \dd v  -
I_{0,\hbar}-J_{0,\hbar}-i\hbar(I_{1,\hbar}+ J_{1,\hbar})\\
&=\cdots\\
&=(i\hbar)^{n+3}\int(1 - \chi(\tfrac{\xi}{\hbar})) 
e^{i\frac{u\xi}{\hbar}}\Big(e^{i\frac{Q(v,u)
    \xi}{\hbar}}-1\Big)\frac{u^{\otimes (n+3)}\cdot\partial^{n+3}_\xi
  \breve{G}(v,u,\xi)}{|u|^{2(n+3)} }\Gamma(u,v)\dd\xi\dd u \dd v  - \sum_{k=0}^{n+2}(i\hbar)^k(I_{k,\hbar}+J_{k,\hbar}).
\end{aligned}
\end{equation}
Let us first bound the $J_{k,\hbar}$ integrals. 
Using \eqref{GclassF} and that $\Big|e^{i\frac{Q(v,u) \xi}{\hbar}}-1\Big|  \le  \tfrac{C}{\hbar} |u|^3 |\xi|$
for $(u,v) \in\R^{2n}$,  for any $k\in\N$ with $k\le n+2$,
\[\begin{aligned}
|J_{k,\hbar} |  & \le \frac{C}{\hbar} \int |\xi| \big| \partial_\xi \chi(\tfrac{\xi}{\hbar})\big|  \frac{\|\partial_\xi^{k} \breve{G}(v,u,\xi)\| }{|u|^{k-2}} \Gamma(u,v)\dd\xi\dd u \dd v \\
&\le C \hbar^{n+1-k}   \int |\partial_\xi \chi(\xi)| |\xi|^{2-k}  \frac{ (\mathrm{g}_B(x)-\mathrm{g}_B(y))^2}{|x-y|^{k-2}} \varkappa_2(x,y) \dd\xi\dd u \dd v .
\end{aligned}\]
where we rescaled $\frac{\xi}{\hbar}\leftarrow \xi$ since $|\partial_\xi\chi|$ has compact support in $\R^n \setminus\{0\}$.
By Proposition~\ref{prop:key}, these integrals are finite so that for every $k\in\N_0$ with $k\le n+2$, 
\begin{equation}\label{eq:bound_Jk}
J_{k,\hbar}= \O(\hbar^{1+n-k}) .
\end{equation}
To estimate $I_{k,\hbar}$, we reverse the change of variables $\Phi$,
and obtain
\begin{equation}\label{Ierror}
  I_{k,\hbar}=\int
  (1-\chi(\tfrac{A(x,y)\xi}{\hbar}))e^{i\frac{(x-y)\cdot
      \xi}{\hbar}}L_k(x,y,\xi)(f|_{\Omega}(x)-f|_{\Omega}(y))^2\dd x
  \dd y \dd \xi
\end{equation}
where $A:\R^{n\times n}\to {\rm GL}_n$ is a smooth map, and $L_k$ is the image by $\Phi^{-1}$ of the
function
\[
(x,y,\xi) \in\R^{3n} \mapsto \frac{Q(\frac{x+y}{2},x-y)(x-y)}{|x-y|^2} \frac{(x-y)^{\otimes k}\cdot\partial_\xi^{k}\breve{G}(x,y,\xi)}{|x-y|^{2k}}|D\phi(x)||D\phi(y)|\det(A(x,y))^{-1}.
\]
Using \eqref{GclassF} and \eqref{eq:Qorder4}, together with the fact
that $\Phi$ maps the diagonal $x=y$ to itself, we verify that  for
every $j \in\N$,
\[
\big|  \partial_\xi^j L_k(x,y,\xi) \big| \le C  |x-y|^{2-k} |\xi|^{1-k-j}  \varkappa(\xi,x,y) .
\]
where $\varkappa$ is compactly supported in a neighbourhood of
$x=y$. Hence, by Proposition \ref{lem:1}, for every $k\geq n+2$,
\[
I_{k,\hbar}= \O(\hbar^{1+n-k}) .
\]

Going back to \eqref{Imain} and using the previous estimates with $k=0$, we obtain 
\[
I_{\hbar}=
i\hbar \int  \overline\chi(\tfrac{\xi}{\hbar}) 
e^{i\frac{u\xi}{\hbar}}\Big(e^{i\frac{Q(v,u) \xi}{\hbar}}-1\Big)\frac{u\cdot\partial_\xi \breve{G}(v,u,\xi)}{|u|^2} \Gamma(u,v)\dd\xi\dd u \dd v   + \O(\hbar^{1+n}) .
\]
The leading term is of the form as $I_{\hbar}$, so we can repeat the process. By induction, after $k\in\N$ integrations by part, we obtain
\[
I_{\hbar}=
(i\hbar)^k \int  \overline\chi(\tfrac{\xi}{\hbar}) 
e^{i\frac{u\xi}{\hbar}}\Big(e^{i\frac{Q(v,u) \xi}{\hbar}}-1\Big)\frac{u^{\otimes k}\cdot\partial_\xi^k \breve{G}(v,u,\xi)}{|u|^{2k}} \Gamma(u,v)\dd\xi\dd u \dd v   
- \sum_{j<k}  (i\hbar)^j \big( I_{j,\hbar}+J_{j,\hbar} \big) . 
\]
The error terms are all controlled as above. Hence, taking $k=n+3$, we obtain
\[ \begin{aligned}
| I_{\hbar} | & \le  \hbar^{n+3} \int  \overline\chi(\tfrac{\xi}{\hbar}) 
e^{i\frac{u\xi}{\hbar}}\Big|e^{i\frac{Q(v,u) \xi}{\hbar}}-1\Big|\frac{\|\partial_\xi^{n+3} \breve{G}(v,u,\xi)\|}{|u|^{n+3}} \Gamma(u,v)\dd\xi\dd u \dd v  + \O(\hbar^{1+n}) \\
&\le C \hbar^{n+2} \int  \overline\chi(\tfrac{\xi}{\hbar}) |\xi|^{-1-n}   \frac{(f|_\Omega(x)-f|_\Omega(y))^2}{|x-y|^{n}} \varkappa_2(x,y)  \dd x \dd y \dd \xi + \O(\hbar^{1+n}) \\
& = \O(\hbar^{1+n})
\end{aligned}\]
where we used again Proposition~\ref{prop:key} to bound the last integral.

In the end, going back to formula \eqref{intchvar}, we conclude that
\[
\int e^{i\frac{(x-y)\cdot\xi}{\hbar}} F(x,y,\xi) (f|_\Omega(x)-f|_\Omega(y))^2 \dd \xi \dd x \dd y 
= \int e^{i\frac{(x-y)\cdot\xi}{\hbar}} G(x,y,\xi) (\mathrm{g}_B(x)-\mathrm{g}_B(y))^2\dd \xi \dd x \dd y +O(\hbar^{n+1}) 
\]
as claimed.
\end{proof}

\subsubsection{Asymptotics in case of the ball} \label{sec:ball}

The goal of this section is to obtain the following asymptotics. 

\begin{prop} \label{prop:Ired}
Let $f\in C^\infty(\R^n)$, $F \in \mathcal F(\R^{2n},\R^n)$ and denote 
\begin{equation} \label{def:q}
q(r) = \frac 1r \int_{\S^{n-1}}  F(\omega,\omega,r\omega) f(\omega)^2 \dd \omega , \qquad r>0 .
\end{equation}
This function is continuous on $\R_+$, with compact support, and we assume that $q(r) \to Q$  as $r\to 0$ and that
\begin{equation} \label{condq}
\int_0^1 \bigg| \frac{q(r)-Q}{r} \bigg| \dd r <\infty . 
\end{equation}
Then, by \eqref{oscint} with $\Omega = B$, 
\[
L_\hbar
= -4 \hbar^{n+1}  \big(  (2\pi)^{n-1} Q \log \hbar^{-1} +\O(1) \big) .
\]
\end{prop}

We first simplify the integral $L_\hbar$ of \eqref{oscint}, using the estimates from Section~\ref{sec:est}. 

\begin{lem} \label{lem:Ired} 
Let $f\in C^\infty(\R^n)$, $F \in \mathcal F(\R^{2n},\R^n)$ and  $g(x,\xi) := F(x,x,\xi)f(x)^2$ for $(x,\xi)\in\R^{2n}$.
Then  $g \in \mathcal F(\R^n,\R^n)$ and
\[
L_\hbar = -2  \int e^{i\frac{(x-y)\cdot\xi}{\hbar}}  g(\tfrac{x+y}{2},\xi)  \1_{B}(x)  \1_{B}(y) \dd \xi \dd x \dd y  + \O(\hbar^{n+1}) .  
\]
\end{lem}

\begin{proof}
By symmetry of the integral in question, we can assume that 
\[
L_\hbar= 2 \int e^{i\frac{(x-y)\cdot\xi}{\hbar}} F(x,y,\xi) (\mathrm{f}_B(x)-\mathrm{f}_B(y))\mathrm{f}_B(x) \dd \xi \dd x \dd y .
\]
Then, we can write
\begin{align}
\label{Ired1}
\frac{L_\hbar}2=  \int e^{i\frac{(x-y)\cdot\xi}{\hbar}} F(x,y,\xi)  \big(f(x)-f(y)\big) \mathrm{f}_{B}(x) \dd \xi \dd x \dd y \\
\label{Ired2} \quad + \int e^{i\frac{(x-y)\cdot\xi}{\hbar}} G(x,y,\xi) \1_{B}(x)  \1_{B^c}(y)   \dd \xi \dd x \dd y  
\end{align}
where $ G(x,y,\xi) := F(x,y,\xi) f(x)f(y)$. 

\smallskip

We begin by showing that \eqref{Ired1} is negligible.
Since $(x,y)  \mapsto F(x,y,\xi) $ is $C^\infty_c(\R^{2n})$ for $\xi\in\R^n$, by a Taylor expansion, 
\[
\big(f(x)-f(y)\big) F(x,y,\xi) 
=    (x-y)\cdot a(\tfrac{x+y}{2},\xi)   +  b(x,y,\xi)
\]
where  $a \in \mathcal F$, $b \in \mathcal F$ and 
\[
a(x,\xi) = \partial f (x) F(x,x,\xi) , \qquad   b(x,y,\xi) =\O\big(|x-y|^3\big) .
\]
In particular, the function $(x,y,\xi) \mapsto  b(x,y,\xi) \mathrm{f}_B(x)$ satisfies the assumptions of  Lemma~\ref{lem:1}. This implies that 
\[
\eqref{Ired1} = 2  \int e^{i\frac{(x-y)\cdot\xi}{\hbar}}  (x-y)\cdot a(\tfrac{x+y}{2},\xi) \mathrm{f}_B(x) \dd \xi \dd x \dd y  + \O(\hbar^{n+1}) 
\]
Then, we perform an integration by part with respect to $\xi$ and a change of variables $\frac{x+y}{2} \leftarrow v$, $(x-y) \leftarrow u$, we obtain 
\[\begin{aligned}
\int e^{i\frac{(x-y)\cdot\xi}{\hbar}}  (x-y)\cdot a(\tfrac{x+y}{2},\xi) \mathrm{f}_B(x) \dd \xi \dd x \dd y
& = (i\hbar)  \int e^{i\frac{(x-y)\cdot\xi}{\hbar}}  a'(\tfrac{x+y}{2},\xi) \mathrm{f}_B(x) \dd \xi \dd x \dd y \\
& =  (i\hbar)  \int e^{i\frac{u\cdot\xi}{\hbar}}  a'(v,\xi) \mathrm{f}_B(v+\tfrac u2) \dd \xi \dd u \dd v
\end{aligned}\]
where $a' = \operatorname{div}_\xi a $ is $L^\infty$ with compact support. 
Using the Fourier transform, computing the integral over $u$, and then over $v$, we obtain
\begin{align} \notag
\int e^{i\frac{(x-y)\cdot\xi}{\hbar}}  (x-y)\cdot a(\tfrac{x+y}{2},\xi) \mathrm{f}_B(x) \dd \xi \dd x \dd y
& = i\hbar (2\pi)^{\frac n2}  \int e^{-i\frac{v\cdot\xi}{\hbar}} a'(v,\xi)  \overline{\widehat{\mathrm{f}_B}(\tfrac{2\xi}{\hbar})} \dd \xi \dd v \\
&\label{Ired3} = i\hbar (2\pi)^{n} \int \widehat{a'}(\tfrac{\xi}{\hbar},\xi)  \overline{\widehat{\mathrm{f}_B}(\tfrac{2\xi}{\hbar})} \dd \xi 
\end{align}
where for $(\zeta,\xi) \in\R^{2n}$ with $\xi \neq 0$, 
\[
\widehat{a'}(\zeta,\xi) = \frac{1}{(2\pi)^{\frac n2}} \int e^{-i v\cdot\zeta} a'(v,\xi) \dd v
= \operatorname{div}_\xi\bigg( \frac{1}{(2\pi)^{\frac n2}} \int e^{-i v\cdot\zeta} a(v,\xi) \dd v \bigg) .
\]
Since $a \in \mathcal F$, $\zeta \mapsto \widehat{a'}(\zeta,\xi)$ is in the Schwartz class for $\xi\neq 0$ and $\xi \mapsto \widehat{a'}(\zeta,\xi)$ is  $L^\infty$. 
In the end, by scaling and Cauchy-Schwarz, 
\[ 
\int \widehat{a'}(\tfrac{\xi}{\hbar},\xi)  \overline{\widehat{\mathrm{f}_B}(\tfrac{2\xi}{\hbar})} \dd \xi 
= \O\big(\hbar^{n} \|a'\|_{L^2\times L^\infty} \|\mathrm{f}_B\|_{L^2} \big) ,
\]
and we conclude that $\eqref{Ired1}  =\O(\hbar^{n+1})$. 

\medskip 

Going back to \eqref{Ired2}, we have
\[
G(x,y,\xi)  = g(\tfrac{x+y}{2},\xi) + e(x,y,\xi)   \qquad
\text{where $g \in \mathcal F$, $e \in \mathcal F$ and } e(x,y,\xi) =\O\big(|x-y|^2\big),
\]
so that the function 
\[
(x,y,\xi) \mapsto e(x,y,\xi) \1_{B}(x)  \1_{B^c}(y)
\]
satisfies the assumptions of  Lemma~\ref{lem:1} -- the condition \eqref{conda1} for $j\le n+2$  follow directly from \eqref{bdsub} with $\Omega=B$.
Altogether, this implies that 
\begin{equation} \label{Ired4}
\frac{L_\hbar}2 = \int e^{i\frac{(x-y)\cdot\xi}{\hbar}}  g(\tfrac{x+y}{2},\xi)  \1_{B}(x)  \1_{B^c}(y) \dd \xi \dd x \dd y  + \O(\hbar^{n+1}) .
\end{equation}
Moreover, as in \eqref{Ired3}, we have 
\[
\int e^{i\frac{(x-y)\cdot\xi}{\hbar}}  g(\tfrac{x+y}{2},\xi)  \1_{B}(x) \dd \xi \dd x \dd y 
= (2\pi)^{n} \int \widehat{g}(\tfrac{\xi}{\hbar},\xi)  \overline{\widehat{\1_B}(\tfrac{2\xi}{\hbar})} \dd \xi 
\]
where  for $(\zeta,\xi) \in\R^{2n} \mapsto \widehat{g}(\zeta,\xi) $ satisfies $|\widehat{g}(\zeta,\xi)| \le |\xi| \varkappa(\zeta)$ where $\varkappa$  is in the Schwartz class; since  $g \in \mathcal F$. Using this bound and the fact that $\widehat{\1_B} \in L^\infty$, we obtain 
\[
\int e^{i\frac{(x-y)\cdot\xi}{\hbar}}  g(\tfrac{x+y}{2},\xi)  \1_{B}(x) \dd \xi \dd x \dd y = \O(\hbar^{n+1}) .
\]
Hence, the claim follows from \eqref{Ired4}. 
\end{proof}

The proof relies on the following explicit asymptotics.

\begin{lem} \label{lem:Bessel}
Let $\J_\nu$ denote the Bessel function of the first kind for
$\nu\in\R_+$. As $\xi\to +\infty$,
\begin{equation} \label{BesselJ}
\widehat{\1_B}(\xi) = \frac{\mathrm{J}_{n/2}(|\xi|)}{ |\xi|^{n/2}} = 
\frac{2\cos(|\xi|+c_n) + \O(|\xi|^{-1})}{\sqrt{2\pi} |\xi|^{(n+1)/2}}.
\end{equation}
\end{lem}

\begin{proof}
The relationship between Bessel functions and the Fourier transform of
$\1_B$ where $B$ is the unit ball in $\R^n$ is classical: see
e.g.~\cite{deleporte_universality_2024}, formula (A.5).
The asymptotics of Bessel functions are also classical, see \cite[\href{https://dlmf.nist.gov/10.17.3}{formula (10.17.3)}]{NIST:DLMF}.
\end{proof}

\begin{proof}[Proof of Proposition~\ref{prop:Ired}]~
  
Starting from Lemma~\ref{lem:Ired} and making a change of variables $\frac{x+y}{2} \leftarrow v$, $(x-y) \leftarrow u$,  we have
\[
L_\hbar =  -2\int e^{i\frac{u \cdot\xi}{\hbar}} g(v,\xi)  \1_{B}(v+\tfrac u2) \1_{B}(v-\tfrac u2)
\dd \xi \dd u \dd v   + \O(\hbar^{n+1}) .  
\]
Since $g \in \mathcal F$, by  Lemma~\ref{lem:trunc},  we can introduce a cutoff $ \overline\chi(\tfrac{\xi}{\hbar}) $ in the previous integral, up to an error $\O(\hbar^{n+1})$. 
Moreover, computing this Fourier transform as a convolution (with our convention, $(2\pi)^{\frac n2}\widehat{uv} = \widehat{u} *\widehat{v} $ for functions $u,v \in L^2(\R^n)$), 
\[
\int e^{i\frac{u \cdot\xi}{\hbar}}  \1_{B}(v+\tfrac u2) \1_{B}(v-\tfrac u2) \dd u 
= 2^n \int e^{-2i v \cdot \zeta}  \widehat{\1_B}(\zeta+\tfrac\xi\hbar)\widehat{\1_B}(\tfrac\xi\hbar-\zeta) \dd \zeta  
=  \int e^{-i v \cdot \zeta}  \widehat{\1_B}(\tfrac\zeta2+\tfrac\xi\hbar)\widehat{\1_B}(\tfrac\zeta2-\tfrac\xi\hbar) \dd \zeta  
\]
This implies that 
\[
L_\hbar = - 2  \int  \overline\chi(\tfrac{\xi}{\hbar})  e^{-i v \cdot \zeta}   g(v,\xi)   \widehat{\1_B}(\tfrac\zeta2+\tfrac\xi\hbar)\widehat{\1_B}(\tfrac\zeta2-\tfrac\xi\hbar)  \dd v \dd \zeta \dd \xi
+ \O(\hbar^{n+1}).
\]
This integral is well-defined since $\1_B\in L^2(\R^n)$ and $g$ has compact support on $\R^{2n}$. 
Moreover we can take the Fourier transform of the smooth function $v \in \R^n \mapsto g(v,\xi)$ for a fixed $\xi$, we obtain
\[ \begin{aligned}
L_\hbar =  -2 (2\pi)^{\frac n2}  \int  \overline\chi(\tfrac{\xi}{\hbar})   \widehat g(\zeta,\xi)   \widehat{\1_B}(\tfrac\zeta2+\tfrac\xi\hbar)\widehat{\1_B}(\tfrac\zeta2-\tfrac\xi\hbar) \dd \zeta \dd \xi
+ \O(\hbar^{n+1})  
\end{aligned}\]
where for $k\in\N$, $\zeta\in\R^n$, $r>0$, 
\begin{equation}\label{hatgest}
\sup_{\xi\in\R^n}|\widehat g(\zeta,\xi)|  \le \frac{C_k}{(1+|\zeta|^2)^k}  \qquad\text{and}\qquad \sup_{\omega\in\S^{n-1}}\big\|  \widehat g(\cdot,r \omega) \big\|_{L^1} \le Cr . 
\end{equation}

Now, using the asymptotics \eqref{BesselJ}, for every $\zeta\in\R^n$, $\xi = r \omega$ with $(r,\omega)\in\R_+ \times \S^{n-1}$, as $r\to\infty$, 
\[\begin{aligned}
\widehat{\1_B}(\zeta+\xi)  \widehat{\1_B}(\zeta-\xi)  & = 
\frac{2\cos(r + \omega \cdot \zeta +c_n)\cos(r - \omega \cdot \zeta +c_n) + \O(r^{-1})}{\pi r^{n+1}} \\
& = \frac{\cos(2r +2c_n)+\cos(2\zeta \cdot\omega) + \O(r^{-1})}{\pi r^{n+1}} .
\end{aligned}\]
The error term is uniform for $|\zeta| \le r^{\alpha}$ for any $0<\alpha<1$ and $\omega\in \S^{n-1}$,
so by \eqref{hatgest}, we can substitute these asymptotics into the previous formula for $L_\hbar$ (up to a negligible error).
This yields 3 terms,
\begin{equation} \label{Ired5}
L_\hbar  = - \tfrac1{\pi^2} (2\pi\hbar)^{n+1}  \big( I_{1,\hbar} + I_{2,\hbar} + I_{3,\hbar} \big) + \O(\hbar^{n+1})  
\end{equation}
where
\[\begin{aligned}
L_{1,\hbar} & = (2\pi)^{-\frac n2}  \int  \overline\chi(\tfrac{r}{\hbar})   \widehat g(\zeta,r \omega) \cos(\zeta \cdot\omega)    \frac{\dd r}{r^2} \dd \zeta \dd \omega \\
L_{2,\hbar} & =  (2\pi)^{-\frac n2} \int  \overline\chi(\tfrac{r}{\hbar})   \widehat g(\zeta,r \omega)  \cos(2r/\hbar +2c_n) \frac{\dd r}{r^2}  \dd \zeta \dd \omega \\
\end{aligned}\]
and $L_{3,\hbar}$ is controlled using \eqref{hatgest} as
\[
L_{3,\hbar} =\O \bigg( \hbar\int  \overline\chi(\tfrac{r}{\hbar}) \sup_{\omega\in\S^{n-1}}\big\|  \widehat g(\cdot,r \omega) \big\|_{L^1}\frac{\dd r}{r^3}  \bigg)  = \O \bigg( \hbar\int  \overline\chi(\tfrac{r}{\hbar})\frac{\dd r}{r^2}  \bigg) . 
\]
Then, since $ \int  \overline\chi(r) \frac{\dd r}{r^2} <\infty$, 
\begin{equation} \label{Ired6}
L_{3,\hbar} =\O (1) .
\end{equation}

We now turn to the highly oscillating term $L_{2,\hbar}$ that we can rewrite (by Fourier's inversion formula)
\[
L_{2,\hbar} =\int  \overline\chi(\tfrac{r}{\hbar}) g(0,r\omega) \cos(2r/\hbar +2c_n)  \frac{\dd r}{r^2} \dd  \omega . 
\]
Let $j(r) := \int  g(0,r\omega)  \dd \omega$, which we may view as an even function in $ \mathcal F(\R)$. So
\[
L_{2,\hbar} =  \Re\bigg( e^{2i c_n} \int  \overline\chi(\tfrac{r}{\hbar}) j(r)e^{i\tfrac{2r}{\hbar}}  \frac{\dd r}{r^2} \bigg) . 
\]
and we can make an integration by parts, 
\[
L_{2,\hbar} = -  \frac\hbar2\Im\bigg(  e^{2i c_n} \int  \partial_r\big\{\overline\chi(\tfrac{r}{\hbar}) j(r)\tfrac1{r^2} \big\} e^{i\tfrac{2r}{\hbar}}  \dd r \bigg) . 
\]
Then, we control $L_{2,\hbar}$ as follows,
\[
\left| \int  \partial_r\overline\chi(\tfrac{r}{\hbar})j(r) r^{-2} \dd r \right| \le C \int |\chi'(r)| \frac{\dd r}{r} <\infty , 
\]
similarly
\[
\hbar \int  \left| \overline\chi(\tfrac{r}{\hbar})j(r) r^{-3}  \right|\dd r 
\le C \hbar \int \overline\chi(\tfrac{r}{\hbar}) \frac{\dd r}{r^2}  =  C \int \overline\chi(r) \frac{\dd r}{r^2} <\infty
\]
and 
\[
\hbar  \int \left|  \overline\chi(\tfrac{r}{\hbar})j'(r) r^{-2}  \right|  \dd r 
\le C \int \overline\chi(r) \frac{\dd r}{r^2} <\infty .
\]
We conclude that 
\begin{equation} \label{Ired7}
L_{2,\hbar} = \O(1) .
\end{equation}
We finally turn to the main $L_{1,\hbar}$ which we compute using Fourier's inversion formula, 
\[
L_{1,\hbar} =   \int  \overline\chi(\tfrac{r}{\hbar})  g(\omega,r \omega)   \frac{\dd r}{r^2} \dd \zeta \dd \omega  
\]
where we used the invariance of the Haar measure $\dd \omega  $ on $\S^{n-1}$. 
Using the notation \eqref{def:q}, this yields
\[
L_{1,\hbar}  =  \int  \overline\chi(\tfrac{r}{\hbar})  q(r)  \frac{\dd r}{r}  .
\]
Since $q (r)\to Q$ as $r\to0$, if $Q\neq 0$,  this integral cannot be bounded as $\hbar\to0$. However, it diverges logarithmically under the assumption \eqref{condq}.
Indeed, we can split if $C>0$ is sufficiently large (depending only on the cutoff), 
\[
L_{1,\hbar} =   Q\bigg(\int_{C\hbar}^1\frac{\dd r}{r}  
+ \int_0^{C\hbar}  \overline\chi(\tfrac{r}{\hbar})  \frac{\dd r}{r} \bigg)
+ \int_1^\infty \frac{ q(r)}{r} \dd r + \O\bigg(  \int_0^1\left| \frac{ q(r)-Q}{r}\right| \dd r \bigg)
\]
The last three terms are $\O(1)$ since $ \int_0^{C}  \overline\chi(r)  \frac{\dd r}{r}<\infty$, $q \in C_c$ and \eqref{condq} holds. This shows that 
\[L_{1,\hbar} =  Q \log \hbar^{-1} +\O(1).\]
Combining this asymptotics with \eqref{Ired6}, \eqref{Ired7} into \eqref{Ired5}, we conclude that as $\hbar\to0$,
\[
L_\hbar = - \tfrac1{\pi^2} (2\pi\hbar)^{n+1}  \big(Q \log \hbar^{-1} +\O(1) \big) . \qedhere
\]
\end{proof}

\subsubsection{Proof of Proposition~\ref{prop:intasymp}} \label{sec:proofintasymp}

We now put together the different steps of the proof of Proposition~\ref{prop:intasymp}.
Recall that $\varphi:\R^n\to\R^n$ is a $C^\infty$-diffeomorphism such that $\varphi^{-1}(\Omega)=B$.
First by Proposition~\ref{prop:smooth-diffeo}  with $g=f(\varphi)$
\begin{align*}
L_\hbar &= \int e^{i\frac{(x-y)\cdot\xi}{\hbar}} F(x,y,\xi) (f|_\Omega(x)-f|_\Omega(y))^2 \dd \xi \dd x \dd y \\
&= \int e^{i\frac{(x-y)\cdot\xi}{\hbar}} G(x,y,\xi) (g|_B(x)-g|_B(y))^2\dd \xi \dd x \dd y 
+ \O(\hbar^{n+1}) 
\end{align*}
and on the diagonal,
\[
G(x,x,\xi) = F\big(\varphi(x),\varphi(x), D\varphi(x)^{-*}\xi\big) \J_\varphi(x) , \qquad (x,\xi) \in\R^{2n}.
\]
Here we used that by construction, $\J_\Phi(x,y,\xi) = \J_\varphi(x)\J_\varphi(y)\J_\varphi^{-1}(\tfrac{x+y}{2})$. 

Then, as  $G \in \mathcal F(\R^{2n},\R^n)$, by Proposition~\ref{prop:Ired}, we obtain
\[
L_\hbar
=  (2\pi\hbar)^{n+1}  \big( \tfrac1{\pi^2} Q \log \hbar^{-1} +\O(1) \big) 
\]
where, if it exists, 
\[
Q = \lim_{r\to0^+} \bigg(\frac 1r \int_{\S^{n-1}}  G(\omega,\omega,r\omega) g(\omega)^2 \dd \omega \bigg) . 
\]
It remains to argue under the assumption \eqref{condR} this limits is well-defined and  \eqref{condq} holds. We have as $r\to0$,
\begin{equation} \label{Jacborder1}
\begin{aligned}
G(w,w,rw)  & = r R(\varphi(\omega))  \J_\varphi(\omega) |D\varphi(\omega)^{-*}\omega| +\O(r^2) \\
&=  r R(\hat\varphi(\omega)) \J_{\hat\varphi}(\omega)  +\O(r^2) 
\end{aligned}
\end{equation}
where the map $\hat\varphi : \partial B \to \partial \Omega$ is the $C^\infty$-diffeomorphism induced by $\varphi$ on the boundary and the errors are controlled uniformly over $\omega\in\partial B$. 
To obtain \eqref{Jacborder1}, note that for any $\omega\in \partial B$, we can decompose $T_\omega(\R^n) = T_\omega(\partial B)\oplus \R\omega$ and, with $\hat x = \varphi(\omega)$,  $T_{\hat x}(\R^n) = T_{\hat x}(\partial\Omega) \oplus \R \nu(\hat x)$ where $\nu(\hat x)$ is the (unit) normal to $\partial\Omega$ at $\hat x$. 
In this decomposition, since $\varphi^{-1}(\Omega)=B$ the matrix of the differential $D\varphi$ of the map $\varphi$ has the following from 
\begin{equation} \label{defalpha}
D\varphi = \begin{pmatrix} 
D\hat\varphi  & 0 \\
\star & \alpha
\end{pmatrix} ,
\qquad \text{where}\quad
\alpha(\omega) = \nu(\hat x) \cdot D\varphi(\omega) \omega , \quad
\hat x = \varphi(\omega) .
\end{equation}
In particular $\alpha>0$ on $\partial B$ and, by taking determinants, this implies that for $\omega\in\partial B$, 
\begin{equation} \label{Jacborder2}
\J_\varphi(\omega) = \J_{\hat\varphi}(\omega) \alpha(\omega)
\end{equation}

Moreover by definition of $\nu$, for any $\omega\in\partial B$ and $v \in T_\omega(\partial B) = \omega^{\perp}$, 
\[
0 = \nu(\varphi(\omega)) \cdot D\varphi(\omega) u ,  
\]
which shows that $D\varphi(\omega)^{*}\nu(\varphi(\omega))$ is proportional to $\omega$. 
According to \eqref{defalpha}, this implies that for $\omega \in\partial B$, 
\[
D\varphi(\omega)^{*}\nu(\varphi(\omega)) = \alpha(\omega) \omega
\]
and since $ \big|\nu(\varphi(\omega))\big| =1$, we conclude that 
\[
1 = \alpha(\omega) |D\varphi(\omega)^{-*}\omega|
\]
Combined with \eqref{Jacborder2}, this proves formula \eqref{Jacborder1}. 
Then, we deduce from this expansion that 
\[
Q= \int_{\partial B}  R(\hat\varphi(\omega))f(\hat\varphi(\omega))^2\J_{\hat\varphi}(\omega) \dd \omega
= \int_{\partial\Omega} R(\hat x) f(\hat x)^2 \dd \hat x
\]
by a change of variable. In addition, we have
\[
\int_0^1  \int_{\S^{n-1}}  \bigg| \frac{G(w,w,rw)-rR(\hat\varphi(\omega)) \J_{\hat\varphi}(\omega) }{r^2} \bigg| g(\omega)^2 \dd r \dd\omega <\infty .
\]
which guarantees that the condition \eqref{condq} holds.
This completes the proof. \qed

\subsection{Partition of unity; Proof of Theorem \ref{thm:J2}.} 
\label{sec:partition-unity-end}

We now combine our previous results, with a decomposition using a partition of unity, to prove Theorem \ref{thm:J2}.
Recall that $\Omega \Subset \mathcal{D}$ is an open set with a smooth boundary and, according to Proposition~\ref{prop:red}, $(2\pi\hbar)^{n-1} \big\| [\widetilde{\Pi}, f|_\Omega] \big\|^2_{\J^2}= \mathcal{Q}(f)+\O(1)$ as $\hbar\to0$ where $\mathcal{Q}_\Omega$ is the quadratic form 
\[
\mathcal{Q}_\Omega:f \mapsto  \frac{-1}{(2\pi\hbar)^{n+1}}\int e^{i\frac{(x-y)\cdot\xi}{\hbar}}
F(x,y,\xi) (f|_\Omega(x)-f|_\Omega(y))^2 \dd \xi \dd x \dd
y
\]
acting on smooth functions.

Moreover, by Proposition~\ref{prop:intasymp} and \eqref{Fpp}, in case
$\Omega$ is contractible ($C^\infty$-diffeormorphic to $B_n$), one has
$\mathcal{Q}_\Omega(f)= \log \hbar^{-1}  \mathcal{V}(f) +\O(1)$
as $\hbar\to0$ where $\mathcal{V}$ is the quadratic form
\begin{equation} \label{var}
\mathcal{V}_\Omega:f\mapsto \frac{\cst_{n-1}}{\pi^2}   \int_{\partial\Omega} |V(\hat x)|^{\frac{n-1}2} f(\hat x)^2 \dd \hat x 
\end{equation}
acting on smooth functions.

\medskip

We claim that there is a smooth partition of unity  $(\chi_j)_{1\leq j\leq J}$ such that
\begin{itemize}[leftmargin=*] 
\item $\chi_j|_\Omega = \chi_j|_{\Omega_j}$ where $\Omega_j$ is a contractible set, for $j\in[J]$.
\item  $(\chi_i+\chi_j)|_\Omega =  (\chi_i+\chi_j)|_{\Omega_{i,j}} $ where $\Omega_{i,j}$ is a contractible set, for $i,j\in[J]$.
\end{itemize}
These sets are arbitrary smooth contractible sets. The condition $\chi|_\Omega=\chi|_{\Omega_j}$ is equivalent to \[\Omega\cap(\supp\chi) = \Omega_j\cap(\supp\chi),\] so we can always choose a  partition sufficiently fine so that these conditions hold. This partition is finite since $\Omega$ is compact. 

Thus, it follows that
\[\begin{aligned}
\mathcal{Q}_\Omega(1) = \mathcal{Q}_\Omega\big({\textstyle \sum_{j=1}^J} \chi_j\big) 
&={\textstyle \sum_{i,j=1}^J}  \mathcal{Q}_\Omega(\chi_j+\chi_i)- {\textstyle \sum_{j=1}^J}  \mathcal{Q}_\Omega(\chi_j) \\
&= {\textstyle \sum_{i,j=1}^J}  \mathcal{Q}_{\Omega_{ij}}(\chi_j+\chi_i)- {\textstyle \sum_{j=1}^J}  \mathcal{Q}_{\Omega_j}(\chi_j) \\
& = \log \hbar^{-1} \Big( {\textstyle \sum_{i,j=1}^J}  \mathcal{V}_{\Omega_{ij}}(\chi_j+\chi_i)- {\textstyle \sum_{j=1}^J}  \mathcal{V}_{\Omega_j}(\chi_j) \Big) +\O(1) 
\end{aligned}\]
Using that  $\mathcal{V}_{\Omega_{ij}}(\chi_j+\chi_i) =  \mathcal{V}_{\Omega}(\chi_j+\chi_i)$ and  $\mathcal{V}_{\Omega_j}(\chi_j)=\mathcal{V}_{\Omega}(\chi_j)$ for $i,j\in[J]$, we conclude that 
\[
\mathcal{Q}_\Omega(1) = \log \hbar^{-1} \mathcal{V}_\Omega\big({\textstyle \sum_{j=1}^J} \chi_j\big) +\O(1) .
\]
Similarly, we have for any $f\in C^\infty(\R^n)$, 
\(
\mathcal{Q}_\Omega(f) = \log \hbar^{-1} \frac{\cst_{n-1}}{\pi^2} \mathcal{V}_\Omega(f) +\O(1) 
\)
as $\hbar\to0$. 

According to Proposition~\ref{prop:regularised_commutator}, this shows
that for any open set $\Omega \Subset \mathcal{D}$ with smooth
boundary and any $f\in C^\infty(\R^n)$, 
\[
(2\pi\hbar)^{n-1} \big\| [\Pi, f|_\Omega] \big\|^2_{\J^2}= (2\pi\hbar)^{n-1} \big\| [\widetilde{\Pi}, f|_\Omega] \big\|^2_{\J^2} +\O(1)=  \log \hbar^{-1}\mathcal{V}_\Omega(f) +\O(1) 
\qquad\text{as $\hbar\to0$}. 
\]
This completes the proof (with $\mu=0$ and $V(x)<0$ for $x\in\Omega$). 

\subsection{Central limit theorem; Proof of Theorem~\ref{thm:clt}.}
\label{sec:clt}

The Gaussian fluctuations of the counting statistics, after rescaling,
are due to the determinantal structure of the free fermions point
process $\X$ and the fact that for a non-trivial smooth set $\Omega$,
$\operatorname{var} \X(\Omega) \to\infty$ as $\hbar\to\infty$. In the
random matrix context, this observation is due to
\cite{costin_gaussian_1995} and it has been extended to general
determinantal processes in \cite{soshnikov_gaussian_2002}. By
\cite[Thm 1]{soshnikov_gaussian_2002}, we deduce from
Theorem~\ref{thm:J2} that for any open set $\Omega \Subset
\mathcal{D}$ with smooth boundary and any $f\in C^\infty(\R^n)$, 
(in distribution) as $\hbar\to 0$, 
\[
\frac{\X(f|_\Omega) - \mathbb{E}[\X(f|_\Omega)]}{\sqrt{-(2\pi\hbar)^{1-n}\log \hbar}} \Rightarrow \mathcal{N}_{0,\mathcal{V}_\Omega(f)} . 
\]
In particular, if $\Omega$ has disjoint components $\{\Omega_1, \cdots , \Omega_k\}$,
$\mathcal{V}_\Omega=\sum_{j=1}^k\mathcal{V}_{\Omega_j}$
according to \eqref{var}, 
then 
\[
\left(\frac{\X(\Omega_1) - \mathbb{E}[\X(\Omega_1)]}{\sqrt{-(2\pi\hbar)^{1-n}\log \hbar}} , \cdots, \frac{\X(\Omega_k) -  \mathbb{E}[\X(\Omega_k)]}{\sqrt{-(2\pi\hbar)^{1-n}\log \hbar}} \right) \Rightarrow \mathcal{N}_{0,\Sigma}
\]
where
\(
\Sigma = \operatorname{diag}\big(\mathcal{V}_{\Omega_1},\cdots, \mathcal{V}_{\Omega_k}\big).
\)

Note that this argument cannot be directly applied to a collection of
sets $\Omega_1, \cdots , \Omega_k \Subset \mathcal{D}$ (open, with
smooth boundaries) with intersecting boundaries. 
However, we can use the off-diagonal decay of the regularized kernel $\widetilde{\Pi}$ and  the estimates from  Lemma~\ref{lem:cov} to show that the cross terms are negligible.

\begin{lem}\label{prop:kerdiag}
Let $\Omega \Subset \mathcal{D}$ be open.
The kernel \eqref{kern_reg} satisfies for $(x,y)\in\Omega^2$, 
\[
|\widetilde{\Pi}(x,y)|^2\leq  \frac{C\hbar^{1-n}}{(\hbar+|x-y|)^{n+1}}
\]
\end{lem}

\begin{proof} 
Let $\chi :\R^n\to[0,1]$ be a smooth cutoff such that $\chi=1$ on $\Omega$ and  $\chi=0$ on $\Omega'$ where $\Omega'$ is a neighborhood of $\Omega$ in $\mathcal D$.  
We proceed as in Section~\ref{sec:statphase}, we make a change of variable $\xi=r\omega$ with $r>\underline{\cst}$ (cf.~\eqref{supp_bulk}), $\omega\in\S^{n-1}$ and we apply the stationary phase method to the integral \eqref{kern_reg} in the variables $(t,r)$ with $(x,y,\omega,\lambda)$ fixed.
The critical point is non-degenerate, given by \eqref{critptrt} with $\sigma=0$, so we obtain
\begin{equation}\label{kern_red} 
\chi(x)\widetilde{\Pi}(x,y)\chi(y) = \frac{1}{(2\pi\hbar)^{n}}\int
e^{i\tfrac{\Psi_1(x,x-y,\omega,\lambda)}{\hbar}}b(x,y,\omega,\lambda)\1\{\lambda\le0\} \dd
\omega \dd \lambda 
\end{equation}
where $b\in S^1(\R^{3n+1})$ is supported on $\big\{|\lambda|, |x-y|\le
\underline{\ell} \}$, and the phase $\Psi_1$ vanishes on the diagonal
$\{x=y\}$. This already implies the following saturated bound, locally uniformly in the bulk:
\[
|\widetilde{\Pi}(x,y) | \le C \hbar^{-n}  .
\]

Then, we can write
\[
\Psi_1(x,u,\omega,\lambda) = u\cdot \zeta(x,u,\omega,\lambda) 
\]
where the map $\zeta :\R^{3n+1} \to \R^n$ is smooth  on $\supp(b)$
and make a change of coordinates $(\lambda,\omega) \to \zeta(x,u,\omega,\lambda)$ in \eqref{kern_red}.
As in Section~\ref{sec:phase} (albeit using a simpler analysis), we claim that 
\[
\zeta(x,u,\omega,\lambda) = R(x,\lambda) \omega+ \O(u) , \qquad R=\sqrt{\lambda-V(x)}
\]
and the error is smooth, so the matrix 
\[
\frac{\partial\zeta}{\partial\lambda\partial\omega} = R^{-1}\mathrm{I}_n + \O(u)
\]
is non-degenerate since $R\ge \underline{\cst}$ on $\supp(b)$ and $\underline{\ell}\ll \underline{\cst}$. Hence, there is $g\in S^1(\R^{3n})$, supported on $\big\{|x-y|\le \underline{\ell} \}$, so that 
\[
\eqref{kern_red}=  \frac{1}{(2\pi\hbar)^{n}}\int
e^{i\tfrac{(x-y)\cdot\zeta}{\hbar}}g(x,y,\zeta)\1\{\lambda(x,y,\zeta)\le0\} \dd\zeta
\]
where the map $\lambda :\R^{3n} \to \R$ is smooth with on the diagonal $\{x=y\}$,
\[
\lambda(x,x,\zeta) = |\zeta|^2 +V(x) . 
\]
By perturbation, ${\rm Hess}_\zeta \lambda(x,y,\zeta) $ is positive-definite for $x\in \Omega' ,|x-y|   \le \underline{\ell}$  so that $\{\zeta\in\R^n :\lambda(x,u,\zeta)\le0\}$ is a strongly convex body.
Thus, by \cite[Corollary 7.7.15]{hormander_analysis_2003}, 
if $u\in\R^n \mapsto \Lambda (x,y,u)$ denotes the Fourier transform of $\zeta \in\R^n\mapsto\1\{\lambda(x,y,\zeta)\le0\}$, then
\[
\sup_{x,y\in\Omega'}| \Lambda (x,y,u)|\leq \frac C{1+|u|^{\frac{n+1}2}}.
\]
This is to be compared with the case of a ball (Lemma~\ref{lem:Bessel}).
This implies that 
\[
\chi(x)\widetilde{\Pi}(x,y)\chi(y) =\hbar^{-n} \Lambda(x,y,\cdot)*\widehat{g}(x,y,\cdot)|_{\tfrac{x-y}{\hbar}}
\]
where $u\in\R^n \mapsto \widehat{g} (x,y,u)$ is a Schwartz function.
We conclude that 
\[
|\chi(x)\widetilde{\Pi}(x,y)\chi(y)| \le C\hbar^{-\frac{n-1}{2}}|x-y|^{-\frac{n+1}{2}}.
\]
which is the appropriate away from the diagonal.
\end{proof}

\begin{rem}
The bound from Lemma~\ref{prop:kerdiag} corresponds to 
$|\widetilde{\Pi}(x,y)| \le C \hbar^{-n}\big|\widehat{\1}\{\cdot \le1\}(\frac{x-y}{\hbar})\big|$  for $(x,y)\in\Omega^2$ (inside the bulk).
This is consistent with the scaling limit of the regularized kernel $\widetilde{\Pi}$. Namely, we showed in \cite{deleporte_universality_2024} that for $x\in  \mathcal{D}$,
$\hbar^n \widetilde{\Pi}(x+ \hbar u ,x+\hbar v ) \to K(u-v) $ where $K = C_n\widehat{\1}\{\cdot \le c_n\}$ for some constants $C_n,c_n$ depending only on the dimension $n$.
\end{rem}

Given two open sets $\Omega_1,\Omega_2 \Subset \mathcal D$ with smooth boundary such that $|\partial \Omega_1\cap \partial \Omega_2|=0$, Lemmas~\ref{prop:kerdiag} and~\ref{lem:cov} imply that as $\hbar\to\infty$, 
\[
\big|\tr( [\widetilde{\Pi},\1_{\Omega_1}][\widetilde{\Pi},\1_{\Omega_2}])\big|
\le C\hbar^{1-n} \int \frac{|(\1_{\Omega_1}(x)-\1_{\Omega_1}(y))(\1_{\Omega_2}(x)-\1_{\Omega_2}(y))|}{(\hbar+|x-y|)^{n+1}}  \dd
x \dd y = o\big(\log(\hbar^{-1})\hbar^{1-n}\big) . 
\]

Then, expanding the square, under the assumptions of Theorem~\ref{thm:clt}, it holds for $\alpha \in\R^k$, 
\[
\frac{ [\widetilde{\Pi},{\textstyle\sum_{j=1}^k \alpha_k \1_{\Omega_k}}] \|^2_{\J^2}}{(2\pi\hbar)^{n-1}(\log \hbar^{-1})}=
\sum_{j=1}^k  \frac{\alpha_k [\widetilde{\Pi},\1_{\Omega_k}] \|^2_{\J^2}}{(2\pi\hbar)^{n-1}(\log \hbar^{-1})} + o(1)
\qquad\text{as $\hbar\to0$}. 
\]
Replacing $\widetilde{\Pi}$ by the projector $\Pi$ using Proposition~\ref{prop:regularised_commutator}, we conclude by Theorem~\ref{thm:J2} that as $\hbar\to0$,
\[
\frac{\|[\Pi,{\textstyle\sum_{j=1}^k \alpha_k \1_{\Omega_k}}] \|^2_{\J^2}}{(2\pi\hbar)^{n-1}(\log \hbar^{-1})}=
\sum_{j=1}^k \alpha_k \mathcal{V}_{\Omega_k}(1) + o(1) .
\]
Again, applying Soshnikov's result \cite[Thm 1]{soshnikov_gaussian_2002}, by the Cramér–Wold argument, this completes the proof of Theorem~\ref{thm:clt}. 

One has the following more precise covariance estimate in the case of
a relatively non-singular intersection.
\begin{rem}
  If $H^{\beta}(\partial \Omega_1\cap \partial \Omega_2)<+\infty$ for
  some $\beta<n-1$, then the last covariance bound
  improves to
  \[
    {\rm Cov}(\X(\1_{\Omega_1}),\X(\1_{\Omega_2}))=\O(\hbar^{1-n})
  \]
  using the second part of Lemma \ref{lem:cov}. At this scale, we
  cannot use the regularised projector from
  Proposition \ref{prop:regularised_commutator} to compute the limit
  of the covariance.
  This difficulty already appears in the variance estimates in \cite{deleporte_universality_2024}.
\end{rem}
\section{Trace norm of commutators \& entropy estimates} \label{sec:tr}

\subsection{Proof of Theorem~\ref{thm:J1}: Multi-scale argument} \label{sec:MS}
The goal of this section is to prove Theorem~\ref{thm:J1}, that is, to
show that for any open set $\Omega \Subset \mathcal{D}$ with a smooth boundary 
\begin{equation} \label{comm0} 
\big\| [\Pi, \1_\Omega] \big\|_{\J^1} = \O\big( \hbar^{1-n} |
\log \hbar|^2 \big) . 
\end{equation}

We begin by a simple truncation step. 

\begin{lem} \label{lem:decomp}
Let $\vartheta : \R \to [0,1]$ be any smooth function which equal to $1$ on $[-\underline{\cst}/2, \underline{\cst}/2]$. 
Let $\Gamma_\pm = \vartheta\1_{\R_{\pm}}(H)$ and $\Theta_{\pm} = \vartheta1_{\R_{\pm}}(\mathrm{w})$. 
It holds
\begin{equation*}  
\| [\Pi,\1_{\Omega}] \|_{\J^1}
\le   
\| \Gamma_- \Theta_+\Gamma_+\Theta_- \|_{\J^1} 
+ \| \Gamma_+ \Theta_+\Gamma_-\Theta_- \|_{\J^1} 
+\| \Gamma_+ \Theta_-\Gamma_-\Theta_+ \|_{\J^1} 
+ \| \Gamma_- \Theta_-\Gamma_+\Theta_+ \|_{\J^1}
+\O\big(\hbar^{1-n}\big). 
\end{equation*}
\end{lem}

\begin{proof}
First observe that since the operator $H$ and the map $\mathrm{w}$ are bounded from below 
\begin{equation} \label{truncation}\begin{cases}
\Pi=\Gamma_- + f_1(H) \\
1-\Pi = \Gamma_+ +1- f_2(H)
\end{cases}
\qquad \qquad 
\begin{cases}
\1_\Omega=\Theta_- + f_1(\mathrm{w}) \\
1-\Pi = \Theta_+ +1- f_2(\mathrm{w})
\end{cases} 
\end{equation}
where  $f_j : \R \to [0,1]$ are smooth functions, supported in $(-\infty,\underline{\cst}]$ with $f_1\1_{\R_+} =0$ and $(1-f_2)\1_{\R_-} =0$.  
Second, observe that
\[\begin{aligned}[]
[\Pi,\1_{\Omega}]
& = \1_{\Omega^c}\Pi \1_{\Omega} - \1_{\Omega} \Pi \1_{\Omega^c}  \\
&= \Pi \1_{\Omega^c}\Pi \1_{\Omega} + (1-\Pi) \1_{\Omega^c}\Pi \1_{\Omega} 
- (1-\Pi) \1_{\Omega} \Pi \1_{\Omega^c} -\Pi \1_{\Omega} \Pi \1_{\Omega^c} .
\end{aligned}\]
Then,  using that $\1_{\Omega} \1_{\Omega^c} =0 $, we can change the
first and fourth terms as follows: 
\begin{equation} \label{expcomm}
[\Pi,\1_{\Omega}]
=  -\Pi\1_{\Omega^c}(1-\Pi)\1_{\Omega}
+(1-\Pi) \1_{\Omega^c}\Pi \1_{\Omega} 
- (1-\Pi) \1_{\Omega} \Pi \1_{\Omega^c} 
+\Pi\1_{\Omega}(1-\Pi)\1_{\Omega^c} .
\end{equation}

Every term on the RHS of \eqref{expcomm} can be handled in the same way, so we focus on the last one. 
Using \eqref{truncation}, since $f_1(H)(1-\Pi) =0$ and $\Gamma_-(1-f_2(H))=0$, one has 
\[
\Pi\1_{\Omega}(1-\Pi)\1_{\Omega^c}
= \Gamma_-\1_{\Omega}\Gamma_+\1_{\Omega^c}
+ [f_1(H) ,\1_{\Omega}](1-\Pi)\1_{\Omega^c}
+ \Gamma_-[\1_{\Omega},1-f_2(H)]\1_{\Omega^c}
\]
According to Proposition~\ref{prop:J1commsmooth}, both $\| [f_j(H) ,\1_{\Omega}] \|_{\J^1}=\O(\hbar^{1-n})$ for $j\in\{1,2\}$, so 
\[
\Pi\1_{\Omega}(1-\Pi)\1_{\Omega^c}
= \Gamma_-\1_{\Omega}\Gamma_+\1_{\Omega^c} +  \O_{\J^1}(\hbar^{1-n} ) .
\]
Similarly, 
\[
\Gamma_-\1_{\Omega}\Gamma_+\1_{\Omega^c} 
=  \Gamma_-\Theta_-\Gamma_+\Theta_+
+   \Gamma_-[ f_1(\mathrm{w}), \Gamma_+]\1_{\Omega^c} 
+ \Gamma_-\Theta_-[\Gamma_+,1-f_2(\mathrm{w})] .
\]
Observe that according to Lemma~\ref{lem:J1commnonsmooth} and \eqref{FM},
\[\begin{aligned}[]
[ f_1(\mathrm{w}), \Gamma_\pm] 
=  [ f_1(\mathrm{w}), f_j(H)]-  [ f_1(\mathrm{w}),\Pi]  = \O_{\J^1}(\hbar^{1-n}) .
\end{aligned}\]
Thus, we conclude that 
\[
\Pi\1_{\Omega}(1-\Pi)\1_{\Omega^c}= \Gamma_-\Theta_-\Gamma_+\Theta_+ +\O_{\J^1}(\hbar^{1-n}) . \qedhere
\]
\end{proof}

Again, every term in the bound from Lemma~\ref{lem:decomp} can be handled in the same way. 
We proceed using a dyadic argument to \emph{isolate the contributions from different scales} that we will estimate based on Proposition~\ref{prop:J2_estimates_dyadic}.
To setup the dyadic decompositions, on top of the notations from Section~\ref{sec:not},  we use the following convention throughout this section.

\begin{notation} \label{def:g} ~
\begin{itemize}[leftmargin=*]
\item Let $\mathrm{w}:\R^n \to \R$ be a smooth function such that $\Omega = \{\mathrm{w} < 0\} \Subset \mathcal{D}$ and $\partial_x \mathrm{w}(x)\neq 0$ for all $x\in \partial \Omega =\{\mathrm{w}=0 \}$. 
In particular, we assume that $\partial_x \mathrm{w} \neq 0$ on $\{|\mathrm{w}|\le \underline{\cst} \}$ by choosing the constant $\underline{\cst}$ sufficiently small.
\item Let $\vartheta :\R \to[0,1]$ supported in $[-\underline{\cst},\underline{\cst}]$ and equal to $1$ on $[-\underline{\cst}/2, \underline{\cst}/2]$.
Let $\chi:\R \to[0,1]$ supported in $[-1,1]$ and equal to 1 on $[-\underline{\cst},\underline{\cst}]$. 
\item Let $L := \lfloor  \log_2(\hbar) \rfloor$ and let  $\eta_k := \hbar2^k$ for $k\in \N_{<L}$.  
\item
Let $\chi_0 := \chi(\hbar^{-1}\cdot)$ and  
\[
\chi_k:x\mapsto \chi_0(2^{-k}x)-\chi_0(2^{1-k}x) ,  \qquad  k\in\N_{<L}. 
\]
We also let $\chi_L : x\mapsto \chi(x) - \chi_0(2^{1-L}x)$.
We use a similar notation for $(\vartheta_k)_{k=0}^L$.
\item We use the shorthand notation, $\chi_k^{\pm}= \chi_k\1_{\R_\pm}$
for $k\in[0,L]$;  similarly for $(\vartheta_k^\pm)_{k=0}^L$.
\item We let $\psi_k^{\pm}:= \vartheta \cdot
  (\chi_k^{\pm}*\rho_\hbar)$ for~$k\in [1,L]$, as in the
  Notation of Section \ref{sec:not}. 
In particular,  $\chi_k^{\pm} , \psi_k^{\pm}: \R \to [0,1]$ are in $S^{\eta_k}$ for~$k\in [1,L]$. 
\item For $k\in [1,L]$, let $\breve{\vartheta}_k^\pm$ be in $S^{\eta_k}$ such that 
\begin{equation} \label{commchi}
\chi_{k}^{\pm} = \chi_{k}^{\pm} \breve\chi_k^{\pm} ,  \qquad \chi_{k}^{\mp} \breve\chi_m^{\pm} =0, \qquad\text{for~$k,m\in [1,L]$}.
\end{equation}
\end{itemize}
\end{notation}

This yields a dyadic decomposition of the operators appearing in Lemma~\ref{lem:decomp}. 

\begin{lem} \label{lem:split}
Using the notation~\ref{def:g}, we have 
$\Theta_{\pm} = \sum_{k=0}^L \vartheta_k(\mathrm{w})$
and 
$\Gamma_{\pm} = \sum_{k=1}^L \widetilde\chi^{\pm}_k(H) +
\Upsilon_{\pm}$ where the errors satisfy \[\|\Upsilon_\pm\|_{\J^1}=
\O(\hbar^{1-n})\qquad \qquad \|\Upsilon_\pm\|_{L^2\to L^2} \le 1.\]
\end{lem}
\begin{proof} The decomposition of $\Theta_\pm$ is obvious, so we focus on the decomposition of $\Gamma_+$ ($\Gamma_-$ is handled similarly).
By construction $\sum_{k=0}^L \chi_k^+ = \1_{\R_+}$ on $[-\underline{\cst},\underline{\cst}]$. So, since $\rho$ is a Schwartz mollifier, we claim that for every $m\in\N$, there is $\varkappa_m:\R \to[-1,1]$, smooth with compact support so that
\[
\textstyle \sum_{k=1}^L \chi_k^+ *\rho_\hbar = \1_{\R_+}+ \varkappa_m(\hbar^{-1}\cdot) +\O(\hbar^m) \qquad \text{with a smooth error on $[-\underline{\cst},\underline{\cst}]$} .
\]
By Lemma~\ref{lem:Weyl}, this implies that $ \sum_{k=1}^L \widetilde\chi^{\pm}_k(H) = \Gamma_+ + \Upsilon_+$ where the operator $\Upsilon_+$ satisfies \[\|\Upsilon_\pm\|_{\J^1}=
\O(\hbar^{1-n})\qquad \qquad \|\Upsilon_\pm\|_{L^2\to L^2} \le 1.\]
\end{proof}

We are now ready to proceed to obtain the estimate \eqref{comm0}. 

\begin{proof}[Proof of Theorem~\ref{thm:J1}]
According to Lemma~\ref{lem:decomp}, it suffices to show that 
\(
\|\Gamma_+ \Theta_+\Gamma_-\Theta_- \|_{\J^1} \le \Cst \hbar^{1-n} L^2
\);
since the other terms are handled similarly. 
By  Lemma~\ref{lem:split}, 
\[
\Gamma_+ \Theta_+\Gamma_-\Theta_-  = 
\sum_{i,k \ge 1} \sum_{j,m \ge 0} \widetilde\chi^+_i(H) \vartheta_j^+(\mathrm w)\widetilde\chi^-_k(H) \vartheta_m^-(\mathrm w) + \O_{\J^1}(\hbar^{1-n})
\] 
where the indices range up to $L$. 
The idea is to split this sum in two parts:
\[
\mathrm{I} := \sum_{\substack{ i,k \ge 1 ; j,m \ge 0\\ i+j \le k+m}} \widetilde\chi^+_i(H) \vartheta_j^+(\mathrm w)\widetilde\chi^-_k(H) \vartheta_m^-(\mathrm w)
\qquad\text{and}\qquad
\mathrm{II} := \sum_{\substack{ i,k \ge 1 ; j,m \ge 0\\ i+j > k+m}} 
\widetilde\chi^+_i(H) \vartheta_j^+(\mathrm w)\widetilde\chi^-_k(H) \vartheta_m^-(\mathrm w)
\]
Both sums will be handled in a similar way, so we focus on $\mathrm{I} $ for now. Using \eqref{commchi}, one has 
\[
\widetilde\chi_{i}^+(H)\vartheta_{j}^+(\mathrm{w})\widetilde\chi_{k}^-(H)\vartheta_{m}^-(\mathrm{w})
= \widetilde\chi_{i}^+(H)\vartheta_{j}^+(\mathrm{w})\big[\big[\widetilde\chi_{k}^-(H) ,\vartheta_{m}^-(\mathrm{w}) \big],\breve\vartheta_{m}^-(\mathrm{w}) \big] .
\]
Then, we can bound 
\[\begin{aligned}
\| \widetilde\chi_{i}^+(H)\vartheta_{j}^+(\mathrm{w})\widetilde\chi_{k}^-(H)\vartheta_{m}^-(\mathrm{w})\|_{\J^1}
\le \|\widetilde\chi_{i}^+(H)\vartheta_{j}^+(\mathrm{w})\|_{\J^2}
\begin{cases}
\| \widetilde\chi_{k}^-(H)\vartheta_{m}^{-}(\mathrm{w})\|_{\J^2}  &\text{if } \eta_k\eta_m\le \hbar \\
\big\| \big[\big[\widetilde\chi_{k}^-(H),\vartheta_{m}^{-}(\mathrm{w}) \big],\breve\vartheta_{m}^{-}(\mathrm{w}) \big]\big\|_{\J^2} &\text{if } \eta_k\eta_m>\hbar
\end{cases}.
\end{aligned}\]
According to Proposition \ref{prop:J2_estimates_dyadic}, there exists a constant $C>0$ so that for any $k \in[1,L]$ and $m\in[0,L]$,
\[ \begin{cases}
\big\| \big[\big[\widetilde\chi_{k}^\pm(H),\vartheta_{m}^{\pm}(\mathrm{w}) \big],\breve\vartheta_{m}^{\pm}(\mathrm{w}) \big]\big\|_{\J^2}^2
\le C \hbar^{3-n}\eta_k^{-2}\eta_m^{-2}  = C\hbar^{-1-n} 2^{-2(k+m)} &\text{if } \eta_k\eta_m>\hbar , \\
\|\widetilde\chi_{k}^\pm(H)\vartheta_{m}^\pm(\mathrm{w})\|_{\J^2}^2 \lesssim \hbar^{-n}\eta_k \|\vartheta_{m}^\pm(\mathrm{w}) \|_{L^2}^2
\le C
\hbar^{-n}\eta_k\eta_m =C \hbar^{2-n} 2^{k+m}.
\end{cases}\]
Here we used that $ \|\vartheta_{m}^\pm(\mathrm{w}) \|_{L^2}^2 \lesssim \eta_m$ for $m\in[0,L]$. 
Altogether, this implies that
\begin{equation*}\label{ms1}
\begin{aligned}
\|\mathrm{I}\|_{\J^1} & \le \sum_{\substack{ i,k \ge 1 ; j,m \ge 0\\ i+j \le k+m}} \|\widetilde\chi_{i}^+(H)\vartheta_{j}^+(\mathrm{w})\widetilde\chi_{k}^-(H)\vartheta_{m}^-(\mathrm{w})\|_{\J^1}
\\ &\le C^2 \Bigg( \hbar^{2-n}  \sum_{\substack{i+j \le k+m \\ \hbar 2^{k+m} \le1}} 2^{\frac{i+j+k+m}{2}} + \hbar^{\frac12-n} \sum_{\substack{i+j \le k+m \\ 1 \le \hbar 2^{k+m} }} 2^{\frac{i+j}{2}-(k+m)} \Bigg). 
\end{aligned}
\end{equation*}
Since all indices range up to $L$, these sums are controlled by 
\[
\sum_{\substack{i+j \le k+m \\ \hbar 2^{k+m} \le1}} 2^{\frac{i+j+k+m}{2}}  \le L^2 \sum_{\hbar 2^m\le 1} 2^{m} \le 2\hbar^{-1} L^2  
\qquad\text{and}\qquad
\sum_{\substack{i+j \le k+m \\ 1 \le \hbar 2^{k+m} }} 2^{\frac{i+j}{2}-(k+m)}
\le  L^2  \sum_{  \hbar 2^m\ge 1} 2^{-\frac m2}  \le 2 \hbar^{\frac12}  L^2 .
\]
Hence, we conclude that  
\begin{equation}\label{ms2}
\|\mathrm{I}\|_{\J^1} =\O\big(L^2 \hbar^{1-n}\big) .
\end{equation}

We have a similar control for $\mathrm{II}$, which we expect by symmetry. 
Let us first record that 
\[
\sum_{i,k \ge 1} \sum_{j,m \ge 0}
\big\|\big[\widetilde\chi_{i}^+(H),\vartheta_{j}^+(\mathrm{w})\big]\big\|_{\J^2} \big\| \big[\widetilde\chi_{k}^-(H),\vartheta_{m}^-(\mathrm{w}) \big]\big\|_{\J^2}
= \bigg(\sum_{k\ge 1, m\ge 0} \big\| \big[\widetilde\chi_{k}^+(H),\vartheta_{m}^{+}(\mathrm{w}) \big]\big\|_{\J^2}\bigg)^2
\]
and, as above (by Proposition \ref{prop:J2_estimates_dyadic}),  
\[
\sum_{k\ge 1, m\ge 0} \big\| \big[\widetilde\chi_{k}^+(H),\vartheta_{m}^{+}(\mathrm{w}) \big]\big\|_{\J^2}
\le  C\bigg( \hbar^{1-\frac n2}   \sum_{ \hbar 2^{k+m} \le1} 2^{\frac{k+m}{2}} 
+ \hbar^{-\frac n2} \sum_{\hbar 2^{k+m} \ge1 } 2^{-\frac{k+m}{2}} \bigg)
=\O\big(L \hbar^{\frac{1-n}2} \big).
\]
This estimate shows that 
\[
\mathrm{II} = \sum_{\substack{ i,k \ge 1 ; j,m \ge 0\\ i+j > k+m}} 
\vartheta_{j}^+(\mathrm{w})\widetilde\chi_{i}^+(H)\vartheta_{m}^-(\mathrm{w})\widetilde\chi_{k}^-(H) + \O_{\J^1}\big(L^2 \hbar^{1-n}\big)  .
\]
Now, using that 
\[
\vartheta_{j}^+(\mathrm{w})\widetilde\chi_{i}^+(H)\vartheta_{m}^-(\mathrm{w})\widetilde\chi_{k}^-(H)=
\big[\breve\vartheta_{j}^+(\mathrm{w}) ,\big[\vartheta_{j}^-(\mathrm{w}),\widetilde\chi_{i}^-(H) \big]\big] \vartheta_{m}^-(\mathrm{w})\widetilde\chi_{k}^-(H)
\]
we can proceed exactly as above to show that 
\(
\|\mathrm{II}\|_{\J^1} =\O\big(L^2 \hbar^{1-n}\big) .
\)
This completes the proof.
\end{proof}

\subsection{Bounds for the entanglement entropy: Proof of Theorem~\ref{thm:ent}} \label{sec:ent}

The estimates of Theorem~\ref{thm:ent} for the (entanglement) entropy follow by interpolation from the bounds for the Hilbert-Schmidt and trace norm of the commutator $[\Pi,\1_\Omega]$ of Theorems~\ref{thm:J2} and~\ref{thm:J1}. 
Our goal is to prove the following general inequalities. 

\begin{lem} \label{lem:ent}
Let $\X$ be a determinantal process on a Polish space $\mathcal{X}$
associated with a (self-adjoint) operator $0<\Pi \le 1$
locally trace-class. Then, for any open set $\Omega
  \Subset \mathcal{X}$ with smooth boundary, 
\[
2\big\| [\Pi,\1_\Omega] \big\|_{\J^2}^2  \le \mathcal{S}_\Omega(\X)\le 4\big\| [\Pi,\1_\Omega] \big\|_{\J^2}^2 \log\bigg(\frac{\| [\Pi,\1_\Omega]\|_{\J^1}}{\| [\Pi,\1_\Omega] \|_{\J^2}^2 } \bigg) . 
\]
\end{lem}

Since we establish that $\| [\Pi,\1_\Omega] \|_{\J^2}^2 \simeq \cst
\hbar^{1-n} \log(\hbar^{-1})$ for some constant $\cst>0$ if
$\Omega\Subset\mathcal D$ is an open set with smooth boundary
and $\| [\Pi,\1_\Omega] \|_{\J^1} \le C  \log(\hbar^{-1}) \| [\Pi,\1_\Omega] \|_{\J^2}^2$ as $\hbar \to0$, Theorem~\ref{thm:ent} follows directly from Lemma~\ref{lem:ent}. We now turn to its proof.

\smallskip

Let $s :  [0,1] \to \R_+$, given by $ s:\lambda \mapsto - \lambda
\log(\lambda) - (1-\lambda)\log(1-\lambda)$ . Recall that for a
general determinantal point process $\X$, defined on a Polish space
$\mathcal{S}$ and associated with a locally trace-class operator
$0<\Pi \le 1$, the (entanglement) entropy of any open set $\Omega
\Subset \mathcal{X}$ with smooth boundary is 
\begin{equation} \label{sEnt}
\mathcal{S}_\Omega(\X)
=  \sum_{n\in\N} s(\lambda_{n}) ,
\end{equation}
where $\{\lambda_{n}\}$ denote the non-trivial eigenvalues of operator $\Pi|_{\Omega} = \1_\Omega \Pi \1_{\Omega}$ --  $\{\lambda_n\}$ is a countable sequence in $(0,1]$ by the spectral theorem for locally compact operators.

\smallskip

The main step of the proof is to describe the relationship between the
spectra of $\Pi\1_{\Omega}\Pi$ and the square-commutator $-[\Pi,\1_{\Omega}]^2$. This follows from a general result. 

\begin{lem}\label{lem:spectral_link}Let $P,Q$ be two self-adjoint projections on a separable Hilbert space.
The non-negative operators $PQP$, $-[P,Q]^2$ 
commute. 
The spectral data of $-[P,Q]^2$, on the orthogonal of its kernel, consists exactly of the eigenpairs
$(\lambda(1-\lambda),u)$ and $(\lambda(1-\lambda),[P,Q]u)$,
where $(\lambda,u)$ is an eigenpair of
$PQP$ with $\lambda> 0$. \\
Thus, given $f:[0,\frac14]\to \R_+$ continuous with $f(0)=0$ and
$g:\lambda \in [0,1]\mapsto f(\lambda(1-\lambda))$, one has
\begin{equation}\label{eq:spectral_link}
\tr f(-[P,Q]^2)=2\tr g(PQP).
\end{equation}
\end{lem}

\begin{proof}
Using that $P^2=P$ and $Q^2=Q$, we simply compute
\begin{equation}\label{PQ}
P[P,Q]^2 = [P,Q]^2P = PQPQP -PQP  = PQP (PQP -1). 
\end{equation}
So $(P,-[P,Q]^2)$ commute. 
Then, using that $(P,-[P,Q]^2)$ and $(Q,-[P,Q]^2)$ commute, we also have 
\[
[PQP,[P,Q]^2] = PQ [P,Q]^2P - P[P,Q]^2QP =P[Q, [P,Q]^2]P =0. 
\]
Thus  $(PQP,-[P,Q]^2)$ also commute as claimed. 

\smallskip

Since $(P,-[P,Q]^2)$ commute, the spectral data of $-[P,Q]^2$ can be decomposed into the spectral data of $-P[P,Q]^2P$ and  $-(1-P)[P,Q]^2(1-P)$. 
Then, \eqref{PQ} shows that any eigenpair $(\lambda,u)$ of $PQP$ with $\lambda>0$ gives an eigenpair $(\lambda(1-\lambda),u)$  of $-P[P,Q]^2P$. 
This means that for $\sigma>0$, the eigenspace of $-P[P,Q]^2P$ associated with $\sigma$ is 
$F_\sigma = \ker(PQP-\lambda)$ if $\sigma=\lambda(1-\lambda)$.
By the spectral theorem, this gives the complete spectral data of $-P[P,Q]^2P$, on the orthogonal of its kernel. 

Now, let $\sigma>0$ be an eigenvalue of $-(1-P)[P,Q]^2(1-P)$ and let $E_{\sigma}$ the associated eigenspace and 
$T:= [P,Q]$. 
Since $ E_\sigma\subset\ker(P)$, one has for $u\in E_\sigma$, 
\[
Tu =  PQu , \qquad 
-[P,Q]^2u = \sigma u
\qquad\text{and}\qquad
-P[P,Q]^2P(Tu) = - PQ[P,Q]^2u= \sigma PQ u = \sigma Tu
\]
using the commutation relations. 
This shows that $(\sigma,Tu)$ is an eigenpair of $-P[P,Q]^2P$.
Moreover, $T^2 = -\sigma \mathrm{I} $ on $E_\sigma$, so that $T: E_\sigma \to F_\sigma$ is 1-1.

In conclusion, $\ker(-[P,Q]^2-\sigma) = \operatorname{span}\{ u, Tu : 
u\in \ker(PQP-\lambda)\}$ if $\sigma=\lambda(1-\lambda)$ with $\lambda>0$ so that choosing an orthonormal basis $\{(\lambda_k,u_k)\}_{k\in\N}$ of eigenfunction of $PQP$, by the spectral theorem ($PQP\ge 0$), for any $f:\R_+ \to\R_+$, continuous with $f(0)=0$, one has
\[
\tr f(-[P,Q]^2) = 2 \sum_{\lambda_k>0} f(\lambda_k(1-\lambda_k)) =2 \tr g(PQP)
\]
where $g(\lambda)= f(\lambda(1-\lambda))$. 
Note that since $PQP \le 1$, one can consider only test functions $f$ defined on $[0,\tfrac14]$, both sides of \eqref{eq:spectral_link} are non-negative and possibly infinite.
\end{proof}

\begin{rem} \label{rem:Widom}
Considering Widom's conjecture, in light of Lemma~\ref{lem:spectral_link}, the spectral asymptotics of  Conjecture~\ref{conj:Widom} are equivalent to, given $f:[0,\frac14]\to \R_+$ continuous with $f(0)=0$, 
\[
\frac{\tr f(-[\Pi,\1_\Omega]^2)}{{(2\pi\hbar)^{1-n}}\log\hbar^{-1}} \to 8C_{\Omega}\int_{[0,\tfrac14]}\frac{f(\sigma)}{\sigma\sqrt{1-4\sigma}}\dd \sigma . 
\]
\end{rem}

We now turn to show the inequalities for the entropy.

\begin{proof}[Proof of Lemma~\ref{lem:ent}]
Considering \eqref{sEnt}, the map $s= f(\lambda(1-\lambda))$ where  $f:[0,\frac14]\to \R_+$ is given by 
\[
f:\sigma\mapsto
-\tfrac{1-\sqrt{1-4\sigma}}{2}\log(\tfrac{1-\sqrt{1-4\sigma}}{2})-\tfrac{1+\sqrt{1-4\sigma}}{2}\log(\tfrac{1+\sqrt{1-4\sigma}}{2}). 
\]
This function is continuous with $f(0)=0$ (one has $f(\sigma)\sim -\sigma\log(\sigma)$ as $\sigma \to 0$)
so that, according to Lemma~\ref{lem:spectral_link}, we can rewrite
\[
\mathcal{S}_\Omega(\X)
=  \tr f(-  [\Pi,\1_\Omega]^2) .
\]
The proof is based on the basic inequalities, for $\sigma \in[0,\frac14]$,
\begin{equation} \label{entub}
\cst \sigma \le f(\sigma) \le -2 \sigma\log \sigma .
\end{equation}
Observe that these function are all smooth, increasing, concave on $(0,\frac14)$  and take the values $0$, respectively $\log 2$ at the endpoints; so the inequality \eqref{entub} follows from the fact that 
$f'(\sigma)\sim \log(1/\sigma)$ as $\sigma \to 0$.

Thus, by monotonicity, with $g:\sigma \mapsto -2\sigma\log \sigma$,  we obtain
\[
\cst \| [\Pi,\1_\Omega] \|_{\J^2}^2  \le \mathcal{S}_\Omega \le   \tr g(-  [\Pi,\1_\Omega]^2) .
\]

We now relate the upper-bound to the trace-norm of the commutator $[\Pi,\1_\Omega]$. We can assume that $\| [\Pi,\1_\Omega]\|_{\J^1} <\infty$, otherwise there is nothing to prove, in which case $  \| [\Pi,\1_\Omega]\|_{\J^1}^2 = \sum_{n\in\N} \sigma_n <\infty$
where $\{\sigma_n\}$ denotes the non-zero eigenvalues of $-[\Pi,\1_\Omega]^2$. Then, by convexity of $\sigma\mapsto \log(\sigma^{-1/2})$ on $\R_+$ and Jensen's inequality, 
\[
{\textstyle \sum_{n\in\N}}  \sigma_n\log(\sigma_n^{-1/2}) \le {\textstyle \big(\sum_{n\in\N}}  \sigma_n\big) \log\bigg(\frac{\sum_{n\in\N}   \sigma_n^{1/2}}{ \sum_{n\in\N}\sigma_n}\bigg)
\]
or equivalently, 
\[
\tfrac14  \tr g(-  [\Pi,\1_\Omega]^2) \le \big\| [\Pi,\1_\Omega] \big\|_{\J^2}^2 \log\bigg(\frac{\| [\Pi,\1_\Omega]\|_{\J^1}}{\| [\Pi,\1_\Omega] \|_{\J^2}^2 } \bigg) .
\]
This completes the proof.
\end{proof}

\appendix

\section{Stationary phase with mild amplitudes}
\label{sec:stat-phase}

We first recall a standard version of the stationary phase method where the amplitude is controlled in $\Co^k$, independently of semiclassical parameter  $\hbar$, and the integral depends on a parameter $z\in\mho$.
We refer for instance to theorem 7.7.5 in \cite{hormander_analysis_2003} or to
our previous work~\cite[Proposition A.15]{deleporte_universality_2024}.

\begin{prop}[Stationary phase lemma] \label{lem:statphase}
For $d,q\in\N$, let $\Omega \Subset \R^d$ and $\mho\Subset\R^q$ be open. 
Let $\Phi : \Omega\times\mho \to \R$ be a smooth function\footnote{In fact, it is not necessary to assume that the phase $\Phi$ and the symbol $a$ are smooth with respect to the parameter $z\in\mho$; continuity suffices.} such that $\partial_x \Phi(x,z)=0$ has a unique solution $(x_z,z) \in \Omega\times\mho$ with $\Phi(x_z,z)=0$  and the Hessian $\nabla^2_x \Phi(x_z,z)$ is non-degenerate for $z\in\mho$. 
Let $a \in S^1 (\Omega\times \mho)$ be classical symbol.
Then, there exists another classical symbol $b \in S^1(\mho)$ so that 
\[
\int_{\Omega}e^{i\tfrac{\Phi(x,z)}\hbar}a(x,z)\dd x =   (2\pi \hbar)^{\frac d2} b(z,\hbar) 
\]
and the principal part of $b$ is given by $b_0(z) =a_0(x_z,z)/\sqrt{\det \nabla^2_x \Phi(x_z ,z) }$ for $z\in A$. \\
On the other hand, if $x\mapsto\Phi(x,z)$ has no critical points in $\Omega$, then
\[
\int_{\Omega}e^{i\tfrac{\Phi(x,z)}\hbar}a(x,z)\dd x =  \O(\hbar^\infty) 
\]
where the error term is controlled in $S^1(\mho)$.\end{prop}

The goal of this section is to generalize the above expansion when the symbol $a \in S^{\delta}$ is sufficiently regular ($\delta \geq \hbar^{\frac 12}$)
with an explicit control of the error. 
We begin by the simple case where the phase $\Phi$ has no critical point within the support of the symbol $a$. 

\begin{lem}\label{prop:nonstatphase}
For $d,q\in\N$, let $\Omega \Subset \R^d$ and $\mho\Subset\R^q$ be open. 
Let $\Phi : \Omega\times\mho \to \R$ be a smooth function such that $\partial_x \Phi(x,z)=0$ has a unique solution $(x_z,z) \in \Omega\times\mho$ and the Hessian $\nabla^2_x \Phi(x_z,z)$ is non-degenerate for all $z\in\mho$. 
Let $\delta \geq \hbar^{\frac 12}$ and $a \in \Co_c(\Omega\times\mho) $ with $a\in  S^{\delta}_x$ and suppose that for $z\in\mho$, $a(x,z)=0$ for 
$x\in B(x_z,\epsilon)$  for some $\epsilon \ge \delta$. Then, for every $k\in \N$, 
\[
\sup_{z\in\mho}\bigg| \hbar^{-\frac{d}{2}}\int e^{i\frac{\Phi(x,z)}{\hbar}}a(x,z)\dd x\bigg| \le C_{k} \left(\frac{\hbar}{\delta\epsilon}\right)^k  . 
\]
In particular, if $a \in S^\delta(\Omega\times\mho)$ with $\delta \geq \hbar^{\frac 12}$ and   $\epsilon\delta\gg \hbar$, then  as $\hbar\to0$
\begin{equation} \label{loc}
\int e^{i\tfrac{\Phi(x,z)}{\hbar}}a(x,z)\dd x =\O_{\Co^\infty}(\hbar^\infty) . 
\end{equation}
\end{lem}

\begin{proof}
To ease notation, we  treat the case without parameter: $\mho=\emptyset$, $\Phi : \Omega \to \R$ is a smooth function with a unique non-degenerate critical point $x_\emptyset\in\Omega$ and $a \in S^{\delta}(\Omega) $ with $a=0$ on
$B(x_\emptyset,\epsilon)$ with $\epsilon \ge \delta \ge \hbar^{1/2}$.
By assumptions, $x\mapsto \frac{a(x)}{|\nabla\Phi(x)|^2}$ is a smooth function, so that integrating by parts, we have
\[
\int e^{i\frac{\Phi(x)}{\hbar}}a(x)\dd x =i\hbar\int
e^{i\frac{\Phi(x)}{\hbar}}\frac{a(x)\Psi(x)+\nabla a(x)\cdot \nabla \Phi(x)}{|\nabla \Phi(x)|^2}\dd x
\]
where $\Psi : x\mapsto \Delta \Phi(x) + |\nabla \Phi(x)|^2  \nabla \Phi(x)  \cdot \nabla |\nabla \Phi(x)|^{-2}$ is also smooth on $\Omega$. 
We can perform this operation $k$ times, and we obtain differential operators $L^{j;k}$ of degree $\le j$ (taking values into symmetric $j$-tensors) whose coefficients are smooth functions (depending on $\Phi$) so that
\begin{equation} \label{ipp}
\int e^{i\frac{\Phi(x)}{\hbar}}a_{\hbar}(x)\dd x=(i\hbar)^{k}\int
e^{i\frac{\Phi(x)}{\hbar}}\sum_{0\le j\le k}\frac{
\nabla\Phi(x)^{\otimes j}\cdot L^{j;k}a (x)}{|\nabla \Phi(x)|^{2k}} \dd x .
\end{equation}
Since $x_\emptyset$ is the unique non-degenerate critical point of
$\Phi$, $|\nabla\Phi(x)| \asymp |x-x_\emptyset|$ for $x\in\Omega$, so
we have the following estimates, for every $0\le j\le k$:
\[
\frac{|\nabla\Phi(x)^{\otimes j}\cdot L^{j;k}a(x)|}{|\nabla
\Phi(x)|^{2k}} \leq C_k\delta^{-j} |x-x_\emptyset|^{-2k+j} ,\qquad x\in\Omega. 
\]
Then, using that $a=0$ on
$B(x_\emptyset,\epsilon)$ with $\epsilon\ge \delta$, we obtain for every $k\ge d$ and for  every $0\le j\le k$,
\[
\int \frac{|\nabla\Phi(x)^{\otimes j}\cdot L^{j;k}a(x)|}{|\nabla
\Phi(x)|^{2k}} \dd x\leq C_k\delta^{-j} \epsilon^{-2k+j +d}
\le C_k\delta^{-k} \epsilon^{-k+d} .
\]
We conclude that for every $k\ge d$, 
\[
\bigg| \hbar^{-\frac{d}{2}}\int_{\Omega}e^{i\frac{\Phi(x)}{\hbar}}a(x)\dd x \bigg| \le C_k \hbar^{k-d/2} \delta^{-k} \epsilon^{-k+d}
= C_k \left(\frac{\hbar}{\delta\epsilon}\right)^{k-d} \left(\frac{\hbar^{1/2}}{\delta}\right)^{d}
\]
which proves the claim in case $\mho=\emptyset$ using the condition $\delta\geq \hbar^{1/2}$. 
The argument remains the same if $(\Phi,a)$ depend continuously on a parameter $z\in\mho$ and the estimates are uniform provided that $\mho$ is relatively compact. 
In particular, if $\epsilon\delta \gg \hbar$, we obtain  uniformly for $z\in\mho$, 
\[
\int e^{i\frac{\Phi(x,z)}{\hbar}}a(x,z)\dd x = \O(\hbar^\infty) .
\]
If we apply $\partial_z^\alpha$ to the LHS, we obtain a similar integral with another symbol $a_\alpha \in S^\delta$ times $\hbar^{-|\alpha|}$, so this proves \eqref{loc}. 
\end{proof}

Proposition~\ref{prop:nonstatphase} already implies that such oscillatory integrals are uniformly bounded. 

\begin{corr} \label{corr:bdd}
Let $\Phi$ satisfy the assumptions of Lemma~\ref{prop:nonstatphase}. 
Then, for any $a \in S^{\delta}(\Omega\times\mho) $ with $\delta \geq \hbar^{\frac 12}$, 
as $\hbar\to 0$,
\[
\hbar^{-\frac{d}{2}}\int_{\Omega}e^{i\tfrac{\Phi(x,z)}{\hbar}}a(x,z)\dd x=\O_{S^\delta}(1) .
\]
\end{corr}

\begin{proof}
By a change of variable, we can assume without loss of generality that $x_z=0$ for $z\in\mho$ and $\delta=\hbar^{\frac 12}$. 
Let  $\chi\in C^{\infty}_c(\R^d)$ be a cutoff with $\chi=1$ on the unit ball and we split 
\[
\hbar^{-\frac{d}{2}}\int_{\Omega}e^{i\tfrac{\Phi(x,z)}{\hbar}}a(x,z)\dd x=
\hbar^{-\frac{d}{2}}\int_{\Omega}e^{i\tfrac{\Phi(x,z)}{\hbar}}a(x,z)\chi(x/\delta)\dd x
+ \hbar^{-\frac{d}{2}}\int_{\Omega}e^{i\tfrac{\Phi(x,z)}{\hbar}}b(x,z)\dd x
\]
where $b: (x,z)\mapsto a(x,z)(1-\chi(\delta^{-1}x))$ is also in $S^\delta$. Hence, by  
Lemma~\ref{prop:nonstatphase}, the second integral is $\O(1)$, while the first integral is also $\O(1)$ by a direct volume estimate.
\end{proof}

Then, we can improve this a-priori estimate by giving the asymptotic expansion of such oscillatory integrals using the stationary phase method with precise estimates on the error term.

\begin{prop}\label{prop:stat_tempered}
For $d,q\in\N$, let $\Omega \Subset \R^d$ and $\mho\Subset\R^q$ be open. 
Let $\Phi : \Omega\times\mho \to \R$ be a smooth function such that $\partial_x \Phi(x,z)=0$ has a unique solution $(x_z,z) \in \Omega\times\mho$ with $\Phi(x_z,z)=0$  and the Hessian $\nabla^2_x \Phi(x_z,z)$ is non-degenerate for $z\in\mho$. 
Let $a \in S^{\delta}(\Omega\times\mho) $ with $\delta \geq \hbar^{\frac 12}$. Then for every $\ell \in \N_0$, as $\hbar\to 0$,
\begin{equation}\label{eq:phase_stat}
(2\pi i\hbar)^{-\frac{d}{2}}\int e^{i\tfrac{\Phi(x,z)}{\hbar}}a(x,z)\dd x
= \sum_{0\le j<\ell} \hbar^j  L^j_xa (x_z,z)+ \O_{S^\delta}\big((\hbar\delta^{-2})^{\ell}\big)
\end{equation}
where, for $j\in\N_0$, $L^j_x$  is a differential operator $($acting on $x)$ of degree $2j$ whose coefficients depend only
on the phase~$\Phi$. 
In addition, if $\supp(a)\subseteq\mathcal{A}$, then for every $\ell \in \N_0$ and $\epsilon\in[\delta,1]$, as $\hbar\to 0$,
\begin{equation}\label{eq:supp_stat}
(2\pi i\hbar)^{-\frac{d}{2}}\int e^{i\tfrac{\Phi(x,z)}{\hbar}}a(x,z)\dd x
= \O_{S^\delta}\bigg(\bigg(\frac{\hbar}{\delta\epsilon}\bigg)^{\ell}\bigg) \qquad\text{uniformly on }\big\{z\in\mho : \dist\big((x_z,z) , \mathcal{A}\big) \ge \epsilon\big\} .
\end{equation}
\end{prop}

\begin{proof}
To compute the integral \eqref{eq:phase_stat}, we apply Morse lemma (with parameters) to the phase~$\Phi$, see \cite[p502]{hormander_analysis_2003}. 
By Lemma~\ref{prop:nonstatphase} applied with a fixed $\epsilon$, we can assume that $a$ is supported on an arbitrary small neighborhood $\mathcal U \subset \R^{d\times q}$ of the critical point $(x_z,z)$, up to a negligible error (in the sense of \eqref{loc}). 
Then, by a $C^\infty$ diffeormorphism $\varphi$, there is coordinate system such that the critical points are $(0,z)$ and the phase 
\[
\Phi(x,z) = \frac{x\cdot H x}{2} ,\qquad 
H=\begin{pmatrix}I_{d-\alpha}&0\\0&-I_{\alpha}
\end{pmatrix},
\]
where $\alpha\in[0,d]$ is the Morse index, that is, the number of negative eigenvalues of the non-degenerate matrix matrix $\nabla^2_x \Phi(x_z,z)$ -- by continuity $\alpha$ is constant  (independent of $z\in\mho$ and so is the phase $\Phi$ in this new coordinate system).  
Moreover, the symbol $b= a(\varphi)\in S^\delta$ with the same $\delta$. Consequently, it suffices to show that the integral 
\[
R_0: z\mapsto (2\pi i\hbar)^{-\frac{d}{2}}\int e^{i\frac{xHx}{2\hbar}}b(x,z)\dd x
\]
has an expansion of the type \eqref{eq:phase_stat} as $\hbar\to0$.
First, by Corollary~\ref{corr:bdd}, since the phase is independent of~$z$, by differentiating under the integral, $R_0\in S^\delta(\mho)$ and it satisfies \eqref{eq:supp_stat}, by Lemma~\ref{prop:nonstatphase}. 

\smallskip

Let $L := i(\nabla_x \cdot H\nabla_x)$ and $b_1 := \delta^{2} Lb$.
Observe that if we integrate by parts twice, for any $\lambda>0$, 
\begin{equation}\label{eq:diff_eq_phase_stat}
-\partial_\lambda\left[\lambda^{\frac d2}\int e^{i\lambda
\tfrac{xHx}{2\hbar}} b(x,z)\dd x\right]= \hbar\lambda^{\frac d2-2}\int
e^{i\lambda \tfrac{xHx}{2\hbar}}L b(x,z)\dd x.
\end{equation}
Moreover, by the standard stationary phase (Proposition~\ref{lem:statphase} -- for a fixed $\hbar$ with $b$ independent of $\lambda$): 
\begin{equation} \label{statphaselim}
\begin{aligned}
\lim_{\lambda\to +\infty}\left(\frac\lambda{2\pi i \hbar}\right)^{\frac d2}\int e^{i\lambda \tfrac{xHx}{2\hbar}}b(x,z)\dd x
&= i^{-\alpha} b(0,z) \\
& = i^{-\alpha} |\det\nabla^2_x \Phi(x_z,z)|^{-1/2} a(x_z,z)
\end{aligned}
\end{equation}
going back to the original coordinate system. 
We have $b_1\in S^\delta$, so by Corollary~\ref{corr:bdd},
\[
\sup_{z\in\mho, \lambda \ge 1}\bigg|\hbar^{-\frac d2}\lambda^{\frac d2}\int
e^{i\lambda \tfrac{xHx}{2\hbar}}b_1(x,z)\dd x \bigg| =O(1)
\]
and 
\[
R_1(z) :=  (2\pi i\hbar)^{-\frac{d}{2}}  \int_1^\infty \lambda^{\frac d2-2}\bigg\{\int e^{i\lambda \tfrac{xHx}{2\hbar}}b_1(x,z)\dd x \bigg\}\dd\lambda
\]
is also in $S^\delta$ (just like the symbol $b_1$ by differentiating under the integral). 
Thus, integrating both sides of \eqref{eq:diff_eq_phase_stat} for $\lambda \in [1,\infty)$, by \eqref{statphaselim}, we obtain
\begin{equation*}
R_0(z)
= i^{-\alpha} b(0,z) + \hbar \delta^{-2} R_1(z) .
\end{equation*}
This proves the claim for $\ell=1$ with $L_0 = i^{-\alpha}$ (in this coordinate system). 
The general expansion follows by induction: for $\ell\in\N$,  let $b_\ell := \delta^{2\ell} L^\ell b$ and 
\begin{equation} \label{Rerror}
R_\ell(z) :=   (2\pi i\hbar)^{-\frac{d}{2}}\int_1^\infty \lambda^{\frac d2-2} \mathrm{g}_\ell(\lambda)\bigg\{\int e^{i\lambda \tfrac{xHx}{2\hbar}}b_\ell(x,z)\dd x \bigg\}\dd\lambda 
\end{equation}
with $ \mathrm{g}_1=1$ and for $\ell\in\N$,  $\mathrm{g}_{\ell+1} :[1,\infty)\to [0,1]$ is the solution of 
\[
\mathrm{g}_{\ell+1}'(\lambda)= \lambda^{-2}\mathrm{g}_\ell(\lambda) \, ,\qquad
\mathrm{g}_{\ell+1}(1)=0.
\]
In particular, $\mathrm{g}_{\ell}$ is non-decreasing and $\mathrm{g}_\ell(\infty)=\lim_{\lambda\to\infty} \mathrm{g}_\ell(\lambda)$ exists for every $\ell\ge 2$.
By an integration by parts, using \eqref{eq:diff_eq_phase_stat}--\eqref{statphaselim}, 
\[\begin{aligned}
R_\ell(z) &= \mathrm{g}_{\ell+1}(\infty)i^{-\alpha} b_\ell(0,z) -(2\pi i\hbar)^{-\frac{d}{2}}   \int_1^\infty \mathrm{g}_{\ell+1}(\lambda) \partial_\lambda\bigg\{ \lambda^{\frac d2}\int e^{i\lambda \tfrac{xHx}{2\hbar}}b_\ell(x,z)\dd x \bigg\}\dd\lambda \\
& = \mathrm{g}_{\ell+1}(\infty)i^{-\alpha} b_\ell(0,z)+\hbar \delta^{-2}R_{\ell+1}(z) . 
\end{aligned}\]
By induction, $\displaystyle \mathrm{g}_{\ell+1}(\infty) = \int_1^\infty \frac{\mathrm{g}_\ell'(\sigma)}{\sigma}\dd\sigma
= \int_1^\infty \frac{\mathrm{g}_{\ell-1}(\sigma)}{\sigma^3}\dd\sigma
=\cdots = \frac{1}{(\ell-1)!} \int_1^\infty \frac{\mathrm{g}_{1}(\sigma)}{\sigma^{\ell+1}} \dd \sigma =  \frac{1}{\ell!}$, so we conclude that  for $\ell\in \N$, 
\[
R_0(z)
= i^{-\alpha}  {\textstyle\sum_{k=0}^{\ell}}  (\hbar \delta^{-2})^k b_k(0,z)/k! +  (\hbar \delta^{-2})^{\ell+1} R_{\ell+1}(z) .
\]
Since $b_k = \delta^{2k} L^k b$ for $k\in\N_0$, this proves \eqref{eq:phase_stat} with $L^k= i^{-\alpha}\frac{(i \nabla \cdot H\nabla)^k}{k!}$ in this coordinate system. Going back to the original coordinates, $L^k$  are differential operators of degree $2k$ whose coefficients depend only
on the phase~$\Phi$ with $L^0 = i^{-\alpha} |\det\nabla_x^2\Phi|^{-1/2} $ as a multiplication operator.
\end{proof}

\begin{rem}\label{rk:decay}
We give a slight generalization of Proposition~\ref{prop:stat_tempered}. 
Let $\mathrm{g} : \Omega\times\mho \to \R$, $\mathcal{C}^1$, such that $\nabla \mathrm{g} \neq 0 $ on $\{|\mathrm{g}| <1\}$.
Let $a \in S^{\delta}(\Omega\times\mho) $ with $\delta \geq \hbar^{\frac 12}$,  assume that $\supp(a) \subset  \{|\mathrm{g}| < 1\}$ and it holds for $\ell\in\N$, $\epsilon\in[\delta,1]$, 
\begin{equation}\label{Lyap}
a = \O_{S^\delta}\big(\big(\tfrac{\hbar}{\delta\epsilon}\big)^{\ell} \big) 
\qquad\text{uniformly on }\{|\mathrm g(x,z)|> \epsilon\} .
\end{equation}
Then, setting $\mathrm{f}(z) := \mathrm{C} \mathrm g(x_z,z) $ for $z\in\mho$ and some constant $\mathrm{C}  \ge 1$, it holds for $\ell\in\N$,$\epsilon\in[\delta,1]$, 
\[
(2\pi i\hbar)^{-\frac{d}{2}}\int e^{i\tfrac{\Phi(x,z)}{\hbar}}a(x,z)\dd x
= \O_{S^\delta}\bigg(\bigg(\frac{\hbar}{\delta\epsilon}\bigg)^{\ell}\bigg) \qquad\text{uniformly on }\{|\mathrm f(z)|> \epsilon\} .
\]
This follows directly by splitting the integral 
\[
(2\pi i\hbar)^{-\frac{d}{2}} \bigg\{ \int e^{i\tfrac{\Phi(x,z)}{\hbar}}a(x,z)\chi_\epsilon(\mathrm g(x,z))\dd x + \int  e^{i\tfrac{\Phi(x,z)}{\hbar}}a(x,z)\big\{1-\chi_\epsilon(\mathrm g(x,z))\big\}\dd x \bigg\}
\]
where $\chi :\R\to[0,1]$ is a smooth cutoff supported in $B_1$ and equals to 1 on a neighborhood of $0$.
In particular, since $\epsilon\ge\delta$, both symbols are in $S^\delta$ and the first integral satisfies \eqref{eq:supp_stat} with $\mathcal{A} = \{|\mathrm g(x,z)|\le\epsilon\}$. Moreover, since $\mathrm{g}$ is not degenerate, 
\[
\big\{z\in\mho :  |\mathrm f(z)|> \epsilon\big\} \subset
\big\{z\in\mho : \dist\big((x_z,z) , \mathcal{A}\big) \ge \epsilon\big\} 
\]
if $\mathrm{C}$ is sufficiently large. 
By Lemma~\ref{prop:nonstatphase}, the second integral is also $\O\big(\big(\tfrac{\hbar}{\delta\epsilon}\big)^\infty\big)$.
\end{rem}

In the case of the ``standard'' phase $\Phi(x,\xi) = x\cdot \xi$, which is relevant for pseudo-differential calculus, we can also perform a stationary phase even in the case where the scalings in $x,\xi$ are different and the amplitude does not necessarily belong to $S^{\delta}(\R^{2n})$ with $\delta\ge \hbar^{\frac 12}$. 
Adapting the proof of Proposition~\ref{prop:stat_tempered}, we obtain the following statement.

\begin{prop} \label{prop:stat_tempered_classic}
Let $a : \R^{2n+m} \to \R$ with $a\in S^{\varepsilon_1}_x \times S^{\varepsilon_2}_\xi \times S^\eta_z$ with $\varepsilon_1,\varepsilon_2,\eta \in [\hbar,1]$ and assume that $ \varepsilon_1\varepsilon_2 \ge \hbar$. Then for every $\ell\in\N_0$, as $\hbar\to0$,
\begin{equation} \label{eq:phase_stat_classic}
\frac1{(2\pi i\hbar)^n}\int e^{i\tfrac{x\cdot\xi}{\hbar}}a(x,\xi,z)\dd x\dd \xi 
= \sum_{0\le k<\ell} \frac1{k!} (i\hbar\partial_x\cdot\partial_\xi)^k a(x,\xi,z)\big|_{x=\xi=0}+ \O_{S^\eta}\big( \big(\tfrac\hbar{\varepsilon_1\varepsilon_2}\big)^{\ell}\big) .
\end{equation}
\end{prop}

\begin{proof}
First, we record from the proof of Proposition~\ref{prop:stat_tempered} that if the phase is (independent of $z$) and given by
$ \Phi: X \mapsto \tfrac{X\cdot H X}{2}$ where $H$ is a constant non-degenerate matrix ($\det H\neq 0$), then the operator
\begin{equation*}
\label{Morse_coord}
L^k :=  i^{-\alpha} (\det H)^{-1/2} \frac{(i \nabla \cdot H\nabla)^k}{k!} , \qquad\text{for $k\in\N_0$} . 
\end{equation*}
Moreover, in  the special case $H = (\begin{smallmatrix} 0 & I_n \\I_n&0\end{smallmatrix})$ and $X=(x,\xi)$ for $n\in\N$  (so that
$\Phi(X) = x\cdot \xi$, if $ \varepsilon_1,\varepsilon_2 \ge \sqrt\hbar$, then the expansion \eqref{eq:phase_stat_classic} follows  directly from Proposition~\ref{prop:stat_tempered}. Note that by differentiating under the integral, since $a \in S^\eta_z$, the error term is controlled in the class $S^\eta$.  

We deal with the general case by a change of variable. 
Without loss of generality, we assume that $\varepsilon_1\ge\varepsilon_2$ and that the symbol $a$ is independent of the parameter $z$. Let $(\gamma,\delta) = (\sqrt{\varepsilon_2/\varepsilon_1},\sqrt{\varepsilon_1\varepsilon_2})$, let $\chi\in \Co^\infty_c(\R^n,[0,1])$ be such that $\chi=1$ on $B_n$ and, let $\theta^k(\xi) =  \frac{1-\chi(\xi)}{|\xi|^{2k}}\xi^{\otimes k}$ for $\ell\in\N_0$.
Note that $\gamma,\delta \in [\sqrt{\hbar},1]$ and let $\chi_\gamma=\chi(\cdot/\gamma)$ and $\theta^k_\gamma =\theta^k(\cdot/\gamma)$. We split
\[
\int e^{i\tfrac{x\cdot\xi}{\hbar}}a(x,\xi)\dd x\dd \xi =
\int e^{i\tfrac{x\cdot\xi}{\hbar}}a(x,\xi)\chi_\gamma(\xi)\dd x\dd \xi 
+ \int e^{i\tfrac{x\cdot\xi}{\hbar}}a(x,\xi)\theta^0_\gamma(\xi)\dd x\dd \xi .
\]
We can perform repeated integration by parts in the second integral (writing $\tfrac{-i\hbar\xi\cdot \partial_x }{|\xi^2|}e^{i\tfrac{x\cdot\xi}{\hbar}} =e^{i\tfrac{x\cdot\xi}{\hbar}}$), we obtain for any $k\in \N$, 
\[
\int e^{i\tfrac{x\cdot\xi}{\hbar}}a(x,\xi)\dd x\dd \xi =
\int e^{i\tfrac{x\cdot\xi}{\hbar}}a(x,\xi)\chi_\gamma(\xi)\dd x\dd \xi 
+ \frac{(i\hbar)^k}{(\varepsilon_1\gamma)^k} \int e^{i\tfrac{x\cdot\xi}{\hbar}}a_k(x,\xi) \cdot\theta^k_\gamma(\xi)\dd x\dd \xi .
\]
where $a_k= \varepsilon_1^k \partial_x^k a$ is in $S^{\varepsilon_1}$ and $\varepsilon_1\gamma=\delta$. Since $a_k$ is uniformly bounded and compactly supported and $ \big\| \theta^k_\gamma(\xi) \big\|_{L^1} =\O_k(1)$ if $k>n$, this shows that 
\[
\int e^{i\tfrac{x\cdot\xi}{\hbar}}a(x,\xi)\dd x\dd \xi =
\int e^{i\tfrac{x\cdot\xi}{\hbar}}a(x,\xi)\chi_\gamma(\xi)\dd x\dd \xi 
+ \O(\hbar^\infty) .
\]
Then, by rescaling (the phase and measure are invariant under this change of variables), letting \[\widetilde{a}(x,\xi) := a(x/\gamma,\xi\gamma)\chi(\xi),\] we have $\widetilde{a}\in S^\delta(\R^{2n})$ (in particular this symbol has a fixed compact support and $\delta \ge \sqrt{\hbar}$) and 
\begin{align*}
\frac1{(2\pi i\hbar)^n}\int e^{i\tfrac{x\cdot\xi}{\hbar}}a(x,\xi)\dd x\dd \xi &=
\frac1{(2\pi i\hbar)^n}\int e^{i\tfrac{x\cdot\xi}{\hbar}}\widetilde{a}(x,\xi)\dd x\dd \xi 
+ \O(\hbar^\infty) \\
&=  \sum_{0\le k<\ell} \hbar^k  L^k \widetilde{a}(x,\xi)\big|_{x=\xi=0}+ \O\big((\hbar\delta^{-2})^{\ell}\big)
\end{align*}
by applying Proposition~\ref{prop:stat_tempered}. Since 
$ L^k \widetilde{a}= L^k a  = \frac1{k!} (i\hbar\partial_x\cdot\partial_\xi)^k a$ in this case, this completes the proof.
\end{proof}

More specifically, we also need the following consequence of Proposition~\ref{prop:stat_tempered_classic}. 

\begin{corr} \label{corr:phase_stat_dyadic}
Let $n,m\in\N$. 
Let $a:\R^{2n+m} \to\R$, bounded with compact support, and assume that
$\displaystyle (x,\xi) \mapsto \int a (x,\xi,z) \dd z$ is in $S^{\varepsilon_1}_x \times S^{\varepsilon_2}_\xi$ for some $\varepsilon_1,\varepsilon_2 \in [\hbar,1]$ with  $ \varepsilon_1\varepsilon_2 \ge \hbar$. 
Then for every $\ell\in\N_0$, as $\hbar\to0$, 
\[
\frac1{(2\pi i\hbar)^n}\int e^{i\tfrac{x\cdot\xi}{\hbar}}a(x,\xi,z)\dd x\dd \xi \dd z
= \sum_{0\le k<\ell} \frac1{k!} \int (i\hbar\partial_x\cdot\partial_\xi)^k a(x,\xi,z)\big|_{x=\xi=0} \dd z+ \O\Big(\big(\tfrac{\hbar}{\varepsilon_1\varepsilon_2}\big)^{\ell}\Big) .
\]
\end{corr}

\begin{proof}
By assumption, we can apply Proposition~\ref{prop:stat_tempered_classic} to the symbol
\(\displaystyle
(x,\xi) \mapsto   \int a (x,\xi,z) \dd z .
\)
This yields the required expansion. Note that since $a$ is smooth with a fixed compact support, one can differentiate under the integral with respect to $z$.  
\end{proof}

\begin{corr} \label{corr:phase_stat_perturbation}
Let $m,\ell\in\N$ and $\varepsilon_1, \varepsilon_2 , \delta\in [\hbar,1]$ with  $ \varepsilon_1 \le \delta$ and  $\varepsilon_1 \varepsilon_2 \ge \hbar$. 
Let $ (t,\sigma,z) \mapsto a (t,\sigma,z) $ be in $S^{\varepsilon_1}_t(\R) \times S^{\varepsilon_2}_\sigma(\R)$ and  supported in $\{z\in\mho\}$ with $\mho \Subset \R^m$ open $(\mho$ is allowed to depend on the scales $\varepsilon_1, \varepsilon_2, \delta)$. 
Assume that $ a (t,\sigma,z) = \O_{S^{\varepsilon_1}_t \times S^{\varepsilon_2}_\sigma}\big(\big(\tfrac{\hbar}{\varepsilon\delta}\big)^\ell\big)$ uniformly for $t\in[-\delta,\delta] , \sigma\in\R,z\in\mho$, then  as $\hbar\to0$, 
\[
\frac1{2\pi i\hbar}\int e^{i\tfrac{t\sigma}{\hbar}}a(t,\sigma,z)\dd t\dd\sigma \dd z = \O\Big(|\mho|\big(\tfrac{\hbar}{\varepsilon\delta}\big)^\ell\Big) , \qquad \varepsilon = \min(\varepsilon_1,\varepsilon_2). 
\]
Let $\mathrm{g} : \R^n \to \R$, $\mathcal{C}^1$, such that $\{|\mathrm{g}| \le 1\}$, $\partial_x \mathrm{g} \neq 0$ on $\{\mathrm{g}=0\}$ is compact, and  $a \in S^{\varepsilon_1}_t(\R) \times S^{\varepsilon_2}_\sigma(\R)$  is supported in $\{|\mathrm{g}| \le 1\}$. 
In particular, if it holds for $\ell\ge 2$ and  $\delta \in[\varepsilon_1 ,1]$ 
\[
a (t,\sigma,z) = \O_{S^{\varepsilon_1}_t \times S^{\varepsilon_2}_\sigma}\big(\big(\tfrac{\hbar}{\varepsilon\delta}\big)^\ell\big) \text{ uniformly for $t\in[-\delta,\delta] , \sigma\in\R, z\in \{|\mathrm{g}|\ge \delta\}$}
\]
then 
\[
\frac1{2\pi i\hbar}\int e^{i\tfrac{t\sigma}{\hbar}}a(t,\sigma,z)\dd t\dd\sigma \dd z =  \O\big(\big(\tfrac{\hbar}{\varepsilon\varepsilon_1}\big)^\ell\big). 
\]
\end{corr}

\begin{proof} We simply introduce a smooth cutoff $\chi:\R\to[0,1]$ such that $\chi=0$ on $\R\setminus[-1,1]$ and $\chi^\dagger=1-\chi$ satisfies $\chi^\dagger=0$ on $[-\frac12,\frac12]$. Then, by repeated integrations by part, it holds for $z\in\mho$, 
\[
\int e^{i\tfrac{t\sigma}{\hbar}}a(t,\sigma,z)\dd t\dd\sigma 
= \int e^{i\tfrac{t\sigma}{\hbar}} \chi(t/\delta)a(t,\sigma,z)\dd t\dd\sigma 
+\big(\tfrac{i\hbar}{\eta\delta}\big)^k\int e^{i\tfrac{t\sigma}{\hbar}}   \tfrac{\chi^\dagger(t/\delta)}{(t/\delta)^k}\eta^k\partial_\sigma^ka(t,\sigma,z)\dd t\dd\sigma 
\]
By assumption, both symbols $(t,\sigma) \mapsto \chi(t/\delta)a(t,\sigma,z)$ and 
$(t,\sigma) \mapsto
\tfrac{\chi^\dagger(t/\delta)}{(t/\delta)^k}\eta^k\partial_\sigma^ka(t,\sigma,z)$
are
$\O_{S^{\varepsilon_1 }_t \times S^{\varepsilon_2}_\sigma}\big(\big(\tfrac{\hbar}{\varepsilon\delta}\big)^k\big)$ uniformly  for $z\in\mho$, hence  Proposition~\ref{prop:stat_tempered_classic} shows that the first integral is $\O\big(\hbar\big(\tfrac{\hbar}{\varepsilon\delta}\big)^k\big)$ and the second is $\O(\hbar)$ uniformly  for $z\in\mho$.
Then, by integrating over $z\in\mho$, this completes the proof. 
\end{proof}

\bibliographystyle{abbrv}
\bibliography{main}
\end{document}